\documentclass[10pt, a4paper]{article}
\usepackage{geometry}
\usepackage{rotating}
\usepackage{mathrsfs}
\usepackage{latexsym}
\usepackage{xy}
\usepackage{amsfonts,amsmath,amssymb,amsthm,tikz-cd}
\usepackage{color}
\usepackage[pagebackref,colorlinks]{hyperref}

\xyoption{all}

\newcommand{\bcen}{\begin{center}}     \newcommand{\ecen}{\end{center}}
\newcommand{\bay}{\begin{array}}      \newcommand{\eay}{\end{array}}
\newcommand{\beq}{\begin{eqnarray*}}      \newcommand{\eeq}{\end{eqnarray*}}

\def\A{\mathcal{A}}
\def\B{\mathcal{B}}

\def\Cone{\mathrm{Cone}}
\def\D{\mathcal{D}}

\def\dim{\mathrm{dim}}

\def\ev{\mathrm{ev}}
\def\Ext{\mathrm{Ext}}

\def\fd{\mathrm{fd}}

\def\gl{\mathrm{gl.dim}}

\def\Hom{\mathrm{Hom}}

\def\id{\mathrm{id}}
\def\Id{\mathrm{Id}}
\def\Im{\mathrm{Im}}
\def\ind{\mathrm{ind}}
\def\inj{\mathrm{inj}}
\def\Ker{\mathrm{Ker}}
\def\mod{\mathrm{mod}}
\def\Mod{\mathrm{Mod}}

\def\ol{\overline}
\def\op{\mathrm{op}}

\def\pd{\mathrm{pd}}
\def\per{\mathrm{per}}

\def\proj{\mathrm{proj}}

\def\rad{\mathrm{rad}}

\def\RHom{\mathrm{RHom}}
\def\sg{\mathrm{sg}}

\def\T{\mathcal{T}}
\def\thick{\mathrm{thick}}

\def\tri{\mathrm{tri}}

\def\Tor{\mathrm{Tor}}

\begin{document}

\newtheorem{definition}{Definition}
\newtheorem{proposition}{Proposition}
\newtheorem{lemma}{Lemma}
\newtheorem{theorem}{Theorem}
\newtheorem{remark}{Remark}
\newtheorem{corollary}{Corollary}
\newtheorem{example}{Example}
\newtheorem*{conjecture}{Conjecture}
\newtheorem{question}{Question}

\title{\bf\Large Singular equivalences of $n$-adjoint type and standard eventually homological isomorphisms \footnote{This work was supported by the National Natural Science Foundation of China (Grant No. 12371043 and 12561008) and the project of Young and Middle-aged Academic and
Technological leader of Yunnan (202305AC160005).}}

\author{Yang Han$^{1,2}$, Yongyun Qin$^3$ and Xianqing Wang$^{1,2}$}

\date{\footnotesize 1. Academy of Mathematics and Systems Science, Chinese Academy of Sciences, Beijing 100190, China. \\ 2. School of Mathematical Sciences, University of Chinese Academy of Sciences, Beijing 100049, China.\\ 3. School of Mathematics, Yunnan Key Laboratory of Modern Analytical Mathematics and Applications,\\ Yunnan Normal University, Kunming 650500, Yunnan, China. \\  E-mail: hany@iss.ac.cn (YH), qinyongyun2006@126.com (YQ)}

\maketitle

\begin{abstract}
We refine the definition of eventually homological isomorphisms so that it becomes more natural and fruitful. We introduce singular equivalences of 
$n$-adjoint type which can be viewed as standard singular equivalences of different levels and are more convenient for reducing various homological conjectures and transferring homological properties. We develop the theories of singular equivalences of $n$-adjoint type for $n\ge 2$ and standard eventually homological isomorphisms between derived categories by giving their characterizations in terms of $\otimes$-(co)perfect complexes of bimodules, two constructions via opposite algebras and tensor product algebras, examples induced by homomorphisms of algebras and idempotents, and applications to the reduction of homological conjectures such as the finitistic dimension conjecture, Gorenstein symmetry conjecture and Auslander-Reiten conjecture, as well as to the preservation of homological properties such as Gorensteinness, the Fg condition, Happel's property and Han's property. Moreover, we show that essentially surjective standard eventually homological isomorphisms between derived categories correspond to standard singular equivalences with canonical left adjoint.
\end{abstract}

\medskip


\medskip


\tableofcontents

\section{Introduction}

Throughout this paper, $k$ is a fixed field. Unless stated otherwise, all algebras are finite dimensional $k$-algebras, all categories are $k$-categories, and all functors are $k$-functors. There are four levels to study the classification of finite dimensional algebras: The first level is to classify them up to isomorphism and study their isomorphism invariance. The second level is to classify them up to Morita equivalence and study their Morita equivalence invariance. The third level is to classify them up to derived equivalence and study their derived equivalence invariance. The fourth level is to classify them up to singular equivalence and study their singular equivalence invariance. Here we focus on the last one.

The {\it singularity category} $\sg(A)$ of an algebra $A$ is defined as the Verdier quotient $\D^b(A)/\per(A)$, where $\D^b(A)$ is the bounded derived category of finitely generated (left) $A$-modules and $\per(A)$ is its thick subcategory of perfect complexes. The notion was introduced by Buchweitz in 1987 \cite{Buchweitz21} under the name {\it stable derived category}, and was rediscovered and renamed by Orlov in 2004 \cite{Orlov04}. The singularity category $\sg(A)$ of an algebra $A$ is the stabilization of the stable category $\underline{\mod}A$ of the category $\mod A$ of finitely generated $A$-modules. For a selfinjective algebra $A$, it is triangle equivalent to $\underline{\mod}A$ \cite{Rickard89}. For a Gorenstein algebra $A$, it is triangle equivalent to the stable category $\underline{\rm Gproj}A$ of the category ${\rm Gproj}A$ of finitely generated Gorenstein projective $A$-modules \cite{Buchweitz21}. The singularity category is an important triangulated invariant that measures homological singularities and captures stable homological information of algebras. For instance, in disproving Han's conjecture, (dg) singularity category plays a crucial role \cite{Kong-Liu-Shen26}.

Two algebras $A$ and $B$ are said to be {\it singularly equivalent} if their singularity categories $\sg(A)$ and $\sg(B)$ are triangle equivalent. Singular equivalence clearly preserves the finiteness of global dimensions. However, it is not easy to find more invariants under singular equivalence. It is well-known that if two algebras are derived equivalent then there exists a standard derived equivalence between them which is given by a two-sided tilting complex \cite{Rickard91}. This standardization renders finding derived equivalence invariants relatively easier. However, for singular equivalences, we have no such standardization at present, which makes it difficult to study the characterization and invariants of singular equivalences. This motivated the introduction of more refined notions of singular equivalence.
In 2012, Chen and Sun introduced {\it singular equivalence of Morita type} \cite{Chen-Sun12,Zhou-Zimmermann13}, which is a generalization of stable equivalence of Morita type \cite{Broue94} but not of derived equivalence. In 2015, Wang introduced {\it singular equivalence of Morita type with level} \cite{Wang15}, which is a common strict generalization of standard derived equivalence and singular equivalence of Morita type. Singular equivalence of Morita type with level preserves the finiteness of finitistic dimension \cite{Wang15}, Hochschild homology of positive degree \cite{Wang15}, Gorensteinness \cite{Wang15}, the Fg condition for Gorenstein algebras \cite{Skartsaeterhagen16}, support variety theory \cite{ChenYP18}, and the Igusa-Todorov property \cite{Qin22}. It also can reduce several homological conjectures such as the finitistic dimension conjecture, Han's conjecture, Auslander-Reiten conjecture and Keller's conjecture \cite{Wang15, Chen-Hu-Qin-Wang23, Chen-Li-Wang25}.
Nonetheless, it is not easy to recognize singular equivalences of Morita type with level, even for standard derived equivalences (see \cite[Theorem 2.3]{Wang15}). Moreover, there exist singular equivalences that are not of Morita type with level (see Example~\ref{Ex-SE2AT-NotSEMTWL}).
 
A concept closely related to singular equivalence is eventually homological isomorphism. {\it Eventually homological isomorphism between abelian categories} was introduced by Psaroudakis, Skarts{\ae}terhagen and Solberg and applied to transfer Gorensteinness and the Fg condition in \cite{Psaroudakis-Skartsaeterhagen-Solberg14} and to reduce Auslander-Reiten conjecture and Gorenstein symmetry conjecture by Qin and Shen in \cite{Qin-Shen24}. It seems more natural and general to define eventually homological isomorphism between derived categories rather than abelian categories. {\it Eventually homological isomorphism between derived categories} was introduced by Qin and applied to transfer Gorensteinness and the Fg condition in \cite{Qin20}. The definition of an eventually homological isomorphism $F$ between abelian or derived categories has nothing to do with the functoriality of $F$. This causes no problem for studying the eventually homological isomorphisms induced by idempotents. However, it is not so natural and fruitful for studying general eventually homological isomorphisms.

In this paper, we first introduce {\it $\otimes$-(co)perfect complexes of bimodules} over finite dimensional algebras which are crucial for characterizing singular equivalences of $n$-adjoint type and standard eventually homological isomorphisms later on. Then we clarify the relationships between $\otimes$-perfect complexes of bimodules and perfect complexes of bimodules.
Secondly, we refine the definition of eventually homological isomorphism so that it takes functoriality into account and it is more natural and fruitful. Then we give characterizations of standard eventually homological isomorphisms in terms of $\otimes$-(co)perfect complexes of bimodules (Theorem~\ref{Thm-EvHomIso-TensorPerfect-BimodComplex} and Theorem~\ref{Thm-EvHomIso-TensorPerfect-BimodComplex-RightAdjoint}), two constructions of standard eventually homological isomorphisms using opposite algebras (Proposition~\ref{Prop-EHI-OppositeAlg}) and tensor product algebras (Proposition~\ref{Prop-EHI-TensorProdAlg}), examples of standard eventually homological isomorphisms induced by homomorphisms of algebras (Corollary~\ref{Cor-EHI-AlgHom}) and idempotents (Proposition~\ref{Prop-Exist-EvHomIso-Idemp}), and an application to reduce Gorenstein symmetry conjecture via essentially surjective eventually homological isomorphisms (Theorem~\ref{Thm-EHI-GorSymConj}).
Thirdly, we introduce {\it singular equivalences of $n$-adjoint type} which can be viewed as standard singular equivalences of different levels, are easier to characterize and recognize, and are more convenient for reducing various homological conjectures and transferring homological properties. Then we give a characterization of singular equivalences of $n$-adjoint type in terms of $\otimes$-(co)perfect complexes of bimodules (Theorem~\ref{Thm-Exist-SingEquAdjTye-FD}), two constructions of singular equivalence of $n$-adjoint type using opposite algebras (Proposition~\ref{Prop-SE-OppositeAlg}) and tensor product algebras (Proposition~\ref{Prop-SE-TensorProdAlg}), examples of singular equivalences of $n$-adjoint type induced by homomorphisms of algebras (Corollary~\ref{Cor-SingEqu-AlgHom}) and idempotents (Corollary~\ref{Cor-Exist-SingEqu-Idemp}), and applications to reduce the finitistic dimension conjecture (Theorem~\ref{Thm-SE2AT-FDC}) and transfer Han's property (Theorem~\ref{Thm-SE2AT-HH-Han}) by singular equivalences of 2-adjoint type, and to reduce Gorenstein symmetry conjecture and Auslander-Reiten conjecture (Theorem~\ref{Thm-SE3AT-G-GSC-ARC}), and transfer Gorensteinness, Happel's property and the Fg condition (Theorem~\ref{Thm-SE3AT-HCH-Happel-Fg}) by singular equivalences of 3-adjoint type.
Finally, we show that essentially surjective standard eventually homological isomorphisms induce standard singular equivalences with canonical left adjoint (Theorem~\ref{Thm-EssSurjEHI=SingEqWithLeftAdj}).

The paper is organized as follows: 
In Section 2, we introduce $\otimes$-(co)perfect complexes of bimodules and clarify the relation between $\otimes$-perfect complex of bimodules and perfect complex of bimodules. 
In Section 3, we refine an eventually homological isomorphism between derived categories and give its characterizations in terms of $\otimes$-(co)perfect complex of bimodules, two constructions using opposite algebras and tensor product algebras, examples induced by homomorphisms of algebras and idempotents, and an application to transfer Gorensteinness and reduce Gorenstein symmetry conjecture.
In Section 4, we introduce singular equivalence of $n$-adjoint type and provide its characterizations in terms of $\otimes$-(co)perfect complexes of bimodules,  two constructions using opposite algebras and tensor product algebras, examples induced by homomorphisms of algebras and idempotents, and applications to reduce various homological conjectures and transfer several homological properties.  
In Section 5, we show that essentially surjective standard eventually homological isomorphisms induce standard singular equivalences with canonical left adjoint.

\medspace

\noindent{\bf Notations.}

Let $k$ be a fixed field.
Write $\otimes=\otimes_k$ and $(-)^*=\Hom_k(-,k)$. 
Let $A$ be a $k$-algebra. 
Denote by $A^\op$ the opposite algebra of $A$, 
and by $A^e=A\otimes A^\op$ the enveloping algebra of $A$.
Denote by $\Mod A$ the category of left $A$-modules,
and by $\mod A$ the full subcategory of $\Mod A$ consisting of all finitely generated left $A$-modules.
Denote by $\D(A):=\D(\Mod A)$ the derived category of $\Mod A$,
$\D^b(A)=\D^b(\mod A)$ (resp. $\D^+(\mod A)$, $\D^-(\mod A)$) the full triangulated subcategory of $\D(A)$ consisting of all complexes with bounded (resp. bounded below, bounded above) and finitely generated cohomology, i.e., all complexes quasi-isomorphic to bounded (resp. bounded below, bounded above) complexes of finitely generated left $A$-modules,
$\per(A)=K^b(\proj A)$ the full triangulated subcategory of $\D(A)$ consisting of all perfect (or compact) objects, i.e., all complexes quasi-isomorphic to bounded complexes of finitely generated projective left $A$-modules,
and $K^b(\inj A)$ the full subcategory of $\D(A)$ consisting of all complexes quasi-isomorphic to bounded complexes of finitely generated injective left $A$-modules.
Moreover, we denote by $\sg(A)$ the Verdier quotient of $\D^b(A)$ by its thick subcategory $\per(A)$.

\section{$\otimes$-(co)perfect complexes of bimodules}

In this section, we introduce $\otimes$-(co)perfect complexes of bimodules which are crucial for characterizing singular equivalences of $n$-adjoint type and standard eventually homological isomorphisms, and compare $\otimes$-perfectness with perfectness for complexes of bimodules.

\subsection{Left or right perfect complexes of bimodules}

Let $\T$ be a triangulated category. A full triangulated subcategory of $\T$ is said to be {\it thick} if it is closed under direct summands.
For objects $T_1,\cdots,T_n\in\T$, denote by $\tri(T_1,\cdots,T_n)$ the smallest full triangulated subcategory of $\T$ containing $T_1,\cdots,T_n$, i.e., the smallest full subcategory of $\T$ containing $T_1,\cdots,T_n$ and closed under shifts and extensions. Denote by $\thick(T_1,\cdots,T_n)$ the smallest thick triangulated subcategory of $\T$ containing $T_1,\cdots,T_n$, i.e., the smallest full subcategory of $\T$ containing $T_1,\cdots,T_n$ and closed under shifts, extensions and direct summands.

Let $A$ and $B$ be two $k$-algebras.
A complex $X$ of $A$-$B$-bimodules is said to be {\it left perfect} if $_AX\in\per(A)$, {\it right perfect} if $X_B\in\per(B^\op)$, and {\it biperfect} if $_AX\in\per(A)$ and $X_B\in\per(B^\op)$.

The following result is a characterization of the left perfectness of a complex of bimodules, in particular, the perfectness of a complex of left modules.

\begin{lemma} \label{Lem-LeftModComplex-Perfect} 
	Let $A$ and $B$ be finite dimensional $k$-algebras, and $X$ a complex of $A$-$B$-bimodules. Then the following conditions are equivalent:
	
	{\rm (1)} $_AX\in\per(A)$.
	
	{\rm (2)} $Y\otimes^L_AX_B\in\D^b(B^\op)$ for all $Y\in\D^b(A^\op)$, i.e., the derived tensor product functor $-\otimes^L_AX_B:\D(A^\op)\to\D(B^\op)$ can be restricted to bounded derived categories.	
	
	{\rm (2')} $M\otimes^L_AX_B\in\D^b(B^\op)$ for all $M\in\mod A^\op$.
	
	{\rm (2'')} $\ol{A}\otimes^L_AX_B\in\D^b(B^\op)$ where $\ol{A}:=A/\rad A$ is the factor algebra of $A$ modulo its Jacobson radical $\rad A$.
	
	{\rm (3)} $\RHom_A(X,Y)\in\D^b(B)$ for all $Y\in\D^b(A)$, i.e., the derived Hom functor $\RHom_A(X,-):\D(A)\to\D(B)$ can be restricted to bounded derived categories.
	
	{\rm (3')} $\RHom_A(X,M)\in\D^b(B)$ for all $M\in\mod A$.
	
	{\rm (3'')} $\RHom_A(X,\ol{A})\in\D^b(B)$.	
	
	{\rm (4)} $_AX\otimes^L_BY\in\per(A)$ for all $Y\in\per(B)$, i.e., the derived tensor product functor $_AX\otimes^L_B-:\D(B)\to\D(A)$ can be restricted to perfect derived categories.
\end{lemma}

\begin{proof}	
	(1)$\Rightarrow$(2): Due to $_AX\in \per(A)=\thick(A)$, for all $Y\in\D^b(A^\op)$, we have $Y\otimes^L_AX \in \thick(Y\otimes^L_AA) = \thick(Y_k) \subseteq \D^b(k)$. Thus $Y\otimes^L_AX \in \D^b(B^\op)$.
	
	(2)$\Leftrightarrow$(2'): {\it Necessity}: It follows from $\mod A^\op\subseteq \D^b(A^\op)$. {\it Sufficiency}: For all $Y\in\D^b(A^\op)$, there are $M_1,\cdots,M_n\in\mod A^\op$ such that $Y\in\tri(M_1,\cdots,M_n)$. By assumption, $M_1\otimes^L_AX,\cdots,M_n\otimes^L_AX \in \D^b(B^\op)$. So $Y\otimes^L_AX\in \tri(M_1\otimes^L_AX,\cdots,M_n\otimes^L_AX) \subseteq \D^b(B^\op)$.
	
	(2')$\Leftrightarrow$(2''): {\it Necessity}: It follows from $\ol{A}\in\mod A^\op$. {\it Sufficiency}: For all $M\in\mod A^\op$, we have the radical filtration 
	$M\supset \rad M\supset \rad^2M\supset\cdots\supset \rad^{l-1}M\supset \rad^lM=0$ of $M$ where $l$ is the radical length of $M$. The short exact sequences $0\to \rad^{i+1}M \to \rad^iM \to \rad^iM/\rad^{i+1}M \to 0$ in $\mod A^\op$, $0\le i\le l-1$, induce triangles $\rad^{i+1}M \to \rad^iM \to \rad^iM/\rad^{i+1}M \to$ in $\D^b(A^\op)$, $0\le i\le l-1$, where $\rad^iM/\rad^{i+1}M, 0\le i\le l-1$, are semisimple right $A$-modules. Since $\rad^iM/\rad^{i+1}M\in\thick(\ol{A})$ for all $0\le i\le l-1$, we have $M\in\thick(\ol{A})$. 
	Thus $M\otimes^L_AX\in \thick(\ol{A}\otimes^L_AX) \subseteq \D^b(B^\op)$.
	
	(2'')$\Rightarrow$(1): We have known (2'')$\Leftrightarrow$(2'). Applying (2') to $M=A$, we obtain $X_B\in\D^b(B^\op)$, and further $_AX\in \D^b(A)$. 	
	By \cite[B.61 Theorem]{Christensen-Foxby-Holm24}, $_AX$ admits a minimal semi-projective resolution $P=(P^\bullet,d^\bullet_P)$ such that $P^i$'s are finitely generated projective and $\Im d^i_P\subseteq \rad P^{i+1}$ for all $i\in \mathbb{Z}$. Assume on the contrary that $P$ has infinitely many nonzero components. Then $\ol{A}\otimes^L_AX \cong \ol{A}\otimes_AP \cong (\cdots \xrightarrow{0} \ol{P^i} \xrightarrow{0} \ol{P^{i+1}} \xrightarrow{0} \cdots)$ with $\ol{P^i}:=P^i/\rad P^i\ne 0$ for infinitely many $i\in \mathbb{Z}$, which contradicts to the assumption $\ol{A}\otimes^L_AX\in \D^b(B^\op) \subseteq \D^b(k)$. Thus $P$ has only finitely many nonzero components. Hence $_AX$ is perfect. 
	
	(1)$\Rightarrow$(3): Due to $X\in \per(A)$, for all $Y\in\D^b(A)$, we have $\RHom_A(X,Y)   \in \thick(\RHom_A(A,Y)) = \thick(_kY) \subseteq\D^b(k)$. Thus $\RHom_A(X,Y) \in \D^b(B)$.
	
	(3)$\Leftrightarrow$(3'): {\it Necessity}: It follows from $\mod A\subseteq \D^b(A)$. {\it Sufficiency}: For all $Y\in\D^b(A)$, there are $M_1,\cdots,M_n\in\mod A$ such that $Y\in\tri(M_1,\cdots,M_n)$. Thus $\RHom_A(X,Y)\in \tri(\RHom_A(X,M_1), \cdots,\linebreak\RHom_A(X,M_n)) \subseteq \D^b(k)$. So $\RHom_A(X,Y) \in \D^b(B)$.
	
	(3')$\Leftrightarrow$(3''): {\it Necessity}: It follows from $\ol{A}\in\mod A$. {\it Sufficiency}: For all $M\in\mod A$, we have $M\in\thick(\ol{A})$. 
	Then $\RHom_A(X,M)\in \thick(\RHom_A(X,\ol{A})) \subseteq \D^b(k)$. Thus $\RHom_A(X,M) \in \D^b(B)$.
	
	(3'')$\Rightarrow$(1): We have known (3'')$\Leftrightarrow$(3'). Applying (3') to $M=\Hom_k(A,k)$, we obtain $\Hom_k(X,k)\in\D^b(B)$, then $X\in\D^b(B^\op)$, and further $X\in \D^b(A)$. By \cite[B.61 Theorem]{Christensen-Foxby-Holm24}, $_AX$ admits a minimal semi-projective resolution $P=(P^\bullet,d^\bullet_P)$.	
	Assume on the contrary that $P$ has infinitely many nonzero components then $\RHom_A(X,\ol{A}) \cong \RHom_A(P,\ol{A}) \cong (\cdots \xrightarrow{0} \Hom_A(P^{i+1},\ol{A}) \xrightarrow{0} \Hom_A(P^i,\ol{A}) \xrightarrow{0} \cdots)$ has infinitely many nonzero components, which contradicts to the assumption $\RHom_A(X,\ol{A})\in \D^b(B) \subseteq \D^b(k)$.	
	Thus $P$ has only finitely many nonzero components. Hence $_AX$ is perfect. 
	
	(1)$\Leftrightarrow$(4): {\it Necessity}: For all $Y\in\per(B)$, $X\otimes^L_BY\in \thick(X\otimes^L_BB)=\thick(_AX)\subseteq \thick(A)$.
	{\it Sufficiency}: Take $Y=B$.
\end{proof}

\begin{remark}{\rm
		(i) The conditions (2'') and (3'') imply that $\ol{A}=A/\rad A$ is a ``testing module'' for the left perfectness of complexes of bimodules.
		
		(ii) Taking $B=k$, we obtain a characterization of the perfectness of complexes of left modules.
}\end{remark}

Dually, we have the following characterization of the right perfectness of a complex of bimodules, in particular, the perfectness of a complex of right modules.

\begin{lemma} \label{Lem-RightModComplex-Perfect} 
	Let $A$ and $B$ be finite dimensional $k$-algebras, and $X$ a complex of $A$-$B$-bimodules. Then the following conditions are equivalent:
	
	{\rm (1)} $X_B\in\per(B^\op)$.
	
	{\rm (2)} $_AX\otimes^L_BY\in\D^b(A)$ for all $Y\in\D^b(B)$, i.e., the derived tensor product functor $_AX\otimes^L_B-:\D(B)\to\D(A)$ can be restricted to bounded derived categories.	
	
	{\rm (2')} $_AX\otimes^L_BM\in\D^b(A)$ for all $M\in\mod B$.
	
	{\rm (2'')} $_AX\otimes^L_B\ol{B}\in\D^b(A)$ where $\ol{B}:=B/\rad B$.
	
	{\rm (3)} $\RHom_{B^\op}(X,Y)\in\D^b(A^\op)$ for all $Y\in\D^b(B^\op)$, i.e., the derived Hom functor $\RHom_{B^\op}(X,-):\D(B^\op)\to\D(A^\op)$ can be restricted to bounded derived categories.
	
	{\rm (3')} $\RHom_{B^\op}(X,M)\in\D^b(A^\op)$ for all $M\in\mod B^\op$.
	
	{\rm (3'')} $\RHom_{B^\op}(X,\ol{B})\in\D^b(A^\op)$.	
	
	{\rm (4)} $Y\otimes^L_AX_B\in\per(B^\op)$ for all $Y_A\in\per(A^\op)$, i.e., the derived tensor product functor $-\otimes^L_AX_B:\D(A^\op)\to\D(B^\op)$ can be restricted to perfect derived categories.	
\end{lemma}

\begin{remark}{\rm
		(i) The conditions (2'') and (3'') imply that $\ol{B}=B/\rad B$ is a ``testing module'' for the right perfectness of complexes of bimodules.
		
		(ii) Taking $A=k$, we obtain a characterization of the perfectness of a complex of right modules.
}\end{remark}

\subsection{$\otimes$-perfect complexes of bimodules}

\begin{proposition} \label{Prop-BimodComplex-TensorPerfect} 
	Let $A$ and $B$ be finite dimensional $k$-algebras, and $X$ a complex of $A$-$B$-bimodules. Then the following conditions are equivalent:
	
	{\rm (1)} $Y\otimes^L_AX\otimes^L_BZ\in\D^b(k)$ for all $Y\in\D^b(A^\op)$ and $Z\in\D^b(B)$.
	
	{\rm (1')} $M\otimes^L_AX\otimes^L_BN\in\D^b(k)$ for all $M\in\mod A^\op$ and $N\in\mod B$.	
	
	{\rm (1'')} $\ol{A}\otimes^L_AX\otimes^L_B\ol{B}\in\D^b(k)$.		
	
	{\rm (2)} $_AX\otimes^L_BZ\in\per(A)$ for all $Z\in\D^b(B)$.	
	
	{\rm (2')} $_AX\otimes^L_BN\in\per(A)$ for all $N\in\mod B$.
	
	{\rm (2'')} $_AX\otimes^L_B\ol{B}\in\per(A)$.
	
	{\rm (3)} $Y\otimes^L_AX_B\in\per(B^\op)$ for all $Y\in\D^b(A^\op)$.
	
	{\rm (3')} $M\otimes^L_AX_B\in\per(B^\op)$ for all $M\in\mod A^\op$.
	
	{\rm (3'')} $\ol{A}\otimes^L_AX_B\in\per(B^\op)$.	
\end{proposition}

\begin{proof}
	(1)$\Leftrightarrow$(2): It follows from Lemma~\ref{Lem-LeftModComplex-Perfect}.
	
	(1)$\Leftrightarrow$(3): It follows from Lemma~\ref{Lem-RightModComplex-Perfect}.
	
	(1)$\Leftrightarrow$(1')$\Leftrightarrow$(1''), (2)$\Leftrightarrow$(2')$\Leftrightarrow$(2''), and (3)$\Leftrightarrow$(3')$\Leftrightarrow$(3''): Similar to the proof of Lemma~\ref{Lem-LeftModComplex-Perfect}.			
\end{proof}

\begin{definition}{\rm 
	Let $A$ and $B$ be finite dimensional $k$-algebras. A complex $X$ of $A$-$B$-bimodules is said to be {\it tensor-perfect} or {\it $\otimes$-perfect} if it satisfies the nine equivalent conditions in Proposition~\ref{Prop-BimodComplex-TensorPerfect}.
}\end{definition}

\begin{example} \label{Ex-TensorPerfect-Bimod} {\rm
		(1) By Proposition~\ref{Prop-BimodComplex-TensorPerfect} (1''), the $A$-$B$-bimodule $A\otimes B$ is $\otimes$-perfect. Then all objects in the thick triangulated subcategory $\thick(A\otimes B)$ of $\D(A\otimes B^\op)$ are $\otimes$-perfect by Proposition~\ref{Prop-TensorPerfect-Property}.
		
		(2) As a complex concentrated in degree 0, the $A$-bimodule $A$ is $\otimes$-perfect if and only if $M\otimes^L_AN\in\D^b(k)$ for all $M\in\mod A^\op$ and $N\in\mod A$ by Proposition~\ref{Prop-BimodComplex-TensorPerfect} (1'), if and only if $\gl A<\infty$, i.e., the algebra $A$ is smooth. 		
}\end{example}

\begin{remark}{\rm
		Let $A$ and $B$ be finite dimensional $k$-algebras. If $X$ is a $\otimes$-perfect complex of $A$-$B$-bimodules, applying Proposition~\ref{Prop-BimodComplex-TensorPerfect} (2) to $Z=B$ and (3) to $Y=A$ respectively, we obtain that $X$ is biperfect, i.e., both $_AX$ and $X_B$ are perfect. Conversely, a biperfect complex of $A$-$B$-bimodules need not be $\otimes$-perfect. For this, it is enough to consider Example~\ref{Ex-TensorPerfect-Bimod} (2) for any algebra of infinite global dimension.
}\end{remark}

The following result implies that all $\otimes$-perfect complexes of bimodules in $\D(A\otimes B^\op)$ form a thick triangulated subcategory of $\D^b(A\otimes B^\op)$. Note that $\D^b(A\otimes B^\op)=\thick((A\otimes B^\op)/\rad(A\otimes B^\op))$.

\begin{proposition} \label{Prop-TensorPerfect-Property}
	Let $A$ and $B$ be finite dimensional $k$-algebras. Then the following statements hold:
	
	{\rm(1)} Any finite direct sum of $\otimes$-perfect complexes of $A$-$B$-bimodules is $\otimes$-perfect.
	
	{\rm(2)} Any direct summand of a $\otimes$-perfect complex of $A$-$B$-bimodules is $\otimes$-perfect.
	
	{\rm(3)} Any shift of a $\otimes$-perfect complex of $A$-$B$-bimodules is $\otimes$-perfect.
	
	{\rm (4)} If $X\to Y\to Z\to$ is a triangle in $\D(A\otimes B^\op)$ then any two of $X, Y$ and $Z$ being $\otimes$-perfect implies that the third one is also $\otimes$-perfect. 	
\end{proposition}

\begin{proof}
	It is clear by Proposition~\ref{Prop-BimodComplex-TensorPerfect} (1'').
\end{proof}

\subsection{Comparison of $\otimes$-perfectness and perfectness}

The following result should be well-known and is more general than \cite[Corollary 5.3.10]{Zimmermann14}.

\begin{lemma} \label{Lem-TensorProdTop=TopTensorProd} 
	Let $A$ and $B$ be finite dimensional $k$-algebras, and either $A/\rad A$ or $B/\rad B$ a separable $k$-algebra.
	Then the following statements hold:
	
	{\rm (1)} $(A/\rad A)\otimes (B/\rad B)$ is a semisimple $k$-algebra.
	
	{\rm (2)} $\rad(A\otimes B) =(\rad A)\otimes B+A\otimes(\rad B)$.
	
	{\rm (3)} $(A\otimes B)/\rad(A\otimes B) \cong (A/\rad A)\otimes(B/\rad B)$.
\end{lemma}

\begin{proof}
	
	{\rm (1)} Since $A/\rad A$ or $B/\rad B$ is separable over $k$, by \cite[Corollary 18]{Eilenberg-Rosenberg-Zelinsky57}, we have $\gl((A/\rad A)\otimes(B/\rad B)) = \gl(A/\rad A) + \gl(B/\rad B)=0$.
	
	{\rm (2)} On one hand, $(\rad A)\otimes B+A\otimes(\rad B)$ is a nilpotent ideal of $A\otimes B$, so $\rad(A\otimes B) \supseteq (\rad A)\otimes B+A\otimes(\rad B)$. 
	On the other hand, $(A\otimes B)/((\rad A)\otimes B+A\otimes(\rad B)) \cong (A/\rad A)\otimes(B/\rad B)$. Indeed, we have a natural surjective homomorphism $(A\otimes B)/((\rad A)\otimes B+A\otimes(\rad B)) \twoheadrightarrow (A/\rad A)\otimes(B/\rad B)$, which is also injective since $\dim((A\otimes B)/((\rad A)\otimes B+A\otimes(\rad B))) = \dim((A/\rad A)\otimes(B/\rad B))$.
	By (1), $(A\otimes B)/((\rad A)\otimes B+A\otimes(\rad B)) \cong (A/\rad A)\otimes(B/\rad B)$ is a semisimple $k$-algebra. 
	Thus $\rad(A\otimes B) \subseteq (\rad A)\otimes B+A\otimes(\rad B)$. Hence $\rad(A\otimes B) = (\rad A)\otimes B+A\otimes(\rad B)$.
	
	{\rm (3)} By (2), we have $(A\otimes B)/\rad(A\otimes B) \cong (A\otimes B)/((\rad A)\otimes B+A\otimes(\rad B)) \cong (A/\rad A)\otimes(B/\rad B)$.
\end{proof}

\begin{example}{\rm 
	(1) Let $A$ be a finite dimensional {\it elementary} $k$-algebra, i.e., $A/\rad A \cong k^n$ for some positive integer $n$. Since $k$ is a separable $k$-algebra, by \cite[Chapter IX, Proposition 7.3]{Cartan-Eilenberg56}, $A/\rad A\cong k^n$ is a separable $k$-algebra.
	Note that any bound quiver algebra $kQ/I$, where $Q$ is a finite quiver and $I$ is an admissible ideal of the path algebra $kQ$, is elementary.
		
	(2) Let $k$ be a perfect field such as fields of characteristic zero and finite fields, and $A$ a finite dimensional $k$-algebra. It follows from \cite[Corollary 6.1.4]{Drozd-Kirichenko94} that $A/\rad A$ is a separable $k$-algebra.	
}\end{example}

The following result compares $\otimes$-perfectness and perfectness for complexes of bimodules.

\begin{proposition} \label{Prop-HomSmooth-TensorPerfect-BimodComplex}
	Let $A$ and $B$ be finite dimensional $k$-algebras. 
	If a complex $X$ of $A$-$B$-bimodules is perfect then it is $\otimes$-perfect. 
	The converse also holds if either $A/\rad A$ or $B/\rad B$ is a separable $k$-algebra.
\end{proposition}

\begin{proof}
	If $X\in\per(A\otimes B^\op)$ then, by Lemma~\ref{Lem-LeftModComplex-Perfect} (2'), we have $\ol{A}\otimes^L_AX\otimes^L_B\ol{B} \cong (\ol{A}\otimes \ol{B})\otimes^L_{A\otimes B^\op}X \in\D^b(k)$. Thus $X$ is $\otimes$-perfect by Proposition~\ref{Prop-BimodComplex-TensorPerfect} (1''). 
	
	If $X$ is $\otimes$-perfect and either $A/\rad A$ or $B/\rad B$ is a separable $k$-algebra then, by Lemma~\ref{Lem-TensorProdTop=TopTensorProd} (3) and Proposition~\ref{Prop-BimodComplex-TensorPerfect} (1''), we have $\overline{A\otimes B^\op}\otimes^L_{A\otimes B^\op} X \cong (\ol{A}\otimes \ol{B})\otimes^L_{A\otimes B^\op}X \cong \ol{A}\otimes^L_AX\otimes^L_B\ol{B} \in\D^b(k)$. Thus $X$ is perfect by Lemma~\ref{Lem-LeftModComplex-Perfect} (2'').	
\end{proof}

The following example implies that a smooth algebra $A$, i.e., an algebra with $\gl A<\infty$, need not be homologically smooth, i.e., $\pd_{A^e}A<\infty$. Thus a $\otimes$-perfect complex of bimodules need not be perfect.

\begin{example} \label{Ex-TensorPer-NotPer}{\rm
	Let $A$ be a finite dimensional semisimple inseparable algebra over a field $k$ of characteristic $p>0$. For instance, let $k:=\mathbb{Z}_p(t)$ be the field of fractions of the polynomial ring $\mathbb{Z}_p[t]$, and $A:=k[x]/(x^p-t)$. In view of $\gl A=0$, by Example~\ref{Ex-TensorPerfect-Bimod}, $A$ is a $\otimes$-perfect $A$-bimodule. By \cite[Theorem 11.1]{Hochschild47}, we have $\pd_{A^e}A=\infty$. Thus $A$ is not perfect over $A^e$. 
}\end{example}

\subsection{$\otimes$-coperfect complexes of bimodules}

\begin{lemma}\label{Lem-BimodComplex-LeftCoper} 
	Let $A$ and $B$ be finite dimensional $k$-algebras, and $X$ a complex of $A$-$B$-bimodules. Then the following conditions are equivalent:
	
	{\rm (1)} $_AX\in K^b(\inj A)$. 	
	
	{\rm (1')} $\RHom_A(Y,X)\in \D^b(B^\op)$ for all $Y\in\D^b(A)$.
	
	{\rm (1'')} $\RHom_A(\ol{A},X)\in \D^b(B^\op)$.
	
	{\rm (2)} $X^*:=\Hom_k(X,k)\in K^b(\proj A^\op)$.
\end{lemma}

\begin{proof}
	(1)$\Rightarrow$(2): The $k$-dual $(-)^*:\D^b(A)\to\D^b(A^\op)$ restricts to the $k$-dual $(-)^*:K^b(\inj A)\to K^b(\proj A^\op)$.
	
	(2)$\Rightarrow$(1''): If $X^*\in K^b(\proj A^\op)$ then $X^*\otimes^L_A\ol{A}\in\D^b(k)$ by Lemma~\ref{Lem-RightModComplex-Perfect}. Thus $\RHom_A(\ol{A},X)\cong \RHom_A(\ol{A},X^{**})\cong (X^*\otimes^L_A\ol{A})^*\in\D^b(k)$. So $\RHom_A(\ol{A},X)\in \D^b(B^\op)$.
		
	(1'')$\Rightarrow$(1): If $\RHom_A(\ol{A},X)\in \D^b(B^\op)$ then $_AX\in K^b(\inj A)$. Otherwise, $_AX$ admits an infinite minimal semi-injective resolution. Furthermore, $\RHom_A(\ol{A},X)\notin \D^b(k)$. It is a contradiction.

	(1')$\Leftrightarrow$(1''): It follows from $\D^b(A)=\thick(\ol{A})$.
\end{proof}

\begin{proposition} \label{Prop-BimodComplex-TensorCoper} 
	Let $A$ and $B$ be finite dimensional $k$-algebras, and $X$ a complex of $A$-$B$-bimodules. Then the following conditions are equivalent:
	
	{\rm (1)} $_AX\otimes^L_BY\in K^b(\inj A)$ for all $Y\in\D^b(B)$.	
	
	{\rm (1')} $_AX\otimes^L_BM\in K^b(\inj A)$ for all $M\in\mod B$.
	
	{\rm (1'')} $_AX\otimes^L_B\ol{B}\in K^b(\inj A)$.
	
	{\rm (2)} $\RHom_{B^\op}(X,Y)\in\per(A^\op)$ for all $Y\in\D^b(B^\op)$.
	
	{\rm (2')} $\RHom_{B^\op}(X,M)\in\per(A^\op)$ for all $M\in\mod B^\op$.
	
	{\rm (2'')} $\RHom_{B^\op}(X,\ol{B})\in\per(A^\op)$.
	
	{\rm (3)} $\RHom_A(Z,X\otimes^L_BY)\in\D^b(k)$ for all $Z\in\D^b(A)$ and $Y\in\D^b(B)$.
	
	{\rm (3')} $\RHom_A(N,X\otimes^L_BM)\in\D^b(k)$ for all $N\in\mod A$ and $M\in\mod B$.	
	
	{\rm (3'')} $\RHom_A(\ol{A},X\otimes^L_B\ol{B})\in\D^b(k)$.			
\end{proposition}

\begin{proof}
	(1)$\Leftrightarrow$(2): Due to the $k$-duals between $\D^b(A)$ and $\D^b(A^\op)$ and their restrictions between $K^b(\inj A)$ and $K^b(\proj A^\op)$, $X\otimes^L_BY\in K^b(\inj A)$ for all $Y\in\D^b(B)$ if and only if $(X\otimes^L_BY)^*\cong\RHom_{B^\op}(X,Y^*)\in K^b(\proj A^\op)$ for all $Y\in\D^b(B)$, if and only if $\RHom_{B^\op}(X,Y)\in K^b(\proj A^\op)$ for all $Y\in\D^b(B^\op)$.
	
	(1)$\Leftrightarrow$(3): It follows from Lemma~\ref{Lem-BimodComplex-LeftCoper}.
	
	(1)$\Leftrightarrow$(1')$\Leftrightarrow$(1''), (2)$\Leftrightarrow$(2')$\Leftrightarrow$(2''), and (3)$\Leftrightarrow$(3')$\Leftrightarrow$(3''): Similar to the proof of Lemma~\ref{Lem-LeftModComplex-Perfect}.			
\end{proof}

\begin{definition}{\rm 
		Let $A$ and $B$ be finite dimensional $k$-algebras. A complex $X$ of $A$-$B$-bimodules is said to be {\it tensor-coperfect} or {\it $\otimes$-coperfect} if it satisfies the nine equivalent conditions in Proposition~\ref{Prop-BimodComplex-TensorCoper}.
}\end{definition}

The following result implies that all $\otimes$-coperfect complexes of bimodules in $\D(A\otimes B^\op)$ form a thick triangulated subcategory of $\D^b(A\otimes B^\op)$.

\begin{proposition} \label{Prop-TensorCoper-Property}
	Let $A$ and $B$ be finite dimensional $k$-algebras. Then the following statements hold:
	
	{\rm(1)} Any finite direct sum of $\otimes$-coperfect complexes of $A$-$B$-bimodules is $\otimes$-coperfect.
	
	{\rm(2)} Any direct summand of a $\otimes$-coperfect complex of $A$-$B$-bimodules is $\otimes$-coperfect.
	
	{\rm(3)} Any shift of a $\otimes$-coperfect complex of $A$-$B$-bimodules is $\otimes$-coperfect.
	
	{\rm (4)} If $X\to Y\to Z\to$ is a triangle in $\D(A\otimes B^\op)$ then any two of $X, Y$ and $Z$ being $\otimes$-coperfect implies that the third one is also $\otimes$-coperfect. 	
\end{proposition}

\begin{proof}
	It is clear by Proposition~\ref{Prop-BimodComplex-TensorCoper} (1'').
\end{proof}

\begin{example}{\rm 
	(1) By Proposition~\ref{Prop-BimodComplex-TensorCoper} (1''), the $A$-$B$-bimodule $A^*\otimes B$ is $\otimes$-coperfect. Then all objects in the thick triangulated subcategory $\thick(A^*\otimes B)$ of $\D^b(A\otimes B^\op)$ are $\otimes$-coperfect by Proposition~\ref{Prop-TensorCoper-Property}.
	
	(2) Viewed as a stalk complex concentrated on degree zero, the identity $A$-bimodule $A$ is $\otimes$-coperfect if and only if all finite dimensional left $A$-modules are of finite injective dimension by Proposition~\ref{Prop-BimodComplex-TensorCoper} (1'), if and only if $\gl A<\infty$, i.e., $A$ is smooth.
	
	(3) When $A$ is a finite dimensional Gorenstein algebra, or equivalently, $K^b(\proj A)=K^b(\inj A)$, a complex $X$ of $A$-$B$-bimodules is $\otimes$-perfect if and only if it is $\otimes$-coperfect by Proposition~\ref{Prop-BimodComplex-TensorPerfect} (2'') and Proposition~\ref{Prop-BimodComplex-TensorCoper} (1'').
}\end{example}

\section{Eventually homological isomorphisms}

In this section, we refine eventually homological isomorphisms (EHI for short) between derived categories, and give their characterizations in terms of $\otimes$-(co)perfect complexes of bimodules, constructions by opposite algebras and tensor product algebras, examples induced by homomorphisms of algebras and idempotents, and applications to transfer Gorensteinness and reduce Gorenstein symmetry conjecture.

\subsection{Definitions }

The eventually homological isomorphisms on module categories, more general, abelian categories, were introduced in \cite[Section 3]{Psaroudakis-Skartsaeterhagen-Solberg14}. 
Let $F : {\cal B} \to {\cal C}$ be a functor between abelian categories and $t$ an integer. The functor $F$ is called a {\it $t$-eventually homological isomorphism} if there is an isomorphism $\Ext^i_{\cal B}(X,Y)\cong\Ext^i_{\cal C}(F(X), F(Y))$ of abelian groups for all objects $X,Y \in {\cal B}$ and $i > t$. If the functor $F$ is a $t$-eventually homological isomorphism for some $t$ then $F$ is called an {\it eventually homological isomorphism}.

\begin{remark}{\rm
	It is strange that the definition of an eventually homological isomorphism $F : {\cal B} \to {\cal C}$ has nothing to do with the functoriality of $F$.	However, it is interesting that in the situation of recollements of module categories, or equivalently, idempotents \cite{Jans65}, the exact functor $e:=a\Lambda\otimes_\Lambda -:\mod \Lambda\to\mod a\Lambda a$, where $a$ is an idempotent of the Artin algebra $\Lambda$, is an eventually homological isomorphism if and only if there is an integer $s$ such that for every pair of $\Lambda$-modules $M$ and $N$, and
	every $j>s$, the map $e_{M,N}^j:\Ext^j_{\Lambda}(M,N)\to\Ext^j_{a\Lambda a}(e(M), e(N))$ is an isomorphism (Ref. \cite[Corollary 3.12]{Psaroudakis-Skartsaeterhagen-Solberg14}).
}\end{remark}

It seems more reasonable to define an eventually homological isomorphism on abelian categories as follows: Let $F : {\cal B} \to {\cal C}$ be an exact functor between abelian categories and $t$ an integer. The functor $F$ is called a {\it $t$-eventually homological isomorphism} if $F$ induces a group isomorphism $\Ext^i_{\cal B}(X,Y)\cong\Ext^i_{\cal C}(F(X), F(Y))$ sending an equivalence class of an $i$-extension $0\to Y\xrightarrow{d^0} E^1\xrightarrow{d^1}\cdots\xrightarrow{d^{i-1}} E^i\xrightarrow{d^i} X\to 0$ to the equivalence class of the $i$-extension $0\to F(Y)\xrightarrow{F(d^0)} F(E^1)\xrightarrow{F(d^1)}\cdots\xrightarrow{F(d^{i-1})} F(E^i)\xrightarrow{F(d^i)} F(X)\to 0$ for all objects $X,Y \in {\cal B}$ and $i > t$. If the functor $F$ is a $t$-eventually homological isomorphism for some $t$ then $F$ is called an {\it eventually homological isomorphism}.

An eventually homological isomorphism on module categories of finite dimensional algebras can be defined simply as the following.

\begin{definition} \label{Def-EHI-ModCat} {\rm
	Let $A$ and $B$ be finite dimensional algebras. An exact functor $F: \mod B \to \mod A$ is called an {\it eventually homological isomorphism} if  for each pair of objects $X,Y \in \mod B$, $F$ induces a linear isomorphism $\Ext^i_B(X,Y)\cong\Ext^i_A(F(X), F(Y))$ sending an equivalence class of an $i$-extension $0\to Y\xrightarrow{d^0} E^1\xrightarrow{d^1}\cdots\xrightarrow{d^{i-1}} E^i\xrightarrow{d^i} X\to 0$ to the equivalence class of the $i$-extension $0\to F(Y)\xrightarrow{F(d^0)} F(E^1)\xrightarrow{F(d^1)}\cdots\xrightarrow{F(d^{i-1})} F(E^i)\xrightarrow{F(d^i)} F(X)\to 0$ for all $i\gg 0$.
}\end{definition}

\begin{remark}{\rm
	It is unnecessary to define a ``$t$-eventually homological isomorphism'' on module categories of finite dimensional algebras, since an eventually homological isomorphism $F:\mod B\to\mod A$ in the sense of Definition~\ref{Def-EHI-ModCat} is always a ``$t$-eventually homological isomorphism'' for some positive integer $t\gg 0$ by Proposition~\ref{Prop-EHI-ModCat-DerCat} and Theorem~\ref{Thm-EvHomIso-TensorPerfect-BimodComplex}. Indeed, by Eilenberg-Watt theorem, we have $F\cong M\otimes_B-$ for a finite dimensional $A$-$B$-bimodule $M$ which is projective as a right $B$-module. Then the positive integer $t$ can be given by the length of the minimal semi-projective resolution of the $B$-module complex $\Cone(\ev^B_M)\otimes^L_B(B/\rad B)$, where $\ev^B_M$ is the evaluation morphism $\RHom_{B^\op}(M,B)\otimes^L_AM\to B$ (Ref. Theorem~\ref{Thm-EvHomIso-TensorPerfect-BimodComplex}).
}\end{remark}

It seems more natural and general to define an eventually homological isomorphism on derived categories rather than abelian categories. The eventually homological isomorphisms on derived categories were introduced in \cite[Definition 3]{Qin20}.
Let ${\cal B}$ and ${\cal C}$ be abelian categories. A triangle functor $F : \D({\cal B})\to\D({\cal C})$ is called an {\it eventually homological isomorphism} if there exists an integer $t$ such that for every pair of objects $B, B'\in {\cal B}$ and every $i>t$, there is an isomorphism 
$\Hom_{\D({\cal B})}(B,B'[i]) \cong \Hom_{\D({\cal C})}(F(B),F(B')[i])$ of abelian groups.

\begin{remark}{\rm
	The above definition of an eventually homological isomorphism on derived categories has nothing to do with the functoriality of $F$. However, in the proof of \cite[Lemma 1]{Qin20}, an eventually homological isomorphism should require that the map $\Hom_{\D({\cal B})}(B,B'[i]) \cong \Hom_{\D({\cal C})}(F(B),F(B')[i])$ is induced by the triangle functor $F : \D({\cal B})\to\D({\cal C})$.
}\end{remark}

Now we redefine an eventually homological isomorphism on derived categories which is more natural and fruitful.

\begin{definition} \label{Def-EHI-DerCat} {\rm 
	Let $A$ and $B$ be finite dimensional algebras. A triangle functor $F:\D^b(B)\to \D^b(A)$ is called an {\it eventually homological isomorphism} if for each pair of objects $X,Y\in\D^b(B)$, the linear map $F_{X,Y[i]}:\Hom_{\D^b(B)}(X,Y[i])\to \Hom_{\D^b(A)}(F(X),F(Y)[i])$ is an isomorphism for all $i\gg 0$. Moreover, an eventually homological isomorphism is said to be {\it standard} if it is a derived tensor product functor $X\otimes^L_B-:\D^b(B)\to \D^b(A)$	given by a right perfect $A$-$B$-bimodule complex $X$.
}\end{definition}

Clearly, every fully faithful triangle functor, in particular, derived equivalence, $F:\D^b(B)\to \D^b(A)$ is an eventually homological isomorphism. Since not all derived equivalences are standard \cite{Hu-Xi-Zhang26,Zhang26}, there exist nonstandard eventually homological isomorphisms. Moreover, if $A,B,C$ are finite dimensional algebras then the composition $FG:\D^b(C)\to \D^b(A)$ of two eventually homological isomorphisms $F:\D^b(B)\to \D^b(A)$ and $G:\D^b(C)\to \D^b(B)$ is also an eventually homological isomorphism.

\begin{remark} {\rm
	More generally, one may define an eventually homological isomorphism on triangulated categories as follows: Let $\T$ and $\T'$ be triangulated categories. A triangle functor $F:\T\to \T'$ is called an {\it eventually homological isomorphism} if for each pair of objects $X,Y\in\T$, the linear map $F_{X,Y[i]}:\T(X,Y[i])\to \T'(F(X),F(Y)[i])$ is an isomorphism for all $i\ll 0$ and $i\gg 0$. Nonetheless, it is unnecessary for us at present. In fact, we know little about an eventually homological isomorphism on unbounded derived categories. 	
}\end{remark}

The most important exact functors on module categories are tensor product functors induced by one-sided projective bimodules. The following result implies that a standard eventually homological isomorphism on module categories induces a standard eventually homological isomorphism on derived categories.

\begin{proposition} \label{Prop-EHI-ModCat-DerCat}
	Let $A$ and $B$ be finite dimensional algebras, $_AM_B$ a finite dimensional $A$-$B$-bimodule, and $M_B$ projective. Then the exact functor $_AM\otimes_B-: \mod B\to\mod A$ is an eventually homological isomorphism in the sense of Definition~\ref{Def-EHI-ModCat} if and only if the triangle functor $_AM\otimes_B-:\D^b(B)\to\D^b(A)$ is an eventually homological isomorphism in the sense of Definition~\ref{Def-EHI-DerCat}.
\end{proposition}

\begin{proof}
	The exact functor $F:=M\otimes_B-: \mod B\to\mod A$ is an eventually homological isomorphism if and only if for each pair of objects $X,Y\in\mod B$, $F$ induces an isomorphism $F^i_{X,Y}: \Ext^i_B(X,Y) \to \Ext^i_A(M\otimes_BX,M\otimes_BY)$ sending an equivalence class of an $i$-extension $0\to Y\xrightarrow{d^0} E^1\xrightarrow{d^1}\cdots\xrightarrow{d^{i-1}} E^i\xrightarrow{d^i} X\to 0$ to the equivalence class of the $i$-extension $0\to F(Y)\xrightarrow{F(d^0)} F(E^1)\xrightarrow{F(d^1)}\cdots\xrightarrow{F(d^{i-1})} F(E^i)\xrightarrow{F(d^i)} F(X)\to 0$ for all $i\gg0$, if and only if for each pair of objects $X,Y\in\mod B$, the linear map $F'_{X,Y[i]}: \Hom_{\D^b(B)}(X,Y[i]) \to \Hom_{\D^b(A)}(M\otimes_BX,M\otimes_BY[i])$ induced by the triangle functor $F':=M\otimes_B-: \D^b(B)\to\D^b(A)$ is an isomorphism for all $i\gg0$, if and only if for each pair of objects $X,Y\in\D^b(B)$, the linear map $F'_{X,Y[i]}: \Hom_{\D^b(B)}(X,Y[i]) \to \Hom_{\D^b(A)}(M\otimes_BX,M\otimes_BY[i])$ induced by the triangle functor $F'=M\otimes_B-: \D^b(B)\to\D^b(A)$ is an isomorphism for all $i\gg0$ due to $\D^b(A)=\tri(\mod A)$, the smallest full triangulated subcategory of $\D(A)$ containing $\mod A$, i.e., the triangle functor $F'=M\otimes_B-:\D^b(B)\to\D^b(A)$ is an eventually homological isomorphism. 
\end{proof}

\begin{remark}{\rm 
	The example \cite[Example 1.4]{Qin-Xu-Zhang-Zhou24} gives an eventually homological isomorphism on derived categories but not an eventually homological isomorphism on module categories, since the $B$-bimodule $A/B$ is perfect but not projective. See the Corollary~\ref{Cor-BoundExt-EHI} below. 
}\end{remark}

\subsection{Characterizations}

In this subsection, we give characterizations of standard eventually homological isomorphisms in terms of $\otimes$-(co)perfect complexes of bimodules. For this, we need to do some preparations.

\begin{lemma} \label{Lem-F_XY-Unit}
	Let $\A$ and $\B$ be categories and $F:\A\to\B$ a functor. 
	
	{\rm (1)} If $L:\B\to\A$ is a left adjoint of $F$, $\alpha:\A(L?,-)\to\B(?,F-)$ is the adjoint isomorphism, and $\eta: LF\to \Id_\A$ is the counit of the adjoint pair $(L,F)$, then $F_{A,A'}=\alpha_{FA,A'}\circ \A(\eta_A,A')$ for all $A,A'\in\A$, i.e., the following diagram is commutative.
	$$\begin{tikzcd}
		\A(A,A') \arrow[r,"{F_{A,A'}}"] \arrow[d,"{\A(\eta_A,A')}"'] & \B(FA,FA') \\
		\A(LFA,A') \arrow[ur,"{\alpha_{FA,A'}}"'] & 		
	\end{tikzcd}$$
	
	{\rm (2)} If $R:\B\to\A$ is a right adjoint of $F$, $\beta:\B(F-,?)\to\A(-,R?)$ is the adjoint isomorphism, and $\varepsilon: \Id_\A\to RF$ is the unit of the adjoint pair $(F,R)$, then $F_{A,A'}=\beta_{A,FA'}^{-1}\circ \A(A,\varepsilon_{A'})$ for all $A,A'\in\A$, i.e., the following diagram is commutative.
	$$\begin{tikzcd}
		\A(A,A') \arrow[r,"{F_{A,A'}}"] \arrow[d,"{\A(A,\varepsilon_{A'})}"'] & \B(FA,FA') \\
		\A(A,RFA') \arrow[ur,"{\beta_{A,FA'}^{-1}}"'] & 		
	\end{tikzcd}$$

    {\rm (3)} If $L:\B\to\A$ is a left adjoint of $F$, $\alpha:\A(L?,-)\to\B(?,F-)$ is the adjoint isomorphism, $\varepsilon: \Id_\B\to FL$ is the unit of the adjoint pair $(L,F)$, $R:\B\to\A$ is a right adjoint of $F$, $\beta:\B(F-,?)\to\A(-,R?)$ is the adjoint isomorphism, and $\eta:FR \to \Id_\B$ is the counit of the adjoint pair $(F,R)$, then $\B(\varepsilon_B,B')=\B(B,\eta_{B'})\circ\alpha_{B,RB'}\circ\beta_{LB,B'}$ for all $B,B'\in\B$, i.e., the following diagram is commutative.
    $$\begin{tikzcd}
	\B(FLB,B') \arrow[r,"{\B(\varepsilon_B,B')}"] \arrow[d,"{\beta_{LB,B'}}"'] & \B(B,B')\\
	\A(LB,RB') \arrow[r,"{\alpha_{B,RB'}}"] & \B(B,FRB').\arrow[u,"{\B(B,\eta_{B'})}"']		
	\end{tikzcd}$$
\end{lemma}

\begin{proof}
	(1) For any $f\in\A(A,A')$, we need to show $Ff=\alpha_{FA,A'}(f\eta_A)$. By the naturality of the adjoint isomorphism, we have the following commutative diagram,
	$$\begin{tikzcd}
		\A(LFA,A) \arrow[r,"{\alpha_{FA,A}}"] \arrow[d,"{\A(LFA,f)}"'] & \B(FA,FA) \arrow[d,"{\B(FA,Ff)}"]  \\
		\A(LFA,A') \arrow[r,"{\alpha_{FA,A'}}"] & \B(FA,FA')
	\end{tikzcd}$$
	i.e., $\B(FA,Ff) \circ \alpha_{FA,A} = \alpha_{FA,A'} \circ \A(LFA,f)$. So $(\B(FA,Ff) \circ \alpha_{FA,A})(\eta_A) =   (\alpha_{FA,A'} \circ \A(LFA,f))(\eta_A)$, i.e., $Ff=\alpha_{FA,A'}(f\eta_A)$.
	
	(2) For any $f\in\A(A,A')$, we need to show that $\beta_{A,FA'}(Ff)=\varepsilon_{A'}f$. By the naturality of the adjoint isomorphism, we have the following commutative diagram,
	$$\begin{tikzcd}
		\B(FA',FA') \arrow[r,"{\beta_{A',FA'}}"] \arrow[d,"{\B(Ff,FA')}"'] & \A(A',RFA') \arrow[d,"{\A(f,RFA')}"]  \\
		\B(FA,FA') \arrow[r,"{\beta_{A,FA'}}"] & \A(A,RFA')
	\end{tikzcd}$$
	i.e., $\beta_{A,FA'} \circ \B(Ff,FA') = \A(f,RFA') \circ \beta_{A',FA'}$. So $(\beta_{A,FA'} \circ \B(Ff,FA'))   (1_{FA'}) = (\A(f,RFA') \circ \beta_{A',FA'})(1_{FA'})$, i.e., $\beta_{A,FA'}(Ff)=\varepsilon_{A'}f$.

(3) For any $f\in\B(FLB,B')$, we need to show that $f\varepsilon_B=\eta _{B'} \circ \alpha_{B,RB'}(\beta_{LB,B'}(f))$. 
By the naturality of the adjoint isomorphism, we have the following commutative diagram 
$$\begin{tikzcd}
	\A(LB,LB) \arrow[r,"{\alpha_{B,LB}}"] \arrow[d,"{\A(LB,\beta_{LB,B'}(f))}"'] & \B(B,FLB) \arrow[d,"{\B(B,F(\beta_{LB,B'}(f)))}"]  \\
	\A(LB,RB') \arrow[r,"{\alpha_{B,RB'}}"] & \B(B,FRB').
\end{tikzcd}$$
So $(\alpha_{B,RB'} \circ\A(LB,\beta_{LB,B'}(f)))(1_{LB}) = (\B(B,F(\beta_{LB,B'}(f)))\circ \alpha_{B,LB})(1_{LB})$, i.e., $\alpha_{B,RB'}(\beta_{LB,B'}(f)) =F(\beta_{LB,B'}(f))\circ \varepsilon_B$.
On the other hand, we have the following commutative diagram
$$\begin{tikzcd}
	\B(FRB',B') \arrow[r,"{\beta_{RB',B'}}"] \arrow[d,"{\B(F(\beta_{LB,B'}(f)),B')}"'] & \A(RB',RB') \arrow[d,"{\A(\beta_{LB,B'}(f), RB')}"]  \\
	\B(FLB,B') \arrow[r,"{\beta_{LB,B'}}"] & \A(LB,RB').
\end{tikzcd}$$
So $(\beta_{LB,B'} \circ \B(F(\beta_{LB,B'}(f)),B'))(\eta_{B'})= (\A(\beta_{LB,B'}(f), RB') \circ \beta_{RB',B'})(\eta_{B'})$,
i.e., $\beta_{LB,B'}(\eta_{B'} \circ \linebreak F(\beta_{LB,B'}(f)))=\beta_{LB,B'}(f)$. Since $\beta_{LB,B'}$ is an isomorphism,
we have $\eta_{B'} \circ F(\beta_{LB,B'}(f))=f$. 
Thus $\eta _{B'} \circ \alpha_{B,RB'}(\beta_{LB,B'}(f))=\eta _{B'}\circ (F(\beta_{LB,B'}(f))\circ \varepsilon_B)=f\varepsilon_B$.
\end{proof}

\begin{lemma} \label{Lem-Cone-Cone}
	Let $A, B, C$ be three algebras, $X$ a complex of $A$-$B$-bimodules, $\varepsilon^C : \Id_{\D(B\otimes C^\op)} \to \RHom_A(X,\linebreak X\otimes^L_B-)$ and $\eta^C : X\otimes^L_B\RHom_A(X,-) \to \Id_{\D(A\otimes C^\op)}$ the unit and counit of the adjoint pair 
	$$X\otimes^L_B-\ :\ \D(B\otimes C^\op) \rightleftarrows \D(A\otimes C^\op)\ :\ \RHom_A(X,-)$$ 
	respectively. Then there exist morphisms $u_Y:\Cone(\varepsilon^B_B)\otimes^L_BY \to \Cone(\varepsilon^C_Y)$ in $\D(B\otimes C^\op)$ and $c_Z:\Cone(\eta^A_A)\otimes^L_AZ \to \Cone(\eta^C_Z)$ in $\D(A\otimes C^\op)$ such that the following two diagrams are commutative for all $Y\in\D(B\otimes C^\op)$ and $Z\in\D(A\otimes C^\op)$,
	$$\begin{tikzcd}[column sep=25]
		B\otimes^L_BY \arrow[r,"{\varepsilon^B_B\otimes^L_BY}"] \arrow[d,"\cong"] & \RHom_A(X,X)\otimes^L_BY \arrow[r] \arrow[d,"l_Y"]  & \Cone(\varepsilon^B_B)\otimes^L_BY \arrow[r] \arrow[d,"u_Y"] & B\otimes^L_BY[1] \arrow[d,"\cong"] \\
		Y \arrow[r,"{\varepsilon^C_Y}"] & \RHom_A(X,X\otimes^L_B Y) \arrow[r] & \Cone(\varepsilon^C_Y) \arrow[r] & Y[1]	
	\end{tikzcd}$$
	$$\begin{tikzcd}[column sep=20]
		X\otimes^L_B\RHom_A(X,A)\otimes^L_AZ \arrow[r,"{\eta^A_A\otimes^L_AZ}"] \arrow[d,"{X\otimes^L_Bm_Z}"] & A\otimes^L_AZ \arrow[d,"\cong"] \arrow[r] & \Cone(\eta^A_A)\otimes^L_AZ \arrow[d,"{c_Z}"] \arrow[r] & X\otimes^L_B\RHom_A(X,A)\otimes^L_AZ[1] \arrow[d,"{X\otimes^L_Bm_Z[1]}"] \\
	    X\otimes^L_B\RHom_A(X,Z) \arrow[r,"{\eta^C_Z}"] & Z\arrow[r] &\Cone(\eta^C_Z) \arrow[r] & X\otimes^L_B\RHom_A(X,Z)[1]
	\end{tikzcd}$$
	where the morphism $l_Y:\RHom_A(X,X)\otimes^L_BY\to \RHom_A(X,X\otimes^L_B Y)$ is induced by the morphism 
	$l_Y:\Hom_A({\bf p}X,{\bf p}X)\otimes_B{\bf p}Y\to \Hom_A({\bf p}X,{\bf p}X\otimes_B {\bf p}Y)$ in the homotopy categories ${\cal H}(B\otimes C^\op)$ 
	given by $l_Y(f\otimes y)(x):=(-1)^{|x||y|}f(x)\otimes y$ for all $f\in \Hom_A({\bf p}X,{\bf p}X), x\in{\bf p}X$ and $y\in{\bf p}Y$, 
	and the morphism $m_Z:\RHom_A(X,A)\otimes^L_AZ\to \RHom_A(X,Z)$ is induced by the morphism 
	$m_Z:\Hom_A({\bf p}X,A)\otimes_A{\bf p}Z\to \Hom_A({\bf p}X,{\bf p}Z)$ in the homotopy categories ${\cal H}(B\otimes C^\op)$
	given by $m_Z(g\otimes z)(x):=(-1)^{|x||z|}g(x)z$ for all $g\in \Hom_A({\bf p}X,A), x\in{\bf p}X$ and $z\in{\bf p}Z$. Here, ${\bf p}$ is a cofibrant replacement functor sending a complex to its cofibrant resolution.
	Furthermore, if $_AX\in\per(A)$ then the morphisms $u_Y:\Cone(\varepsilon^B_B)\otimes^L_BY \to \Cone(\varepsilon^C_Y)$ and $c_Z:\Cone(\eta^A_A)\otimes^L_AZ \to \Cone(\eta^C_Z)$ are isomorphisms. 		
\end{lemma}

\begin{proof} It is easy to check that the following two diagrams in the homotopy categories ${\cal H}(B\otimes C^\op)$ and ${\cal H}(A\otimes C^\op)$ respectively are commutative.
	$$\begin{tikzcd}[column sep=25]
		B\otimes_B{\bf p}Y \arrow[r,"{\varepsilon^B_B\otimes_B{\bf p}Y}"] \arrow[d,"\cong"] & \Hom_A({\bf p}X,{\bf p}X)\otimes_B{\bf p}Y \arrow[d,"l_Y"] \\
		{\bf p}Y \arrow[r,"{\varepsilon^C_Y}"] & \Hom_A({\bf p}X,{\bf p}X\otimes_B {\bf p}Y)	
		\end{tikzcd}\quad\quad
		\begin{tikzcd}[column sep=25]
		{\bf p}X\otimes_B\Hom_A({\bf p}X,A)\otimes_A{\bf p}Z \arrow[r,"{\eta^A_A\otimes_A{\bf p}Z}"] \arrow[d,"{{\bf p}X\otimes_Bm_Z}"] & A\otimes_A{\bf p}Z \arrow[d,"\cong"] \\
		{\bf p}X\otimes_B\Hom_A({\bf p}X,{\bf p}Z) \arrow[r,"{\eta^C_Z}"] & {\bf p}Z	
		\end{tikzcd}$$
	\noindent Then the following two diagrams in the derived categories $\D(B\otimes C^\op)$ and $\D(A\otimes C^\op)$ respectively are commutative.
    $$\begin{tikzcd}[column sep=25]
		B\otimes_BY \arrow[r,"{\varepsilon^B_B\otimes^L_BY}"] \arrow[d,"\cong"] & \RHom_A(X,X)\otimes^L_BY \arrow[d,"l_Y"] \\
		Y \arrow[r,"{\varepsilon^C_Y}"] & \RHom_A(X,X\otimes^L_B Y)	
	\end{tikzcd}\quad\quad
	\begin{tikzcd}[column sep=25]
		X\otimes^L_B\RHom_A(X,A)\otimes^L_AZ \arrow[r,"{\eta^A_A\otimes^L_AZ}"] \arrow[d,"{X\otimes^L_Bm_Z}"] & A\otimes_AZ \arrow[d,"\cong"] \\
		X\otimes^L_B\RHom_A(X,Z) \arrow[r,"{\eta^C_Z}"] & Z	
	\end{tikzcd}$$
	\noindent By the axioms of triangulated category, there exist morphisms $u_Y:\Cone(\varepsilon^B_B)\otimes^L_BY \to \Cone(\varepsilon^C_Y)$ in $\D(B\otimes C^\op)$ and $c_Z:\Cone(\eta^A_A)\otimes^L_AZ \to \Cone(\eta^C_Z)$ in $\D(A\otimes C^\op)$ respectively such that the two diagrams in Lemma~\ref{Lem-Cone-Cone} are commutative.
	
	Furthermore, if $_AX\in\per(A)$ then the morphisms $l_Y:\RHom_A(X,X)\otimes^L_BY\to \RHom_A(X,X\otimes^L_B Y)$ and $m_Z:\RHom_A(X,A)\otimes^L_AZ\to \RHom_A(X,Z)$ are isomorphisms. 
	Thus the morphisms $u_Y:\Cone(\varepsilon^B_B)\otimes^L_BY \to \Cone(\varepsilon^C_Y)$ and $c_Z:\Cone(\eta^A_A)\otimes^L_AZ \to \Cone(\eta^C_Z)$ are isomorphisms. 
\end{proof}

\begin{remark} \label{Rmk-UnitAct-CounitEv} {\rm
		(i) The morphism $\varepsilon^B_B: B\to \RHom_A(X,X)$ is just the {\it right action} ({\it coevaluation}) {\it morphism} $r^B_X: B\to \RHom_A(X,X)$, which is induced by the morphism $r^B_{{\bf p}X}: B\to \Hom_A({\bf p}X,{\bf p}X)$ in the homotopy categories ${\cal H}(B^e)$ given by $r^B_{{\bf p}X}(b)(x):=xb$ for all $b\in B$ and $x\in{\bf p}X$. 
		
		(ii) The morphism $\eta^A_A: X\otimes^L_B\RHom_A(X,A)\to A$ is just the {\it evaluation morphism} $\ev^A_X: X\otimes^L_B\RHom_A(X,A)\to A$, which is induced by the morphism $\ev^A_{{\bf p}X}: {\bf p}X\otimes_B\Hom_A({\bf p}X,A)\to A$ in the homotopy categories ${\cal H}(A^e)$ given by $\ev^A_{{\bf p}X}(x\otimes f):=(-1)^{|f||x|}f(x)$ for all $x\in{\bf p}X$ and $f\in\Hom_A({\bf p}X,A)$. 
		
		(iii) In the case of $C=k$, we denote the unit $\varepsilon^k$ and $\eta^k$ just by $\varepsilon$ and $\eta$ respectively for short.
}\end{remark}

\begin{example}{\rm 
		Let $A$ be an algebra, and $e$ an idempotent in $A$. Then the triangle functor $eA\otimes^L_A-\cong \RHom_A(Ae,-):\D(A)\to\D(eAe)$ has a left adjoint $Ae\otimes^L_{eAe}- \cong \RHom_{A^\op}(eA,A)\otimes^L_{eAe}-: \D(eAe)\to\D(A)$ and a right adjoint $\RHom_{eAe}(eA,-): \D(eAe)\to\D(A)$. 
		$$\begin{tikzcd}[column sep=130]
			\D(A) \arrow[r,"{eA\otimes^L_A-\cong \RHom_A(Ae,-)}"] & \D(eAe) \arrow[l,shift right=6,"{Ae\otimes^L_{eAe}-\cong \RHom_{A^\op}(eA,A)\otimes^L_{eAe}-}"'] \arrow[l,shift left=6,"{\RHom_{eAe}(eA,-)}"']
		\end{tikzcd}$$
		The unit of the adjoint pair $(Ae\otimes^L_{eAe}-,eA\otimes^L_A-)$ is the natural isomorphism $\varepsilon:\Id_{\D(eAe)}\to eA\otimes^L_AAe\otimes^L_{eAe}-$, and the counit of the adjoint pair $(Ae\otimes^L_{eAe}-,eA\otimes^L_A-)$ is the natural transformation $\eta:Ae\otimes^L_{eAe}eA\otimes^L_A-\to \Id_{\D(A)}$. 		
		The right action morphism is the isomorphism $r^{eAe}_{Ae}:eAe\to \RHom_A(Ae,Ae)$ in $\D((eAe)^e)$. 
		Moreover, the following diagram is commutative:
		$$\begin{tikzcd}
			&\RHom_A(Ae,Ae) \\
			eAe \arrow[ur,bend left=10,near end,"r^{eAe}_{Ae}","\cong"'] \arrow[r,"\varepsilon^{eAe}_{eAe}","\cong"'] \arrow[rd,near end,bend right=10,"\cong"] &\RHom_A(Ae,Ae\otimes^L_{eAe}eAe) \arrow[u,"\cong"'] \arrow[d,"\cong"]\\
			&eA\otimes^L_AAe\otimes^L_{eAe}eAe
		\end{tikzcd}$$
		where the left lower morphism is given by the unit of the adjoint pair $Ae\otimes^L_{eAe}-: \D(eAe\otimes (eAe)^\op)\rightleftarrows\D(A\otimes (eAe)^\op):eA\otimes^L_A-$. We will not distinguish the left three isomorphisms.
		The evaluation morphism is the morphism $\ev^A_{Ae}:Ae\otimes^L_{eAe}\RHom_A(Ae,A)\to A$ in $\D(A^e)$. 
		Moreover, the following diagram is commutative:
		$$\begin{tikzcd}
			\RHom_{A^\op}(eA,A)\otimes^L_{eAe}eA \arrow[rd,"\ev^A_{eA}"] \arrow[d,"\cong"'] & \\
			Ae\otimes^L_{eAe}eA \arrow[r,""] & A\\
			Ae\otimes^L_{eAe}\RHom_A(Ae,A) \arrow[ru,"\ev^A_{Ae}=\eta^A_A"'] \arrow[u,"\cong"] &
		\end{tikzcd}$$
		where the middle right morphism is given by the counit of the adjoint pair $Ae\otimes^L_{eAe}-: \D(eAe\otimes A^\op)\rightleftarrows\D(A\otimes A^\op):eA\otimes^L_A-$. We will not distinguish the right three morphisms.
		Since $Ae$ is a projective left $A$-module, we have $\Cone(\ev^A_{Ae})\otimes^L_AZ\cong \Cone(\eta_Z)\cong\Cone(\ev^A_{eA})\otimes^L_AZ$ for all $Z\in\D(A)$, and $\Cone(r^{eAe}_{Ae})\otimes^L_{eAe}Y\cong\Cone(\varepsilon_Y)\cong 0$ for all $Y\in\D(eAe)$.
		
		The unit of the adjoint pair $(eA\otimes^L_A-,\RHom_{eAe}(eA,-))$ is the natural transformation $\varepsilon:\Id_{\D(A)}\to \RHom_{eAe}(eA,eA\otimes^L_A-)$, and the counit of the adjoint pair $(eA\otimes^L_A-,\RHom_{eAe}(eA,-))$ is the natural isomorphism $\eta:eA\otimes^L_A\RHom_{eAe}(eA,-)\to \Id_{\D(eAe)}$. The right action morphism is the morphism $r^A_{eA}:A\to \RHom_{eAe}(eA,eA)$ in $\D(A^e)$, and the evaluation morphism is the isomorphism $\ev^{eAe}_{eA}:eA\otimes^L_A\RHom_{eAe}(eA,eAe)\to eAe$ in $\D((eAe)^e)$. If $_{eAe}eA\in\per(eAe)$ then we have $\Cone(r^A_{eA})\otimes^L_AZ\cong \Cone(\varepsilon_Z)$ for all $Z\in\D(A)$, and $\Cone(\ev^{eAe}_{eA})\otimes^L_{eAe}Y\cong\Cone(\eta_Y)\cong 0$ for all $Y\in\D(eAe)$.		
}\end{example}

The following result is a characterization of a standard eventually homological isomorphism.

\begin{theorem} \label{Thm-EvHomIso-TensorPerfect-BimodComplex} 
	Let $A$ and $B$ be finite dimensional algebras. Then a right perfect complex $X$ of $A$-$B$-bimodules induces an eventually homological isomorphism $_AX\otimes^L_B-:\D^b(B) \to \D^b(A)$ if and only if the $B$-bimodule complex $\Cone(\ev^B_X)$ is $\otimes$-perfect, where the evaluation morphism $\ev^B_X: \RHom_{B^\op}(X,B) \otimes^L_AX\to B$ in $\D(B^e)$ is induced by the morphism $\ev^B_{{\bf p}X}: \Hom_{B^\op}({\bf p}X,B) \otimes_A{\bf p}X\to B$ in the homotopy category ${\cal H}(B^e)$ which is given by $\ev^B_{{\bf p}X}(f\otimes x):=f(x)$ for all $f\in\Hom_{B^\op}({\bf p}X,B)$ and $x\in{\bf p}X$. 	
\end{theorem}

\begin{proof} 
	Due to $X_B\in\per(B^\op)$, we have $\RHom_{B^\op}(X,B)\in\per(B)$. The triangle functor $\RHom_{B^\op}(X,B)\linebreak \otimes^L_A-$ has a right adjoint $\RHom_B(\RHom_{B^\op}(X,B),-)\cong X\otimes^L_B-$, i.e., we have an adjoint pair  
	$$\RHom_{B^\op}(X,B)\otimes^L_A- : \D(A) \rightleftarrows \D(B) : \RHom_B(\RHom_{B^\op}(X,B),-) \cong X\otimes^L_B-.$$
	
	Denote by $(\bigstar)$ the statement that the triangle functor $F:=\ _AX\otimes^L_B-:\D^b(B) \to \D^b(A)$ is an eventually homological isomorphism. By definition, $(\bigstar)$ holds if and only if the triangle functor $F$ induces an isomorphism $F_{Y,Z[i]}:\Hom_{\D^b(B)}(Y,Z[i]) \to \Hom_{\D^b(A)}(X\otimes^L_BY,X\otimes^L_BZ[i])$ for all $Y,Z\in\D^b(B)$ and $i\gg 0$. 
	
	By Lemma~\ref{Lem-F_XY-Unit} (1), we have the following commutative diagram
	$$\begin{tikzcd}
		\Hom_{\D^b(B)}(Y,Z[i]) \arrow[r,"{F_{Y,Z[i]}}"] \arrow[dr,"{\Hom_{\D(B)}(\eta_Y,Z[i])}"'] & \Hom_{\D^b(A)}(X\otimes^L_BY,X\otimes^L_BZ[i]) \\
		& \Hom_{\D(B)}(\RHom_{B^\op}(X,B)\otimes^L_AX\otimes^L_BY,Z[i]) \arrow[u,"{\cong}","{\alpha_{X\otimes^L_BY,Z[i]}}"'] 		
	\end{tikzcd}$$
	where $\eta$ is the counit of the adjoint pair $(\RHom_{B^\op}(X,B) \otimes^L_A-,X\otimes^L_B-)$. So $(\bigstar)$ holds if and only if the linear map $\Hom_{\D(B)}(\eta_Y,Z[i]) : \Hom_{\D(B)}(Y,Z[i]) \to \Hom_{\D(B)}(\RHom_{B^\op}(X,B)\otimes^L_AX\otimes^L_BY,Z[i])$ is an isomorphism for all $Y,Z\in\D^b(B)$ and $i\gg 0$.
		
	Applying the triangle functor $\RHom_B(-,Z)$ to the triangle 
	$$\RHom_{B^\op}(X,B)\otimes^L_AX\otimes^L_BY \xrightarrow{\eta_Y} Y \to\Cone(\eta_Y) \to$$
	in $\D(B)$, we get a triangle
	$$\RHom_B(\Cone(\eta_Y),Z) \to \RHom_B(Y,Z) \xrightarrow{\RHom_B(\eta_Y,Z)} \RHom_B(\RHom_{B^\op}(X,B)\otimes^L_AX\otimes^L_BY,Z) \to$$
	in $\D(k)$. Then we have a long exact sequence of $k$-vector spaces
	$$\begin{array}{rl}
		\cdots \to & \Hom_{\D(B)}(\Cone(\eta_Y),Z[i]) \to \Hom_{\D(B)}(Y,Z[i]) \xrightarrow{\Hom_{\D(B)}(\eta_Y,Z[i])} \\ [2mm]
		& \Hom_{\D(B)}(\RHom_{B^\op}(X,B)\otimes^L_AX\otimes^L_BY,Z[i]) \to \cdots.
	\end{array}$$
	Thus $(\bigstar)$ holds if and only if $\Hom_{\D(B)}(\Cone(\eta_Y),Z[i])=0$ for all $Y,Z\in\D^b(B)$ and $i\gg 0$.
	
	Due to $\RHom_{B^\op}(X,B)\in\per(B)$, applying Lemma~\ref{Lem-Cone-Cone} to the adjoint pair  
	$(\RHom_{B^\op}(X,B)\otimes^L_A-,\RHom_B(\RHom_{B^\op}(X,B),-)),$
	we obtain $\Cone(\eta_Y) \cong \Cone(\eta^B_B)\otimes^L_BY$, where $\eta^B_B$ is the counit of the adjoint pair
	$\RHom_{B^\op}(X,B)\otimes^L_A- : \D(A\otimes B^\op) \rightleftarrows \D(B^e) : \RHom_B(\RHom_{B^\op}(X,B),-) \cong X\otimes^L_B-.$ 
	Moreover, the following diagram in $\D(B^e)$ is commutative,
	$$\begin{tikzcd}
		\RHom_{B^\op}(X,B)\otimes^L_AX \arrow[r,"{\ev^B_X}"] \arrow[d,"\cong","{\RHom_{B^\op}(X,B)\otimes^L_Ah_X}"'] & B \\
		\RHom_{B^\op}(X,B)\otimes^L_A\RHom_B(\RHom_{B^\op}(X,B),B) \arrow[ur,"\eta^B_B"'] &
	\end{tikzcd}$$  
	where the isomorphism $h_X:X\to \RHom_B(\RHom_{B^\op}(X,B),B)$ in $\D(A\otimes B^\op)$ is induced by the isomorphism $h_{{\bf p}X}:{\bf p}X\to \Hom_B(\Hom_{B^\op}({\bf p}X,B),B)$ in the homotopy category ${\cal H}(A\otimes B^\op)$ given by $h_{{\bf p}X}(x)(f):=(-1)^{|f||x|}f(x)$ for all $x\in {\bf p}X$ and $f\in \Hom_{B^\op}({\bf p}X,B)$, 
	since the following diagram in ${\cal H}(B^e)$ is commutative.
	$$\begin{tikzcd}
		\Hom_{B^\op}({\bf p}X,B)\otimes_A{\bf p}X \arrow[r,"{\ev^B_{{\bf p}X}}"] \arrow[d,"\cong","{\Hom_{B^\op}({\bf p}X,B)\otimes_Ah_{{\bf p}X}}"'] & B \\
		\Hom_{B^\op}({\bf p}X,B)\otimes_A\Hom_B(\Hom_{B^\op}({\bf p}X,B),B) \arrow[ur,"\eta^B_B"'] &
	\end{tikzcd}$$  
	Thus $(\bigstar)$ holds if and only if $\Hom_{\D(B)}(\Cone(\ev^B_X)\otimes^L_BY,Z[i])=0$ for all $Y,Z\in\D^b(B)$ and $i\gg 0$. Since 
	$\RHom_{B^\op}(X,B) \in \D^b(B\otimes A^\op)$ and $X\in \D^b(A\otimes B^\op)$, we have $\RHom_{B^\op}(X,B)\otimes^L_AX \in \D^-(\mod B^e)$. Then $\Cone(\ev^B_X) \in \D^-(\mod B^e)$, further $\Cone(\ev^B_X)\otimes^L_BY \in \D^-(\mod B)$, so $\RHom_B(\Cone(\ev^B_X)\otimes^L_BY,Z) \in\D^+(k)$. 
	Hence $(\bigstar)$ holds if and only if $\RHom_B(\Cone(\ev^B_X)\otimes^L_BY,Z) \in\D^b(k)$ for all $Y,Z\in\D^b(B)$, if and only if $\Cone(\ev^B_X)\otimes^L_BY \in \per(B)$ for all $Y\in\D^b(B)$ by Lemma~\ref{Lem-LeftModComplex-Perfect}, if and only if the $B$-bimodule complex $\Cone(\ev^B_X)$ is $\otimes$-perfect by Proposition~\ref{Prop-BimodComplex-TensorPerfect}. 
\end{proof}

The following result is a characterization of a standard eventually homological isomorphism induced by a biperfect complex of bimodules.

\begin{theorem}	\label{Thm-EvHomIso-TensorPerfect-BimodComplex-RightAdjoint} 
	Let $A$ and $B$ be finite dimensional $k$-algebras, and $X$ a biperfect complex of $A$-$B$-bimodules. Then the following statements are equivalent:
	
	{\rm(1)} The triangle functor $_AX\otimes^L_B-:\D^b(B) \to \D^b(A)$ is an eventually homological isomorphism.
	
	{\rm(2)} The $B$-bimodule complex $\Cone(\ev^B_X)$ is $\otimes$-perfect, where the evaluation morphism $\ev^B_X: \linebreak \RHom_{B^\op}(X,B) \otimes^L_AX\to B$ in $\D(B^e)$ is induced by the morphism $\ev^B_{{\bf p}X}:\Hom_{B^\op}({\bf p}X,B) \otimes_A{\bf p}X\to B$ in ${\cal H}(B^e)$ which is given by $\ev^B_{{\bf p}X}(f\otimes x):=f(x)$ for all $f\in\Hom_{B^\op}({\bf p}X,B)$ and $x\in{\bf p}X$. 	
	
	{\rm(3)} The $B$-bimodule complex $\Cone(r^B_X)$ is $\otimes$-coperfect, where the right action morphism $r^B_X:B\to\RHom_A(X,X)$ in $\D(B^e)$ is induced by the morphism $r^B_{{\bf p}X}: B\to \Hom_A({\bf p}X,{\bf p}X)$ in ${\cal H}(B^e)$ which is given by $r^B_{{\bf p}X}(b)(x):=xb$ for all $b\in B$ and $x\in{\bf p}X$. 
\end{theorem}

\begin{proof} (1)$\Leftrightarrow$(2): It follows from Theorem~\ref{Thm-EvHomIso-TensorPerfect-BimodComplex}.
	
	(1)$\Leftrightarrow$(3): Denote by $(\bigstar)$ the statement that the triangle functor $F:=X\otimes^L_B-:\D^b(B)\to\D^b(A)$ is an eventually homological isomorphism. By definition, $(\bigstar)$ holds if and only if $F$ induces an isomorphism $F_{Y,Z[i]}:\Hom_{\D(B)}(Y,Z[i])\to\Hom_{\D(A)}(X\otimes^L_BY,X\otimes^L_BZ[i])$ for all $Y,Z\in \D^b(B)$ and $i\gg 0$.
	
	Applying Lemma~\ref{Lem-F_XY-Unit} (2) to the adjoint pair $(X\otimes^L_B-,\RHom_A(X,-))$, we have the following commutative diagram
	$$\begin{tikzcd}
		\Hom_{\D(B)}(Y,Z[i]) \arrow[r,"{F_{Y,Z[i]}}"] \arrow[dr,"{\Hom_{\D(B)}(Y,\varepsilon_{Z[i]})}"'] & \Hom_{\D(A)}(X\otimes^L_BY,X\otimes^L_BZ[i]) \\
		& \Hom_{\D(B)}(Y,\RHom_A(X,X\otimes^L_BZ[i])) \arrow[u,"{\beta_{Y,X\otimes^L_BZ[i]}^{-1}}"',"\cong"]
	\end{tikzcd}$$
	where $\varepsilon$ is the unit of the adjoint pair $(X\otimes^L_B-,\RHom_A(X,-))$. Note that $\varepsilon_{Z[i]}=\varepsilon_Z[i]$.
	Thus $(\bigstar)$ holds if and only if the linear map $\Hom_{\D(B)}(Y,\varepsilon_Z[i])$ is an isomorphism for all $Y,Z\in \D^b(B)$ and $i\gg 0$.
	
	Applying the triangle functor $\RHom_B(Y,-)$ to the triangle 
	$$Z \xrightarrow{\varepsilon_Z} \RHom_A(X,X\otimes^L_BZ) \to \Cone(\varepsilon_Z) \to$$
	in $\D(B)$, we obtain a triangle 
	$$\RHom_B(Y,Z) \xrightarrow{\RHom_B(Y,\varepsilon_Z)} \RHom_B(Y,\RHom_A(X,X\otimes^L_BZ)) \to \RHom_B(Y,\Cone(\varepsilon_Z)) \to$$
	in $\D(k)$. Then we have a long exact sequence of $k$-vector spaces
	$$\begin{array}{ll}
		\cdots & \to \Hom_{\D(B)}(Y,Z[i]) \xrightarrow{\Hom_{\D(B)}(Y,\varepsilon_Z[i])} \Hom_{\D(B)}(Y,\RHom_A(X,X\otimes^L_BZ[i])) \\ [2mm]
		& \to \Hom_{\D(B)}(Y,\Cone(\varepsilon_Z)[i]) \to\cdots.
	\end{array}$$
	Thus $(\bigstar)$ holds if and only if $\Hom_{\D(B)}(Y,\Cone(\varepsilon_Z)[i])=0$ for all $Y,Z\in \D^b(B)$ and $i\gg 0$.
	Due to $_AX\in\per(A)$ and $X_B\in\per(B^\op)$, we have $\Cone(\varepsilon_Z)\in\D^b(B)$, and further $\RHom_B(Y,\Cone(\varepsilon_Z))\in\D^+(\mod k)$. Hence $(\bigstar)$ holds if and only if $\RHom_B(Y,\Cone(\varepsilon_Z))\in\D^b(k)$ for all $Y,Z\in \D^b(B)$.
	
	Thanks to $_AX\in\per(A)$, we can apply Lemma~\ref{Lem-Cone-Cone} to the adjoint pair $(X\otimes^L_B-,\RHom_A(X,-))$, and obtain $\Cone(\varepsilon^B_B)\otimes^L_BZ\cong \Cone(\varepsilon_Z)$ where $\varepsilon^B$ is the unit of the adjoint pair $X\otimes^L_B-: \D(B\otimes B^\op) \rightleftarrows \D(A\otimes B^\op):\RHom_A(X,-)$. Thus $(\bigstar)$ holds if and only if $\RHom_B(Y,\Cone(\varepsilon^B_B)\otimes^L_BZ)\in\D^b(k)$ for all $Y,Z\in \D^b(B)$, if and only if the $B$-bimodule complex $\Cone(\varepsilon^B_B)$ is $\otimes$-coperfect by Proposition~\ref{Prop-BimodComplex-TensorCoper}, if and only if the $B$-bimodule complex $\Cone(r^B_X)$ is $\otimes$-coperfect by Remark~\ref{Rmk-UnitAct-CounitEv}.
\end{proof}

\subsection{Constructions}

The following result gives a construction of standard eventually homological isomorphisms by opposite algebras.

\begin{proposition} \label{Prop-EHI-OppositeAlg}
	Let $A$ and $B$ be finite dimensional algebras. Then a right perfect complex $X$ of $A$-$B$-bimodules induces an eventually homological isomorphism $X\otimes^L_B-:\D^b(B) \to \D^b(A)$ if and only if the left perfect complex $\RHom_{B^\op}(X,B)$ of $B$-$A$-bimodules induces an eventually homological isomorphism $\RHom_{B^\op}(X,B)\otimes^L_{B^\op}- = - \otimes^L_B\RHom_{B^\op}(X,B):\D^b(B^\op)\to\D^b(A^\op)$.
\end{proposition}

\begin{proof}
	Since $X$ is a right perfect complex of $A$-$B$-bimodules, $\RHom_{B^\op}(X,B)$ is a left perfect complex of $B$-$A$-bimodules, or equivalently, a right perfect complex of $A^\op$-$B^\op$-bimodules. Thus we have the following adjoint pair
	$$\begin{tikzcd}[column sep=220]
		\D(B^\op) \arrow[r,shift right=2,"{\RHom_{B^\op}(X,B)\otimes^L_{B^\op}-\ =\ -\otimes^L_B\RHom_{B^\op}(X,B)}"'] & \D(A^\op). \arrow[l,shift right=2,"{\RHom_B(\RHom_{B^\op}(X,B),B)\otimes^L_{A^\op}-\ \cong\ X\otimes^L_{A^\op}-\ =\ -\otimes^L_AX}"']
	\end{tikzcd}$$
	
	By Theorem~\ref{Thm-EvHomIso-TensorPerfect-BimodComplex}, the triangle functor $X\otimes^L_B-:\D^b(B) \to \D^b(A)$ is an eventually homological isomorphism if and only if the $B$-bimodule complex $\Cone(\ev^B_X:\RHom_{B^\op}(X,B) \otimes^L_AX\to B)$ is $\otimes$-perfect, and the triangle functor $\RHom_{B^\op}(X,B)\otimes^L_{B^\op}-: \D^b(B^\op) \to\D^b(A^\op)$ is an eventually homological isomorphism if and only if $B^\op$-bimodule complex $\Cone(\ev^{B^\op}_{\RHom_{B^\op}(X,B)}:\RHom_B(\RHom_{B^\op}(X,B),B)\otimes^L_{A^\op}\RHom_{B^\op}(X,B) \to B^\op)$ is $\otimes$-perfect. 
	
	The following commutative diagram   
	$$\begin{tikzcd}[column sep=60]
		\RHom_B(\RHom_{B^\op}(X,B),B)\otimes^L_{A^\op}\RHom_{B^\op}(X,B) \arrow[r,"{\ev^{B^\op}_{\RHom_{B^\op}(X,B)}}"] \arrow[d,"\cong"] & B^\op \arrow[d,"\cong"] \\
		\RHom_{B^\op}(X,B)\otimes^L_AX \arrow[r,"{\ev^B_X}"] & B
	\end{tikzcd}$$
	implies that the $B$-bimodule complex $\Cone(\ev^{B^\op}_{\RHom_{B^\op}(X,B)}) \cong \Cone(\ev^B_X)$. Thus the $B$-bimodule complex  $\Cone(\ev^{B^\op}_{\RHom_{B^\op}(X,B)})$ is $\otimes$-perfect if and only if $\Cone(\ev^B_X)$ is $\otimes$-perfect.
\end{proof}

The following result provides a construction of standard eventually homological isomorphisms by tensor products of algebras.

\begin{proposition} \label{Prop-EHI-TensorProdAlg}
	Let $A$ and $B$ be finite dimensional $k$-algebras, $C$ a finite dimensional $k$-algebra with $\gl C<\infty$, and $B/\rad B$ or $C/\rad C$ separable over $k$. Then a right perfect complex $X$ of $A$-$B$-bimodules induces an eventually homological isomorphism $X\otimes^L_B-: \D^b(B)\to\D^b(A)$ if and only if the right perfect complex $X\otimes C$ of $(A\otimes C)$-$(B\otimes C)$-bimodules induces an eventually homological isomorphism $(X\otimes C)\otimes^L_{B\otimes C}- \cong X\otimes^L_B-: \D^b(B\otimes C)\to\D^b(A\otimes C)$.
\end{proposition}

\begin{proof}
	Note that the assumption $\gl C<\infty$ implies $\ol{C}\otimes^L_C\ol{C}\in\D^b(k)$ where $\ol{C}:=C/\rad C$. And the assumption that $B/\rad B$ or $C/\rad C$ is separable over $k$ implies $\ol{B\otimes C}=\ol{B}\otimes\ol{C}$ by Lemma~\ref{Lem-TensorProdTop=TopTensorProd}. 
	
	By Theorem~\ref{Thm-EvHomIso-TensorPerfect-BimodComplex}, a right perfect complex $X$ of $A$-$B$-bimodules induces an eventually homological isomorphism $X\otimes^L_B-: \D^b(B)\to\D^b(A)$ if and only if the $B$-bimodule complex $\Cone(\ev^B_X:\RHom_{B^\op}(X,B)\otimes_AX\to B)$ is $\otimes$-perfect.
	
	The triangle functor $_AX\otimes^L_B-: \D^b(B\otimes C)\to\D^b(A\otimes C)$ is naturally isomorphic to the triangle functor $_{A\otimes C}(X\otimes C)\otimes^L_{B\otimes C}-: \D^b(B\otimes C)\to\D^b(A\otimes C)$. By Theorem~\ref{Thm-EvHomIso-TensorPerfect-BimodComplex}, the triangle functor $_{A\otimes C}(X\otimes C)\otimes^L_{B\otimes C}-: \D^b(B\otimes C)\to\D^b(A\otimes C)$ is an eventually homological isomorphism if and only if $\Cone(\ev^{B\otimes C}_{X\otimes C})$ is $\otimes$-perfect in $\D((B\otimes C)^e)$. 
	
	We have the following isomorphisms
	$$\begin{array}{ll}
		\RHom_{B^\op}(X,B)\otimes^L_AX\otimes C & \cong (\RHom_{B^\op}(X,B)\otimes C)\otimes^L_{A\otimes C}(X\otimes C) \\ [2mm]
		& \cong \RHom_{(B\otimes C)^\op}(X\otimes C,B\otimes C)\otimes^L_{A\otimes C}(X\otimes C)
	\end{array}$$
	in $\D((B\otimes C)^e)$ which give the following commutative diagram	
	$$\begin{tikzcd}[column sep=30]
		\RHom_{B^\op}(X,B)\otimes^L_AX\otimes C \arrow[r,"\ev^B_X\otimes C"] \arrow[d,"\cong"]  & B\otimes C  \\
		\RHom_{(B\otimes C)^\op}(X\otimes C,B\otimes C)\otimes^L_{A\otimes C}(X\otimes C) \arrow[ur,"\ev^{B\otimes C}_{X\otimes C}"'] &  
		\end{tikzcd}$$
	\noindent in $\D((B\otimes C)^e)$. Thus $\Cone(\ev^{B\otimes C}_{X\otimes C})\cong \Cone(\ev^B_X\otimes C) \cong \Cone(\ev^B_X)\otimes C$. Therefore, $\Cone(\ev^{B\otimes C}_{X\otimes C})$ is $\otimes$-perfect in $\D((B\otimes C)^e)$ if and only if $\Cone(\ev^B_X)\otimes C$ is $\otimes$-perfect in $\D((B\otimes C)^e)$, if and only if $(\ol{B}\otimes\ol{C})\otimes^L_{B\otimes C}(\Cone(\ev^B_X)\otimes C)\otimes^L_{B\otimes C}(\ol{B}\otimes\ol{C}) \cong (\ol{B}\otimes^L_B\Cone(\ev^B_X)\otimes^L_B\ol{B})\otimes(\ol{C}\otimes^L_C\ol{C})\in\D^b(k)$ by Proposition~\ref{Prop-BimodComplex-TensorPerfect}, if and only if $\ol{B}\otimes^L_B\Cone(\ev^B_X)\otimes^L_B\ol{B}\in\D^b(k)$ by K\"{u}nneth formula for complexes of $k$-vector spaces (Ref. \cite[Theorem 3.6.3]{Weibel94}) since $\ol{C}\otimes^L_C\ol{C}\in\D^b(k)$ is nonzero, if and only if $\Cone(\ev^B_X)$ is $\otimes$-perfect in $\D(B^e)$ by Proposition~\ref{Prop-BimodComplex-TensorPerfect}.
	
	Finally, the triangle functor $(X\otimes C)\otimes^L_{B\otimes C}-: \D^b(B\otimes C)\to\D^b(A\otimes C)$ is an eventually homological isomorphism if and only if the $(B\otimes C)$-bimodule complex $\Cone(\ev^{B\otimes C}_{X\otimes C})$ is $\otimes$-perfect, if and only if the $B$-bimodule complex $\Cone(\ev^B_X)$ is $\otimes$-perfect in $\D(B^e)$, if and only if the triangle functor $X\otimes^L_B-: \D^b(B)\to\D^b(A)$ is an eventually homological isomorphism. 	
\end{proof}

\begin{remark}{\rm
	In Proposition~\ref{Prop-EHI-TensorProdAlg}, the conditions that the algebra $C$ satisfies $\gl C  <\infty$ and the algebra $B/\rad B$ or $C/\rad C$ is separable over $k$ is necessary. Indeed, from the proof of Proposition~\ref{Prop-EHI-TensorProdAlg}, it is clear that the condition $\gl C<\infty$ is necessary even though the algebra $C/\rad C$ is separable over $k$. 	
	Moreover, the extension $k:=\mathbb{Z}_p(t) \subset K:= k[x]/(x^p-t)$ of fields induces an eventually homological isomorphism $_kK\otimes_K-: \D^b(K)\to \D^b(k)$ with a left adjoint $_KK\otimes_k-: \D^b(k)\to \D^b(K)$ by Theorem~\ref{Thm-EvHomIso-TensorPerfect-BimodComplex}, since the $K$-bimodule complex $\Cone(\ev^K_K: K\otimes_kK\to K)$ is $\otimes$-perfect by Proposition~\ref{Prop-TensorPerfect-Property} and Proposition~\ref{Prop-HomSmooth-TensorPerfect-BimodComplex}. In fact, $K$ is a $\otimes$-perfect $K$-bimodule complex due to $\gl K=0$, and $K\otimes_kK$ is a free $K$-bimodule, in particular, a perfect $K$-bimodule complex. However, the triangle functor $_kK\otimes_K-: \D^b(K\otimes_kK)\to \D^b(K)$ is not an eventually homological isomorphism since $K\otimes_kK\cong K[y]/(y^p)$ is a $K$-algebra of infinite global dimension,  $\Hom_{\D^b(K\otimes_kK)}(K,K[i])\ne 0$ for all $i\ge 1$, but $\Hom_{\D^b(K)}(K,K[i])= 0$ for all $i\ge 1$.
}\end{remark}

\subsection{Examples}

In this subsection, we observe the examples of standard eventually homological isomorphisms induced by homomorphisms of algebras and idempotents.

\medspace

\noindent{\bf EHI induced by homomorphisms of algebras.} A homomorphism $f:A\to B$ of algebras induces an adjoint pair 
$_BB\otimes^L_A-:\D(A)\rightleftarrows\D(B):\ _AB\otimes^L_B-$.

\begin{corollary} \label{Cor-EHI-AlgHom}
	Let $f:A\to B$ be a homomorphism of finite dimensional algebras. Then the following statements hold:
	
	{\rm (1)} The triangle functor $_AB\otimes^L_B-:\D^b(B) \to \D^b(A)$ is an eventually homological isomorphism if and only if the $B$-bimodule complex $\Cone(\ev^B_B:B\otimes^L_AB\to B)$ is $\otimes$-perfect.
	
	{\rm (2)} Under the assumption $\pd_{A^\op}B<\infty$, the triangle functor $_BB\otimes^L_A-:\D^b(A) \to \D^b(B)$ is an eventually homological isomorphism if and only if the $A$-bimodule complex $\Cone(f)$ is $\otimes$-coperfect.	 
\end{corollary}

\begin{proof}
	{\rm (1)} It follows from Theorem~\ref{Thm-EvHomIso-TensorPerfect-BimodComplex}.
	
	{\rm (2)} Due to $_BB\in\per(B)$, it follows from Theorem~\ref{Thm-EvHomIso-TensorPerfect-BimodComplex-RightAdjoint} that the triangle functor $_BB\otimes^L_A-:\D^b(A) \to \D^b(B)$ is an eventually homological isomorphism if and only if the $A$-bimodule complex $\Cone(r^A_B: A\to \RHom_B(B,B))\cong \Cone(f)$ is $\otimes$-coperfect.
	$$\begin{tikzcd}
		A \arrow[r,"r^A_B"] \arrow[d,equal] & \RHom_B(B,B) \arrow[r] \arrow[d,"\cong"] & \Cone(r^A_B) \arrow[r] \arrow[d,"\cong"] & A[1] \arrow[d,equal] \\
		A \arrow[r] & B \arrow[r,two heads] & \Cone(f) \arrow[r] & A[1]
	\end{tikzcd}$$
\end{proof}

A homomorphism $f:A\to B$ of algebras is called a {\it homological epimorphism of algebras} (\cite[Definition 4.5 and Theorem 4.4]{Geigle-Lenzing91}) if the multiplication map $B\otimes_AB \to B$ is an isomorphism and $\Tor^A_i(B,B)=0$ for all $i\ge 1$, or equivalently, the evaluation morphism $\ev^B_B: B\otimes^L_AB \to B$ is an isomorphism in $\D(B^e)$, or equivalently, the derived tensor product functor $_AB\otimes^L_B-:\D^b(B)\to\D^b(A)$ is fully faithful. Let $f:A\to B$ be a homological epimorphism of finite dimensional algebras. Then the fully faithful triangle functor $_AB\otimes^L_B-:\D^b(B)\to\D^b(A)$ is an eventually homological isomorphism. In this case, since the evaluation morphism $\ev^B_B: B\otimes^L_AB\to B$ is an isomorphism, the $B$-bimodule complex $\Cone(\ev^B_B)$ is zero in $\D(B^e)$, thus $\otimes$-perfect.

In particular, if $A$ is a finite dimensional algebra, $e$ is an idempotent of $A$, and $AeA$ is a {\it stratifying ideal} of $A$, i.e., the natural morphism $Ae\otimes^L_{eAe}eA\to AeA$ is an isomorphism (Ref. \cite[Definition 2.1.1]{Cline-Parshall-Scott96}), or equivalently, the canonical surjection $A\twoheadrightarrow A/AeA=: B$ is a homological epimorphism (Ref. \cite[Remark 2.1.2]{Cline-Parshall-Scott96}), then the derived tensor product functor $_AB\otimes^L_B-:\D^b(B)\to\D^b(A)$ is an eventually homological isomorphism.

\medspace

\noindent{\bf EHI induced by extensions of algebras.} If $B$ is a subalgebra of an algebra $A$ with identity $1_B=1_A$, i.e., the inclusion map $B\hookrightarrow A$ is a homomorphism of algebras, then  $B\subseteq A$ is called an {\it extension of algebras}.

\begin{corollary} \label{Cor-EHI-AlgExt}
	Let $B\subseteq A$ be an extension of finite dimensional algebras. Then the following statements hold:
	
	{\rm (1)} The triangle functor $_BA\otimes^L_A-:\D^b(A) \to \D^b(B)$ is an eventually homological isomorphism if and only if the $A$-bimodule complex $\Cone(\ev^A_A)$ is $\otimes$-perfect where the evaluation morphism $\ev^A_A:A\otimes^L_BA\to A$ is given by the counit of the adjoint pair $_AA\otimes^L_B-:\D(B\otimes A^\op) \rightleftarrows \D(A\otimes A^\op): \RHom_A(A-)\cong\ _BA\otimes^L_A-$.
	
	{\rm (2)} Under the assumption $\pd_{B^\op}A<\infty$, the triangle functor $_AA\otimes^L_B-:\D^b(B) \to \D^b(A)$ is an eventually homological isomorphism if and only if the $B$-bimodule complex $A/B$ is $\otimes$-coperfect.	 
\end{corollary}

Now we consider three special extensions of algebras.

An extension $B \subseteq A$ of algebras is {\it left} (resp. {\it right}) {\it bounded} (\cite[Definition 2.7]{Cibils-Lanzilotta-Marcos-Solotar22}) if it satisfies the following three conditions:

(1) $A/B$ is of finite projective dimension as a $B$-bimodule.

(2) $A/B$ is a left (resp. right) projective $B$-module.

(3) $A/B$ is {\it $B$-tensor nilpotent}, i.e., $(A/B)^{\otimes_Bi}=0$ for $i\gg 0$.

An extension $B \subseteq A$ of algebras is {\it bounded} (\cite[Definition 1.1]{Qin-Xu-Zhang-Zhou24}) if it satisfies the following three conditions:

(1) $A/B$ has finite projective dimension as a $B$-bimodule.

(2) $A/B$ is $B$-tensor nilpotent.

(3) $\Tor^B_i(A/B,(A/B)^{\otimes_Bj}) = 0$ for all $i,j \ge 1$.

Clearly, a left or right bounded extension must be a bounded extension. Conversely, the example \cite[Example 1.4]{Qin-Xu-Zhang-Zhou24} implies that a bounded extension need not be a left or right bounded extension.	

Let $B\subseteq A$ be an extension of algebras. An $A$-module $M$ is said to be {\it relatively $B$-projective} or {\it $(A, B)$-projective} \cite{Hochschild56} if it satisfies either of the following equivalent conditions:
(1) $M$ is isomorphic to a direct summand of the induced module $A \otimes_B V$ for some left $B$-module $V$;
(2) Every $A$-module epimorphism to $M$ splitting as a $B$-module homomorphism splits as an $A$-module homomorphism.
An exact sequence of $A$-modules
$\cdots \to M_{i+1} \xrightarrow{f_{i+1}} M_i \xrightarrow{f_i} M_{i-1} \to\cdots$
is said to be  {\it $(A, B)$-exact} if $\Ker f_i$ is a direct summand of $M_i$ as $B$-modules for all $i\in\mathbb{Z}$.
The {\it relative $B$-projective dimension} or {\it $(A, B)$-projective dimension} $\pd_{(A,B)}M$ of an $A$-module $M$ is the minimal non-negative integer $n$ such that there is an $(A, B)$-exact sequence $0\to P_n\to \cdots\to P_1\to P_0\to M\to 0$ where $P_0,\cdots,P_n$ are $(A, B)$-projective $A$-modules. If such an exact sequence does not exist, the $(A, B)$-projective dimension of $M$ is infinite.

An extension $B \subseteq A$ of algebras is said to be {\it strongly proj-bounded} (\cite[Definition 4.1]{Iusenko-MacQuarrie25}) if it satisfies
the following four conditions:

(1) $A/B$ is of finite projective dimension as a $B$-bimodule;

(2) $A/B$ is projective as either a left or a right $B$-module;

(3) There exists an integer $p \ge 1$ such that the tensor power $(A/B)^{\otimes_Bi}$ is projective as a $B^e$-module for all $i \ge p$;

(4) $A$ has finite $(A^e, B^e)$-projective dimension.

Clearly, a left or right bounded extensions must be a strongly proj-bounded extension (\cite[Example 4.2]{Iusenko-MacQuarrie25}). The examples in \cite[4.2]{Iusenko-MacQuarrie25} implies that strongly proj-bounded extensions need not be a left or right bounded extension. The example \cite[Example 1.4]{Qin-Xu-Zhang-Zhou24} implies that a bounded extension need not be a strongly proj-bounded extension.

\begin{corollary} \label{Cor-BoundExt-EHI} {\rm(cf. \cite[Theorem 2.10]{Qin-Xu-Zhang-Zhou24})}
	Let $B \subseteq A$ be a bounded or strongly proj-bounded extension of finite dimensional algebras. Then the triangle functor $_BA\otimes^L_A-:\D^b(A) \to \D^b(B)$ is an eventually homological isomorphism.
\end{corollary}

\begin{proof}
	If $B \subseteq A$ is a bounded extension then by \cite[Lemma 2.3 (1)]{Qin-Xu-Zhang-Zhou24} we have $A\otimes^L_BA\cong A\otimes_BA$. Moreover, $A$ admits an $(A^e,B^e)$-projective resolution $0\to A\otimes_B(A/B)^{\otimes_Bp-1}\otimes_BA \to\cdots\to A\otimes_B(A/B)\otimes_BA \to A\otimes_BA\to A\to 0$ for some non-negative integer $p$ (Ref. \cite[The proof of Theorem 2.7]{Qin-Xu-Zhang-Zhou24}). Thus the $A$-bimodule complex $\Cone(\ev^A_A:A\otimes^L_BA\to A)$ is perfect, hence $\otimes$-perfect by Proposition~\ref{Prop-HomSmooth-TensorPerfect-BimodComplex}. 
	
	If $B \subseteq A$ is a strongly proj-bounded extension then $A\otimes^L_BA\cong A\otimes_BA$. 
	By \cite[Proposition 2.1]{Cibils-Lanzilotta-Marcos-Solotar22}, $A$ admits an $(A^e,B^e)$-projective resolution
	$$\cdots \xrightarrow{b} A\otimes_B(A/B)^{\otimes_Bi}\otimes_BA \xrightarrow{b} \cdots \xrightarrow{b} A\otimes_BA/B\otimes_BA \xrightarrow{b} A\otimes_BA \xrightarrow{b} A\to 0.$$
	Due to $\pd_{(A^e,B^e)}A<\infty$, by \cite[Lemma 3.1]{Iusenko-MacQuarrie25}, there is a direct summand $K$ of $A$-bimodule $A\otimes_B(A/B)^{\otimes_Bd}\otimes_BA$ with $d:=\pd_{(A^e,B^e)}A$, such that the following sequence is exact.
	$$0\to K\to A\otimes_B(A/B)^{\otimes_Bd-1}\otimes_BA \xrightarrow{b} \cdots \xrightarrow{b} A\otimes_BA/B\otimes_BA \xrightarrow{b} A\otimes_BA \xrightarrow{b} A\to 0$$
	Since $A/B$ is of finite projective dimension as a $B^e$-module, we have $A\otimes_B(A/B)^{\otimes_Bi}\otimes_BA\in\per(A^e)$ for all $i\ge 1$.
	Then the $A$-bimodule complex $0\to K\to A\otimes_B(A/B)^{\otimes_Bd-1}\otimes_BA \xrightarrow{b} \cdots \xrightarrow{b} A\otimes_BA/B\otimes_BA\to 0$ is perfect. Thus the $A$-bimodule complex $\Cone(\ev^A_A:A\otimes^L_BA\to A)$ is perfect, hence $\otimes$-perfect by Proposition~\ref{Prop-HomSmooth-TensorPerfect-BimodComplex}.	
	
	In both cases, it follows from Corollary~\ref{Cor-EHI-AlgExt} (1) that the triangle functor $_BA\otimes^L_A-:\D^b(A) \to \D^b(B)$ is an eventually homological isomorphism.
\end{proof}

\medspace

\noindent{\bf EHI induced by idempotents of algebras.} The following result characterizes the eventually homological isomorphisms induced by idempotents of algebras.

\begin{proposition} \label{Prop-Exist-EvHomIso-Idemp} 
	Let $A$ be a finite dimensional algebra, and $e$ an idempotent in $A$. Then the following statements are equivalent:
	
	{\rm (1)} The triangle functor $eA\otimes^L_A-: \D^b(A)\to \D^b(eAe)$ is an eventually homological isomorphism.
	
	{\rm (2)} $\pd_A(\frac{A/AeA}{\rad(A/AeA)}) < \infty$ and $\pd_{(eAe)^\op} Ae < \infty$.
	
	{\rm (3)} $\id_A(\frac{A/AeA}{\rad(A/AeA)}) < \infty$ and $\pd_{eAe}eA < \infty$.
\end{proposition}

\begin{proof} 
	
	The idempotent $e$ induces a recollement (Ref. \cite[Proposition 2.16]{Jin-Yang-Zhou23})
	$$\begin{tikzcd}[column sep=80]
		\D(Q) \arrow[r,"{_AQ\otimes^L_Q-}"] & \D(A) \arrow[r,"{_{eAe}eA\otimes^L_A-}"] \arrow[l,shift left=6,"{\RHom_A(Q,-)}"'] \arrow[l,shift right=6,"_QQ\otimes^L_A-"'] & \D(eAe) \arrow[l,shift left=6,"{\RHom_{eAe}(eA,-)}"'] \arrow[l,shift right=6,"{_AAe\otimes^L_{eAe}-}"']
	\end{tikzcd}$$
	where $Q$ is a cohomology non-positive dg algebra. 	
	We have an adjoint pair $Ae\otimes^L_{eAe}- : \D(eAe) \rightleftarrows \D(A) : eA\otimes^L_A-$. The counit of the adjoint pair is the natural transformation $\eta : Ae\otimes^L_{eAe}eA\otimes^L_A- \to A\otimes^L_A-$. We have an adjoint pair $Ae\otimes^L_{eAe}- : \D(eAe\otimes A^\op) \rightleftarrows \D(A\otimes A^\op) : eA\otimes^L_A-$. The counit of the adjoint pair is the natural transformation $\eta^A : Ae\otimes^L_{eAe}eA\otimes^L_A- \to A\otimes^L_A-$. Since $_AAe$ is projective, by Lemma~\ref{Lem-Cone-Cone}, we have $\Cone(\eta_X)\cong \Cone(\eta^A_A)\otimes^L_AX$ for all $X\in\D(A)$. For simplicity, we denote $\Cone(\eta^A_A)$ by $C$.
	
	\medspace
	
	(1)$\Rightarrow$(2): Acting the triangle functor $-\otimes^L_AAe/\rad Ae$ to the triangle $$Ae\otimes^L_{eAe}eA \to A \to C \to$$ in $\D(A^e)$, we obtain a triangle
	$$Ae\otimes^L_{eAe}eAe/\rad eAe \to Ae/\rad Ae \to C\otimes^L_AAe/\rad Ae \to$$ in $\D(A)$. 
	By Theorem~\ref{Thm-EvHomIso-TensorPerfect-BimodComplex}, the $A$-bimodule complex $C$ is $\otimes$-perfect. Then we have $C\otimes^L_AAe/\rad Ae\in\per(A)$. Thus $Ae\otimes^L_{eAe}eAe/\rad eAe\in\D^b(A)$. By Lemma~\ref{Lem-RightModComplex-Perfect}, we get $Ae_{eAe}\in\per((eAe)^\op)$. So $\pd_{(eAe)^\op} Ae < \infty$.  
	
	Acting the triangle functor $-\otimes^L_A\frac{A/AeA}{\rad(A/AeA)}$ to the triangle $$Ae\otimes^L_{eAe}eA \to A \to C \to$$ in $\D(A^e)$, we obtain a triangle 
	$$Ae\otimes^L_{eAe}eA\otimes^L_A\frac{A/AeA}{\rad(A/AeA)} \to \frac{A/AeA}{\rad(A/AeA)} \to C\otimes^L_A\frac{A/AeA}{\rad(A/AeA)} \to$$
	in $\D(A)$.	
	Due to $eAeA=eA$, we have $eA\otimes^L_AA/AeA=0$. Then $eA\otimes^L_A\frac{A/AeA}{\rad(A/AeA)}=0$, and further $Ae\otimes^L_{eAe}eA\otimes^L_A\frac{A/AeA}{\rad(A/AeA)} = 0$. 	
	Since the $A$-bimodule complex $C$ is $\otimes$-perfect, we have $\frac{A/AeA}{\rad(A/AeA)} \cong C\otimes^L_A\frac{A/AeA}{\rad(A/AeA)}\in\per(A)$. So $\pd_A\frac{A/AeA}{\rad(A/AeA)}<\infty$.
	
	\medspace
	
	(2)$\Rightarrow$(1): Due to $\pd_{(eAe)^\op} Ae < \infty$, the recollement above extends one step upwards. Moreover, $Q$ is a cohomology non-positive finite dimensional dg algebra. By \cite[Lemma 2.18]{Jin-Yang-Zhou23}, it can restrict to a left recollement:
	$$\begin{tikzcd}[column sep=55]
		\D_\fd(Q) \arrow[r,"{_AQ\otimes^L_Q-}"] & \D^b(A) \arrow[r,"{_{eAe}eA\otimes^L_A-}"] \arrow[l,shift right=6,"_QQ\otimes^L_A-"'] & \D^b(eAe) \arrow[l,shift right=6,"{_AAe\otimes^L_{eAe}-}"']
	\end{tikzcd}$$
	where $\D_\fd(Q)$ is the full subcategory of $\D(Q)$ consisting of dg $Q$-modules $X$ of finite dimensional total cohomology. Then we have $_QQ\otimes^L_A\D^b(A) = \D_\fd(Q)$ and $_AQ\otimes^L_Q\D_\fd(Q) = \Ker(eA\otimes^L_A-) = \D^b(A)_{\mod A/AeA}$, the full subcategory of $\D^b(A)$ consisting of $A$-module complexes with cohomologies in $\mod A/AeA$.
	Thus $\pd_A\frac{A/AeA}{\rad (A/AeA)}<\infty$ if and only if $\frac{A/AeA}{\rad (A/AeA)}\in\per(A)$,  if and only if $\mod A/AeA\subseteq \per(A)$, if and only if $_AQ\otimes^L_A\D^b(A) =\ _AQ\otimes^L_QQ\otimes^L_A\D^b(A) \subseteq\ _AQ\otimes^L_Q\D_\fd(Q) = \D^b(A)_{\mod A/AeA}\subseteq \per(A).$
	From the triangles $$Ae\otimes^L_{eAe}eA\otimes^L_AX \xrightarrow{\eta^A_A\otimes^L_AX} A\otimes^L_AX \to \Cone(\eta^A_A)\otimes^L_AX \to$$ and $$Ae\otimes^L_{eAe}eA\otimes^L_AX \xrightarrow{\eta_X} A\otimes^L_AX \to Q\otimes^L_AX \to$$ we obtain $Q\otimes^L_AX \cong C\otimes^L_AX$ for all $X\in\D(A)$. So $C\otimes^L_A\D^b(A) \subseteq \per(A).$ By Proposition~\ref{Prop-BimodComplex-TensorPerfect}, the $A$-bimodule complex $C$ is $\otimes$-perfect. Thus the triangle functor $eA\otimes^L_A-: \D^b(A)\to \D^b(eAe)$ is an eventually homological isomorphism by Theorem~\ref{Thm-EvHomIso-TensorPerfect-BimodComplex}.
	
	\medspace
   
    (1)$\Leftrightarrow$(3): The triangle functor $eA\otimes^L_A-: \D^b(A)\to \D^b(eAe)$ is an eventually homological isomorphism if and only if the triangle functor $\RHom_{A^\op}(eA,A)\otimes^L_{A^\op}- \cong e^\op A^\op \otimes^L_{A^\op} -:\D^b(A^\op)\to\D^b((eAe)^\op)$ is an eventually homological isomorphism by Proposition~\ref{Prop-EHI-OppositeAlg}, if and only if $\pd_{A^\op}(\frac{A^\op/A^\op e^\op A^\op}{\rad(A^\op/A^\op e^\op A^\op)})<\infty$ and $\pd_{(e^\op A^\op e^\op)^\op}A^\op e^\op<\infty$ by the fact (1)$\Leftrightarrow$(2), if and only if $\id_A(\frac{A/AeA}{\rad(A/AeA)})<\infty$ and $\pd_{eAe}eA<\infty$ due to the equations $\id_A(\frac{A/AeA}{\rad(A/AeA)}) = \pd_{A^\op}(\frac{A/AeA}{\rad(A/AeA)})   = \pd_{A^\op}(\frac{A^\op/A^\op e^\op A^\op}{\rad(A^\op/A^\op e^\op A^\op)})$ and $_{eAe}eA_A=\ _{A^\op}A^\op e^\op_{e^\op A^\op e^\op}$ which are induced by the $A$-bimodule isomorphism $A\xrightarrow{\cong} A^\op, a\mapsto a^\op$.
\end{proof}

\begin{corollary} {\rm (cf. \cite[Main theorem (i)]{Psaroudakis-Skartsaeterhagen-Solberg14})}  
	Let $A$ be a finite dimensional algebra, and $e$ an idempotent in $A$. Then the following statements are equivalent:
	
	{\rm (1)} The exact functor $eA\otimes_A-: \mod A\to \mod eAe$ is an eventually homological isomorphism in the sense of Definition~\ref{Def-EHI-ModCat}.
	
	{\rm (2)} $\pd_A(\frac{A/AeA}{\rad(A/AeA)}) < \infty$ and $\pd_{(eAe)^\op} Ae < \infty$.
	
	{\rm (3)} $\id_A(\frac{A/AeA}{\rad(A/AeA)}) < \infty$ and $\pd_{eAe}eA < \infty$.
\end{corollary}

\begin{proof}
	By Proposition~\ref{Prop-EHI-ModCat-DerCat}, the exact functor $eA\otimes_A-: \mod A\to \mod eAe$ is an eventually homological isomorphism if and only if so is the triangle functor $eA\otimes_A-: \D^b(A)\to \D^b(eAe)$. Then the corollary follows from Proposition~\ref{Prop-Exist-EvHomIso-Idemp}.
\end{proof}

\subsection{Application to reduce Gorenstein symmetry conjecture}

In this subsection, we apply essentially surjective eventually homological isomorphisms to transfer Gorensteinness and reduce Gorenstein symmetry conjecture. In the last section, we will give more applications. 

Recall that a ring is said to be {\it Gorenstein} if it is two-sided Noetherian and has finite left and right injective dimension as a module over itself \cite{Iwanaga79}. 

\medspace

\noindent{\bf Gorenstein symmetry conjecture.} (\cite[Conjecture (13)]{Auslander-Reiten-Smalo95}) An Artin algebra has finite left injective dimension if and only if it has finite right injective dimension as a module over itself.

Gorenstein symmetry conjecture can be reduced by the recollements of derived categories \cite{Qin-Han16} and essentially surjective eventually homological isomorphisms on module categories \cite{Qin-Shen24}. More generally, the following result implies that it also can be reduced by essentially surjective eventually homological isomorphisms between derived categories.

\begin{theorem} \label{Thm-EHI-GorSymConj} 
	Let $A$ and $B$ be finite dimensional algebras, and $F: \D^b(B)\to\D^b(A)$ an essentially surjective eventually homological isomorphism. Then the following statements hold: 
	
	{\rm(1)} $A$ is Gorenstein if and only if so is $B$.
	
	{\rm(2)} $A$ satisfies the Gorenstein symmetry conjecture if and only if $B$ does.
\end{theorem}

\begin{proof} 
	For a triangulated subcategory $\T$ of $\D^b(B)$, denote by $F(\T)$ the essential image of $\T$ under $F$, i.e., the full triangulated subcategory of $\D^b(A)$ consisting of all objects isomorphic to $F(T)$ for some $T\in\T$.
	Since $F$ is an eventually homological isomorphism, we have linear isomorphisms $\Hom_{\D^b(B)}(Y,Z[i])\cong \Hom_{\D^b(A)}(F(Y),F(Z)[i])$ for all $Y,Z\in \D^b(B)$ and $i\gg 0$. 
	
	{\it Claim 1.} $F(K^b(\proj B)) = K^b(\proj A)$. 
	
	For any $Y\in K^b(\proj B)$ and $Z'\in \D^b(A)$, since $F$ is essentially surjective, there exists an object $Z\in\D^b(B)$ such that $F(Z)\cong Z'$. Then we have $\Hom_{\D^b(A)}(F(Y),Z'[i])\cong \Hom_{\D^b(B)}(Y,Z[i])=0$ for all $i\gg 0$. Thus 	
	$F(Y)\in K^b(\proj A)$. Hence $F(K^b(\proj B)) \subseteq K^b(\proj A)$.  
	
	For any $Y'\in K^b(\proj A)$, since $F$ is essentially surjective, there exists an object $Y\in\D^b(B)$ such that $F(Y)\cong Y'$. Then we have $\Hom_{\D^b(B)}(Y,Z[i])\cong \Hom_{\D^b(A)}(Y',F(Z)[i])=0$ for all $Z\in \D^b(B)$ and $i\gg 0$. Thus $Y\in K^b(\proj B)$. Hence $K^b(\proj A) \subseteq F(K^b(\proj B))$.	
	
	{\it Claim 2.} $F(K^b(\inj B)) = K^b(\inj A)$.
	
	For any $Z\in K^b(\inj B)$ and $Y'\in \D^b(A)$, since $F$ is essentially surjective, there exists an object $Y\in\D^b(B)$ such that $F(Y)\cong Y'$. Then we have $\Hom_{\D^b(A)}(Y',F(Z)[i])\cong \Hom_{\D^b(B)}(Y,Z[i])=0$ for all $i\gg 0$. Thus 	
	$F(Z)\in K^b(\inj A)$. Hence $F(K^b(\inj B)) \subseteq K^b(\inj A)$.
	
	For any $Z'\in K^b(\inj A)$, since $F$ is essentially surjective, there exists an object $Z\in\D^b(B)$ such that $F(Z)\cong Z'$. Then we have $\Hom_{\D^b(B)}(Y,Z[i])\cong \Hom_{\D^b(A)}(F(Y),Z'[i])=0$ for all $Y\in \D^b(B)$ and $i\gg 0$. Thus $Z\in K^b(\inj B)$. Hence $K^b(\inj A) \subseteq F(K^b(\inj B))$.
	
	{\it Claim 3.} $Y\in K^b(\proj B) \Leftrightarrow F(Y)\in K^b(\proj A)$.
	
	If $Y\in K^b(\proj B)$ then $F(Y)\in K^b(\proj A)$ by Claim 1.
	Conversely, if $F(Y)\in K^b(\proj A)$ then we have $\Hom_{\D^b(B)}(Y,Z[i])\cong \Hom_{\D^b(A)}(F(Y),F(Z)[i])  =0$ for all $Z\in \D^b(B)$ and $i\gg 0$. Thus $Y\in K^b(\proj B)$.
	
	{\it Claim 4.} $Z\in K^b(\inj B) \Leftrightarrow F(Z)\in K^b(\inj A)$.
	
	If $Z\in K^b(\inj B)$ then $F(Z)\in K^b(\inj A)$ by Claim 2.
	Conversely, if $F(Z)\in K^b(\inj A)$ then we have $\Hom_{\D^b(B)}(Y,Z[i])\cong \Hom_{\D^b(A)}(F(Y),F(Z)[i])=0$ for all $Y\in \D^b(B)$ and $i\gg 0$. Thus $Z\in K^b(\inj B)$.
	
	{\it Claim 5.} 	$A$ is Gorenstein if and only if $K^b(\proj A) = K^b(\inj A)$. 
	
	Clearly, $\id_AA<\infty$ if and only if $K^b(\proj A) \subseteq K^b(\inj A)$.	
	Moreover, $\id_{A^\op}A<\infty$ if and only if $\pd_AA^*<\infty$, if and only if $K^b(\inj A) \subseteq K^b(\proj A)$.
	
	(1) $A$ is Gorenstein if and only if $K^b(\proj A) = K^b(\inj A)$ by Claim 5, if and only if $F(K^b(\proj B)) = F(K^b(\inj B))$ by Claims 1 and 2, if and only if $K^b(\proj B) = K^b(\inj B)$ by Claims 3 and 4, if and only if $B$ is Gorenstein by Claim 5. 
	
	(2) $A$ satisfies the Gorenstein symmetry conjecture if and only if ``$K^b(\proj A)   \subseteq K^b(\inj A) \Leftrightarrow K^b(\inj A) \linebreak \subseteq K^b(\proj A)$'' by the proof of Claim 5, if and only if ``$F(K^b(\proj B)) \subseteq F(K^b(\inj B)) \Leftrightarrow F(K^b(\inj B)) \subseteq F(K^b(\proj B))$'' by Claims 1 and 2, if and only if ``$K^b(\proj B) \subseteq K^b(\inj B) \Leftrightarrow K^b(\inj B) \subseteq K^b(\proj B)$'' by Claims 3 and 4, if and only if $B$ satisfies the Gorenstein symmetry conjecture by the proof of Claim 5.
\end{proof}

\begin{remark} \label{Rem-ESEHI-Gorenstein-General} {\rm
	(i) The essentially surjective eventually homological isomorphism in Theorem~\ref{Thm-EHI-GorSymConj} needs not to be standard, i.e., a derived tensor product functor. 
	
	(ii) Theorem~\ref{Thm-EHI-GorSymConj} still holds even if the linear isomorphisms in the definition of eventually homological isomorphism are not induced by the functor, i.e., we use the definition of eventually homological isomorphism given by Qin in \cite{Qin20}.
	
	(iii) Every derived equivalence $F:\D^b(B)\to \D^b(A)$ is an essentially surjective eventually homological isomorphism.
}\end{remark}

\section{Singular equivalences of $n$-adjoint type}

In this section, we introduce singular equivalences of $n$-adjoint type which can be viewed as standard singular equivalences of different levels. For $n\ge 2$, we give their characterizations in terms of $\otimes$-(co)perfect bimodule complexes, constructions by opposite algebras and tensor product algebras, examples induced by homomorphisms of algebras and idempotents, and applications to reduce homological conjectures and transfer homological properties.

\subsection{Standard triangle functors between singularity categories}

The following result is a characterization of a {\it standard} triangle functor, i.e., a derived tensor product functor given by a bimodule complex, between singularity categories, which improves \cite[Proposition 3.3]{Oppermann-Psaroudakis-Stai19}.

\begin{lemma} \label{Lem-Exist-TriFunSingCat-FD-Complex} 
	Let $A$ and $B$ be finite dimensional algebras. Then a complex $X$ of $A$-$B$-bimodules induces a triangle functor $_AX\otimes^L_B-: \sg(B)\to \sg(A)$ if and only if $X$ is biperfect.	
\end{lemma}

\begin{proof}
	A complex $X$ of $A$-$B$-bimodules always induces a triangle functor $_AX\otimes^L_B-: \D(B)\to \D(A)$ which induces a triangle functor $_AX\otimes^L_B-: \sg(B)\to \sg(A)$ if and only if it can be restricted to triangle functors $_AX\otimes^L_B-: \D^b(B)\to \D^b(A)$ and $_AX\otimes^L_B-: \per(B)\to \per(A)$, if and only if $X$ is biperfect by Lemma~\ref{Lem-RightModComplex-Perfect} and Lemma~\ref{Lem-LeftModComplex-Perfect} respectively.
\end{proof}

The following result is a characterization of a standard triangle functor on singularity categories with a canonical left adjoint.

\begin{lemma} \label{Lem-SingCatTriFunLeftAdj-FD-Complex} 
	Let $A$ and $B$ be finite dimensional algebras. Then a complex $X$ of $A$-$B$-bimodules induces a triangle functor $_AX\otimes^L_B-: \sg(B)\to \sg(A)$ with a left adjoint $\RHom_{B^\op}(X,B)\otimes^L_A-:\sg(A)\to\sg(B)$ if and only if both $X\in\D(A\otimes B^\op)$ and $\RHom_{B^\op}(X,B)\in\D(B\otimes A^\op)$ are biperfect.
\end{lemma}

\begin{proof}
	It follows from Lemma~\ref{Lem-Exist-TriFunSingCat-FD-Complex} that the triangle functors $\RHom_{B^\op}(X,B)  \otimes^L_A-:\sg(A)\to\sg(B)$ and $_AX\otimes^L_B-: \sg(B)\to \sg(A)$ are well-defined if and only if both $X\in\D(A\otimes B^\op)$ and $\RHom_{B^\op}(X,B)\in\D(B\otimes A^\op)$ are biperfect. In this case, since the triangle functors $\RHom_{B^\op}(X,B)\otimes^L_A-:\D^b(A)\to\D^b(B)$ and $_AX\otimes^L_B-: \D^b(B)\to \D^b(A)$ form an adjoint pair, the triangle functors $\RHom_{B^\op}(X,B)\otimes^L_A-:\sg(A)\to\sg(B)$ and $_AX\otimes^L_B-: \sg(B)\to \sg(A)$ also form an adjoint pair by \cite[Lemma 1.2]{Orlov04}. 
\end{proof}

The following is a characterization of a fully faithful standard triangle functor on singularity categories with a canonical left adjoint.

\begin{lemma} \label{Lem-SingCatFFTriFunLeftAdj-FD-Complex} 
	Let $A$ and $B$ be finite dimensional algebras. Then a complex $X$ of $A$-$B$-bimodules induces a fully faithful triangle functor $_AX\otimes^L_B-: \sg(B)\to \sg(A)$ with a left adjoint $\RHom_{B^\op}(X,B)\otimes^L_A-:\sg(A)\to\sg(B)$ if and only if $X$ and $\RHom_{B^\op}(X,B)$ are biperfect, and the $B$-bimodule complex $\Cone(\ev^B_X:\RHom_{B^\op}(X,B)\otimes^L_AX\to B)$ is $\otimes$-perfect.	
\end{lemma}

\begin{proof}
	By Lemma~\ref{Lem-SingCatTriFunLeftAdj-FD-Complex}, the triangle functor $X\otimes^L_B-: \D(B) \to \D(A)$ induces a triangle functor $_AX\otimes^L_B-: \sg(B)\to \sg(A)$ with a left adjoint $\RHom_{B^\op}(X,B)\otimes^L_A-:\sg(A)\to \sg(B)$ if and only if both $X\in\D(A\otimes B^\op)$ and $\RHom_{B^\op}(X,B)\in\D(B\otimes A^\op)$ are biperfect. In this case, we have an adjoint pair $\RHom_{B^\op}(X,B)\otimes^L_A-:\sg(A) \rightleftarrows \sg(B):X\otimes^L_B-$. The triangle functor $_AX\otimes^L_B-: \sg(B)\to \sg(A)$ is fully faithful if and only if the counit $\eta:\RHom_{B^\op}(X,B)\otimes^L_AX\otimes^L_B- \to \Id_{\sg(B)}$ of the adjoint pair is a natural isomorphism by \cite[Proposition 1.5.6 (ii)]{Kashiwara-Schapira06}, if and only if the morphism $\eta_Y:\RHom_{B^\op}(X,B)\otimes^L_AX\otimes^L_BY \to Y$ in $\D^b(B)$ becomes an isomorphism in $\sg(B)$ for all $Y\in \D^b(B)$, if and only if $\Cone(\eta_Y)\in\per(B)$ for all $Y\in \D^b(B)$, if and only if $\Cone(\ev^B_X)\otimes^L_BY\in\per(B)$ for all $Y\in\D^b(B)$ by Lemma~\ref{Lem-Cone-Cone}, if and only if the $B$-bimodule complex $\Cone(\ev^B_X)$ is $\otimes$-perfect by Proposition~\ref{Prop-BimodComplex-TensorPerfect}.
\end{proof}

The following result is a characterization of a standard triangle functor on singularity categories with a fully faithful canonical left adjoint.

\begin{lemma} \label{Lem-jSingCatTriFunFFLeftAdj-FD-Complex} 
	Let $A$ and $B$ be finite dimensional algebras. Then a complex $X$ of $A$-$B$-bimodules induces a triangle functor $_AX\otimes^L_B-: \sg(B)\to \sg(A)$ with a fully faithful left adjoint $\RHom_{B^\op}(X,B)\otimes^L_A-:\sg(A)\to\sg(B)$ if and only if both $X$ and $\RHom_{B^\op}(X,B)$ are biperfect, and the $A$-bimodule complex $\Cone(l^A_X:A\to \RHom_{B^\op}(X,X))$ is $\otimes$-perfect, where the morphism $l^A_X:A\to \RHom_{B^\op}(X,X)$ in $\D(A^e)$ is induced by the morphism $l^A_{{\bf p}X}:A\to \Hom_{B^\op}({\bf p}X,{\bf p}X)$ in ${\cal H}(A^e)$ given by $l^A_{{\bf p}X}(a)(x):=ax$ for all $a\in A$ and $x\in{\bf p}X$.	
\end{lemma}

\begin{proof}
	By Lemma~\ref{Lem-SingCatTriFunLeftAdj-FD-Complex}, a complex $X$ of $A$-$B$-bimodules induces a triangle functor $_AX\otimes^L_B-: \sg(B)\to \sg(A)$ with a left adjoint $\RHom_{B^\op}(X,B)\otimes^L_A-:\sg(A)\to \sg(B)$ if and only if both $X$ and $\RHom_{B^\op}(X,B)$ are biperfect. In this case, we have an adjoint pair $$\RHom_{B^\op}(X,B)\otimes^L_A-:\sg(A) \rightleftarrows \sg(B):\RHom_B(\RHom_{B^\op}(X,B),-) \cong X\otimes^L_B-.$$ The triangle functor $\RHom_{B^\op}(X,B) \otimes^L_A-:\sg(A)\to\sg(B)$ is fully faithful if and only if the unit $\varepsilon:\Id_{\sg(A)} \to \RHom_B(\RHom_{B^\op}(X,B),\RHom_{B^\op}(X,B)\otimes^L_A-)$ of the adjoint pair is a natural isomorphism by \cite[Proposition 1.5.6 (i)]{Kashiwara-Schapira06}, if and only if the morphism $\varepsilon_Y:Y \to \RHom_B(\RHom_{B^\op}(X,B),\linebreak\RHom_{B^\op}(X,B)\otimes^L_AY)$ in $\D^b(A)$ is an isomorphism in $\sg(A)$ for all $Y\in \D^b(A)$, if and only if $\Cone(\varepsilon_Y)\in\per(A)$ for all $Y\in \D^b(A)$, if and only if $\Cone(\varepsilon^A_A:A\to \RHom_B(\RHom_{B^\op}(X,B),\RHom_{B^\op}(X,B)))\otimes^L_AY\in\per(A)$ for all $Y\in\D^b(A)$ by Lemma~\ref{Lem-Cone-Cone}, if and only if the $A$-bimodule complex $\Cone(\varepsilon^A_A)= \Cone(r^A_{\RHom_{B^\op}(X,B)})$, which is isomorphic to $\Cone(l^A_X)$ by the following commutative diagram,
	$$\begin{tikzcd}
		A \arrow[r,"l^A_X"] \arrow[d,"r^A_{\RHom_{B^\op}(X,B)}"'] & \RHom_{B^\op}(X,X) \arrow[dl,"\cong"] \\ 
		\RHom_B(\RHom_{B^\op}(X,B),\RHom_{B^\op}(X,B)) &
	\end{tikzcd}$$
	is $\otimes$-perfect by Proposition~\ref{Prop-BimodComplex-TensorPerfect}. 
\end{proof}

Similarly, we can obtain the following characterization of a standard triangle functor on singularity categories with a canonical right adjoint.

\begin{lemma} \label{Lem-SingCatTriFunRightAdj-FD-Complex} 
	Let $A$ and $B$ be finite dimensional algebras. Then a complex $X$ of $A$-$B$-bimodules induces a triangle functor $_AX\otimes^L_B-: \sg(B)\to \sg(A)$ with a right adjoint $\RHom_A(X,A)\otimes^L_A-:\sg(A)\to\sg(B)$ if and only if both $X$ and $\RHom_A(X,A)$ are biperfect.
\end{lemma}

The following result is a characterization of a fully faithful standard triangle functor on singularity categories with a canonical right adjoint.

\begin{lemma} \label{Lem-SingCatFFTriFunRightAdj-FD-Complex} 
	Let $A$ and $B$ be finite dimensional algebras. Then a complex $X$ of $A$-$B$-bimodules induces a fully faithful triangle functor $_AX\otimes^L_B-: \sg(B)\to \sg(A)$ with a right adjoint $\RHom_A(X,A)\otimes^L_A-:\sg(A)\to\sg(B)$ if and only if both $X$ and $\RHom_A(X,A)$ are biperfect, and the $B$-bimodule complex $\Cone(r^B_X:B\to\RHom_A(X,X))$ is $\otimes$-perfect.
\end{lemma}

The following result is a characterization of a standard triangle functor on singularity categories with a fully faithful canonical right adjoint.

\begin{lemma} \label{Lem-SingCatTriFunFFRightAdj-FD-Complex} 
	Let $A$ and $B$ be finite dimensional algebras. Then a complex $X$ of $A$-$B$-bimodules induces a triangle functor $_AX\otimes^L_B-: \sg(B)\to \sg(A)$ with a fully faithful right adjoint $\RHom_A(X,A)\otimes^L_A-:\sg(A)\to\sg(B)$ if and only if both $X$ and $\RHom_A(X,A)$ are biperfect, and the $A$-bimodule complex $\Cone(\ev^A_X:X\otimes^L_B\RHom_A(X,A)\to A)$ is $\otimes$-perfect.
\end{lemma}

\subsection{Definitions and characterizations}

Two finite dimensional algebras $A$ and $B$ are said to be {\it singularly equivalent} if there is a triangle equivalence between $\sg(A)$ and $\sg(B)$. In general, it is difficult to characterize a singular equivalence even a {\it standard} singular equivalence, i.e., a singular equivalence induced by a derived tensor product functor on derived categories. However, it is easier to characterize a standard singular equivalence with canonical adjoints.

\begin{definition}{\rm 
	Let $A$ and $B$ be finite dimensional algebras. We say that a complex $X$ of $A$-$B$-bimodules induces a {\it singular equivalence of $(m,n)$-adjoint type} for positive integers $m$ and $n$ if the triangle functor $_AX\otimes^L_B-: \D(B)\to \D(A)$ induces a triangle equivalence $_AX\otimes^L_B-: \sg(B)\to \sg(A)$, and the bimodule complexes $X^m,\cdots,X^2,X^1:=X=:X_1,X_2\cdots,X_n$ are biperfect,
	where $X^{2i}:=\RHom_{B^\op}(X^{2i-1},B),\linebreak X^{2i+1} :=\RHom_{A^\op}(X^{2i},A),X_{2i}:=\RHom_A(X_{2i-1},A),   X_{2i+1}:=\RHom_B(X_{2i},B)$, for all $i\ge 1$. In particular, a singular equivalence of $(1,n)$-adjoint type is called a {\it singular equivalence of (right) $n$-adjoint type}, and a singular equivalence of $(m,1)$-adjoint type is called a {\it singular equivalence of left $m$-adjoint type}.
}\end{definition}

If a complex $X$ of $A$-$B$-bimodules induces a singular equivalence of $(m,n)$-adjoint type then we have 
an adjoint $(m+n-1)$-tuple
$$\begin{tikzcd}[column sep=60]
	\sg(B) \arrow[r,"{X\otimes^L_B-}"] \arrow[r,shift left=12,"{X^3\otimes^L_B-}"] \arrow[r,no head,shift left=24,"{X^m\otimes^L_{B_m}-}"] \arrow[r,shift right=12,"{X_3\otimes^L_B-}"] \arrow[r,no head,shift right=18,"\vdots","{X_n\otimes^L_{B_n}-}"'] \arrow[r,no head,shift right=24] & \sg(A) \arrow[l,shift left=6,"{X_2\otimes^L_A-}"'] \arrow[l,shift right=6,"{X^2\otimes^L_A-}"'] \arrow[l,shift right=18,"\vdots"']
\end{tikzcd}$$
where $B_p:=B$ if $p$ is odd, and $B_p:=A$ if $p$ is even. Obviously, all triangle functors in the above diagram are triangle equivalences.
Note that a singular equivalence of $(1,1)$-adjoint type induced by $X\in\D(A\otimes B^\op)$ is just a singular equivalence $_AX\otimes^L_B-: \sg(B)\to \sg(A)$. A singular equivalence of $(2,1)$-adjoint type induced by $X\in\D(A\otimes B^\op)$ is just a singular equivalence $_AX\otimes^L_B-: \sg(B)\to \sg(A)$ with a left adjoint $\RHom_{B^\op}(X,B)\otimes^L_A-: \sg(A)\to \sg(B)$. A singular equivalence of $(1,2)$-adjoint type induced by $X\in\D(A\otimes B^\op)$ is just a singular equivalence $_AX\otimes^L_B-: \sg(B)\to \sg(A)$ with a right adjoint $\RHom_A(X,A)\otimes^L_A-: \sg(A)\to \sg(B)$. Moreover, a singular equivalence of $(m,n)$-adjoint type induced by $X\in\D(A\otimes B^\op)$ must be a singular equivalence of $(m',n')$-adjoint type induced by $X\in\D(A\otimes B^\op)$ for all $1\le m'\le m$ and $1\le n'\le n$.

\begin{remark}{\rm 
	A singular equivalence of $(m,n)$-adjoint type induced by $X\in\D(A\otimes B^\op)$ is a singular equivalence of left $(m+n-1)$-adjoint type induced by $X_n$, and a singular equivalence of (right) $(m+n-1)$-adjoint type induced by $X^m$. 
	Thus usually we only study singular equivalences of (right) $n$-adjoint type.
}\end{remark}

The following result is a characterization of singular equivalences of 2-adjoint type, which is more subtle than \cite[Theorem 3.6]{Dalezios21}.

\begin{theorem} \label{Thm-Exist-SingEquRightAdjoit-FD} 
	Let $A$ and $B$ be finite dimensional algebras. Then a complex $X$ of $A$-$B$-bimodules induces a singular equivalence $_AX\otimes^L_B-: \sg(B)\to \sg(A)$ with right adjoint $\RHom_A(X,A)\otimes^L_A-:\sg(A)\to \sg(B)$ if and only if the bimodule complexes $X$ and $\RHom_A(X,A)$ are biperfect, and either {\rm (1)} or {\rm (2)} is satisfied:
	
	{\rm (1)} The $B$-bimodule complex $\Cone(r^B_X:B\to \RHom_A(X,X))$ and the $A$-bimodule complex $\Cone(\ev^A_X:X\otimes^L_B\RHom_A(X,A)\to A)$ are $\otimes$-perfect. 

    {\rm (2)} The $B$-bimodule complex $\Cone(\ev^B_{X'}: X'\otimes^L_A\RHom_{B}(X',B)\to B)$ and the $A$-bimodule complex $\Cone(r^A_{X'}:A\to \RHom_B(X',X'))$ are $\otimes$-coperfect,
where $X':=\RHom_A(X,A)$.
\end{theorem}

\begin{proof}
	It follows from Lemma~\ref{Lem-SingCatTriFunRightAdj-FD-Complex} that a complex $X$ of $A$-$B$-bimodules induces a triangle functor $_AX\otimes^L_B-: \sg(B)\to \sg(A)$ with a right adjoint $\RHom_A(X,A)\otimes^L_A-:\sg(A)\to \sg(B)$ if and only if the bimodule complexes $X$ and $\RHom_A(X,A)$ are biperfect. In this case, we have an adjoint pair $X\otimes^L_B-:\sg(B) \rightleftarrows \sg(A): \RHom_A(X,A)\otimes^L_A-$.
	Furthermore, the triangle functor $X\otimes^L_B-: \sg(B)\to \sg(A)$ is a singular equivalence if and only if the triangle functors $X\otimes^L_B-$ and $\RHom_A(X,A)\otimes^L_A-$ are fully faithful (Ref. \cite[Proposition 1.5.6 (iii) (c)]{Kashiwara-Schapira06}), if and only if the $B$-bimodule complex $\Cone(r^B_X:B\to \RHom_A(X,X))$ and the $A$-bimodule complex $\Cone(\ev^A_X:X\otimes^L_B\RHom_A(X,A)\to A)$ are $\otimes$-perfect by Lemma~\ref{Lem-SingCatFFTriFunRightAdj-FD-Complex} and Lemma~\ref{Lem-SingCatTriFunFFRightAdj-FD-Complex}.

Now we establish the equivalence between conditions (1) and (2), which completes the proof.

Let $\varepsilon$ be the unit of the adjoint pair $(_AX\otimes^L_B-, \RHom_A(X, -))$, and $Y, Z\in \mathcal{D}^b(B)$.
Then there is a triangle
$$Y \xrightarrow{\varepsilon_Y} \RHom_A(X,X\otimes^L_BY) \to \mathrm{Cone}(\varepsilon_Y) \to $$ in $\mathcal{D}(B)$.
Applying the triangle functor $\RHom _{B}(-, Z)$ and taking cohomologies, we get a long exact sequence of $k$-vector spaces 
$$\begin{array}{ll}
	\cdots & \rightarrow \Hom _{\mathcal{D}(B)}(\mathrm{Cone}(\varepsilon_Y), Z[i]) \\
	& \rightarrow \Hom _{\mathcal{D}(B)}(\RHom_A(X,X\otimes^L_BY), Z[i])
	\xrightarrow{\Hom _{\mathcal{D}(B)}(\varepsilon_Y, Z[i])} \Hom _{\mathcal{D}(B)}(Y, Z[i])\rightarrow \cdots.
\end{array}$$
By Proposition~\ref{Prop-BimodComplex-TensorPerfect}, the $B$-bimodule complex $\Cone(r^B_X)$ is $\otimes$-perfect if and only if $\Cone(r^B_X)\otimes^L_BY\in \per (B)$ for all $Y\in \mathcal{D}^b(B)$, 
if and only if $\RHom _{B}(\Cone(r^B_X)\otimes^L_BY, Z)\in \mathcal{D}^b(k)$ for all $Y, Z\in \mathcal{D}^b(B)$.
Since $_AX$ is perfect, we have $\mathrm{Cone}(\varepsilon_Y)\cong \Cone(r^B_X)\otimes^L_BY$ by Lemma~\ref{Lem-Cone-Cone}.
Due to $\RHom_A(X,X\otimes^L_BY) \in \mathcal{D}^b(B)$, we obtain $\Cone(r^B_X)\otimes^L_BY\in \mathcal{D}^b(B)$, and thus $\RHom _{B}(\Cone(r^B_X)\otimes^L_BY, Z)\in \mathcal{D}^+(k)$.
So $\RHom _{B}(\Cone(r^B_X)\otimes^L_BY, Z)\in \mathcal{D}^b(k)$ if and only if 
$\Hom _{\mathcal{D}(B)}(\Cone(r^B_X)\otimes^L_BY, Z[i])=0$ for $i\gg 0$,
if and only if $\Hom _{\mathcal{D}(B)}(\varepsilon_Y, Z[i])$ is an isomorphism for $i\gg 0$.
Applying Lemma~\ref{Lem-F_XY-Unit} (3) to the adjoint triple $(X\otimes^L_B-, \RHom_A(X,-)\cong X'\otimes^L_A-, \RHom_B(X',-))$, 
we get that $$\Hom _{\mathcal{D}(B)}(\varepsilon_Y, Z[i]): \Hom _{\mathcal{D}(B)}(\RHom_A(X,X\otimes^L_BY), Z[i]) \rightarrow \Hom _{\mathcal{D}(B)}(Y, Z[i])$$
is an isomorphism if and only if $$\Hom _{\mathcal{D}(B)}(Y, \eta _Z[i]): \Hom _{\mathcal{D}(B)}(Y, X'\otimes^L_A\RHom_{B}(X',Z)[i]) \rightarrow \Hom _{\mathcal{D}(B)}(Y, Z[i])$$
is an isomorphism, where $\eta$ is the counit of the adjoint pair $(X'\otimes^L_A-, \RHom_B(X',-))$.
Hence the $B$-bimodule complex $\Cone(r^B_X)$ is $\otimes$-perfect if and only the map $\Hom _{\mathcal{D}(B)}(Y, \eta _Z[i])$
is an isomorphism for all $Y, Z\in \mathcal{D}^b(B)$ and $i\gg 0$.
Applying $\RHom _B(Y,-)$ to the triangle 
$$X'\otimes^L_A\RHom_B(X',Z) \xrightarrow{\eta _Z} Z \rightarrow \Cone(\eta _Z)\rightarrow$$
and taking cohomologies, we get a long exact sequence of $k$-vector spaces 
$$\begin{array}{ll}
	\cdots & \rightarrow \Hom _{\mathcal{D}(B)}(Y, X'\otimes^L_A\RHom_B(X',Z)[i]) \xrightarrow{\Hom _{\mathcal{D}(B)}(Y, \eta_Z[i])} \Hom_{\mathcal{D}(B)}(Y, Z[i]) \\ [2mm]
	& \rightarrow \Hom _{\mathcal{D}(B)}(Y, \Cone(\eta _Z)[i]) \rightarrow \cdots.
\end{array}$$
Thus $\Hom _{\mathcal{D}(B)}(Y, \eta _Z[i])$ is an isomorphism for $i\gg 0$
if and only if $\Hom _{\mathcal{D}(B)}(Y, \Cone(\eta _Z)[i])=0$ for $i\gg 0$. On the other hand, it follows from
Lemma~\ref{Lem-Cone-Cone} that $\Cone(\eta _Z)\cong \Cone(\ev^B_{X'})\otimes^L_BZ$, and thus $\Cone(r^B_X)$ is $\otimes$-perfect
if and only if $\Hom _{\mathcal{D}(B)}(Y, \Cone(\ev^B_{X'})\otimes^L_BZ[i])=0$ for  all $Y, Z\in \mathcal{D}^b(B)$ and $i\gg 0$.
Since $_BX'$ is perfect, the functor $\RHom_{B}(X',-)$ can be restricted to bounded derived categories,
and thus $X'\otimes^L_A\RHom_{B}(X',Z)\in \mathcal{D}^b(B)$. 
By the triangle $X'\otimes^L_A\RHom_{B}(X',Z) \xrightarrow{\eta _Z} Z \rightarrow \Cone(\eta _Z)\rightarrow$,
we obtain $\Cone(\eta _Z)\in \mathcal{D}^b(B)$, which implies
$\Cone(\ev^B_{X'})\otimes^L_BZ \in \mathcal{D}^b(B)$. 
Then $\RHom _B(Y, \Cone(\ev^B_{X'})\otimes^L_BZ)\in \mathcal{D}^+(k)$ for any $Y,Z\in \mathcal{D}^b(B)$,
and thus $\Cone(r^B_X)$ is $\otimes$-perfect
if and only if 
$\RHom_B(Y, \Cone(\ev^B_{X'})\otimes^L_BZ)\in \mathcal{D}^b(k)$ for all $Y, Z\in \mathcal{D}^b(B)$,
if and only if
$\Cone(\ev^B_{X'})\otimes^L_BZ\in K^b(\inj B)$ for all $Z\in \mathcal{D}^b(B)$, that is,
$\Cone(\ev^B_{X'})$ is $\otimes$-coperfect. 

By Theorem~\ref{Thm-EvHomIso-TensorPerfect-BimodComplex-RightAdjoint}, the $A$-bimodule complex $\Cone(\ev^A_X)$ is $\otimes$-perfect
if and only if the functor $X'\otimes^L_A-:\D^b(A)\to\D^b(B)$ is an eventually homological isomorphism,
if and only if the $A$-bimodule complex $\Cone(r^A_{X'})$ is $\otimes$-coperfect. 
Therefore, the conditions (1) and (2) are equivalent.
\end{proof}

Similarly, we have the following characterization of singular equivalences of left 2-adjoint type.

\begin{theorem} \label{Thm-Exist-SingEquLeftAdjoit-FD} 
	Let $A$ and $B$ be finite dimensional algebras. Then a complex $X$ of $A$-$B$-bimodules induces a singular equivalence $_AX\otimes^L_B-: \sg(B)\to \sg(A)$ with a left adjoint $\RHom_{B^\op}(X,B)\otimes^L_A-:\sg(A)\to \sg(B)$ if and only if the bimodule complex $X$ and $\RHom_{B^\op}(X,B)$ are biperfect, and either
{\rm (1)} or {\rm (2)} is satisfied:
	
{\rm (1)} The $A$-bimodule complex $\Cone(l^A_X:A\to \RHom_{B^\op}(X,X))$ and the $B$-bimodule complex $\Cone(\ev^B_X:\RHom_{B^\op}(X,B)\otimes^L_AX\to B)$ are $\otimes$-perfect.	

{\rm (2)} The $A$-bimodule complex $\Cone(\ev^A_X:X\otimes^L_B\RHom_A(X,A)\to A)$ and the $B$-bimodule complex $\Cone(r^B_X:B\to \RHom_A(X,X))$ are $\otimes$-coperfect. 
\end{theorem}

From Theorem~\ref{Thm-Exist-SingEquRightAdjoit-FD}, we can obtain the following characterization of singular equivalences of $n$-adjoint type for $n\ge 2$.

\begin{theorem} \label{Thm-Exist-SingEquAdjTye-FD} 
	Let $A$ and $B$ be finite dimensional algebras. Then a complex $X$ of $A$-$B$-bimodules induces a singular equivalence $X\otimes^L_B-:\sg(B)\to \sg(A)$ of $n$-adjoint type for some $n\ge 2$ if and only if the bimodule complexes $X_1,\cdots,X_n$ are biperfect for all $i\ge 1$ where $X_1:=X, X_{2i}:=\RHom_A(X_{2i-1},A),\linebreak X_{2i+1}:=\RHom_B(X_{2i},B)$,
and one of the following conditions is satisfied:
	
	{\rm (1)} The $B$-bimodule complex $\Cone(r^B_{X_{2i-1}}:B\to \RHom_A(X_{2i-1},X_{2i-1}))$ and the $A$-bimodule complex $\Cone(\ev^A_{X_{2i-1}}:X_{2i-1}\otimes^L_B\RHom_A(X_{2i-1},A)\to A)$ are $\otimes$-perfect for some or all $1\le i\le [n/2]$.

{\rm (1')} The $A$-bimodule complex $\Cone(r^A_{X_{2i}}:A\to \RHom_B(X_{2i},X_{2i}))$ and the $B$-bimodule complex $\Cone(\ev^B_{X_{2i}}:X_{2i}\otimes^L_A\RHom_B(X_{2i},B)\to B)$ are $\otimes$-coperfect for some or all $1\le i\le [n/2]$. 
	
	{\rm (2)} The $A$-bimodule complex $\Cone(r^A_{X_{2i}}:A\to \RHom_B(X_{2i},X_{2i}))$ and the $B$-bimodule complex $\Cone(\ev^B_{X_{2i}}:X_{2i}\otimes^L_A\RHom_B(X_{2i},B)\to B)$ are $\otimes$-perfect for some or all $1\le i\le [(n-1)/2]$. 

{\rm (2')} The $B$-bimodule complex $\Cone(r^B_{X_{2i+1}}:B\to \RHom_A(X_{2i+1},X_{2i+1}))$ and the $A$-bimodule complex $\Cone(\ev^A_{X_{2i+1}}:X_{2i+1}\otimes^L_B\RHom_A(X_{2i+1},A)\to A)$ are $\otimes$-coperfect for some or all $1\le i\le [(n-1)/2]$.
	
\end{theorem}

\begin{proof}
	If a complex $X$ of $A$-$B$-bimodules induces a singular equivalence $X\otimes^L_B-:\sg(B)\to \sg(A)$ of $n$-adjoint type for some $n\ge 2$ then the complex $X_{2i-1}$ of $A$-$B$-bimodules induces a singular equivalence $X_{2i-1}\otimes^L_B-:\sg(B)\to\sg(A)$ of 2-adjoint type for all $1\le i\le [n/2]$, and the complex $X_{2i}$ of $B$-$A$-bimodules induces a singular equivalence $X_{2i}\otimes^L_A-:\sg(A)\to\sg(B)$ of 2-adjoint type for all $1\le i\le [(n-1)/2]$. It follows from Theorem~\ref{Thm-Exist-SingEquRightAdjoit-FD} that
the bimodule complexes $X_1,\cdots,X_n$ are biperfect for all $i\ge 1$, and one of (1), (1'), (2) or (2') holds. Conversely, if the bimodule complexes $X_1,\cdots,X_n$ are biperfect for all $i\ge 1$ and one of (1), (1'), (2) or (2') holds, then by Theorem~\ref{Thm-Exist-SingEquRightAdjoit-FD}, the complex $X_{2i-1}$ or $X_{2i}$ induces a singular equivalence of 2-adjoint type. Thus the complex $X$ of $A$-$B$-bimodules induces a singular equivalence $X\otimes^L_B-:\sg(B)\to \sg(A)$ of $n$-adjoint type.	
\end{proof}

\medspace

\noindent{\bf Comparison with singular equivalences of Morita type with level.}
If two finite dimensional $k$-algebras $A$ and $B$ are singularly equivalent of 2-adjoint type, and $A/\rad A$ and $B/\rad B$ are separable over $k$, then $A$ and $B$ are singularly equivalent of Morita type with level $l$ for some $l\ge 0$ (Ref. \cite[Theorem 1.1]{Qin22}). 

The following is an example of singular equivalences of 2-adjoint type but not singular equivalences of Morita type with level.

\begin{example} \label{Ex-SE2AT-NotSEMTWL} {\rm 
		The inseparable extension of fields $k:=\mathbb{Z}_p(t) \subset K:= k[x]/(x^p-t)$ induces a trivial singular equivalence $_KK\otimes^L_k-=0: \sg(k)=0\to \sg(K)=0$ with a right adjoint $_kK\otimes^L_K-=0: \sg(K)=0\to \sg(k)=0$. However, $k$ and $K$ are not singularly equivalent of Morita type with level $l$ for all $l\ge 0$. Otherwise, there exist a $k$-$K$-bimodule $M=K^m$ and a $K$-$k$-bimodule $N=K^n$ such that $M\otimes_KN\cong \Omega^l_{k^e}(k)$ in $\underline{\mod} k^e$ and $N\otimes_kM\cong \Omega^l_{K^e}(K)$ in $\underline{\mod}K^e$ for some $l\ge 0$. Due to $K^e\cong K[y]/(y^p)$, we have $\Omega^l_{K^e}(K)\notin \proj(K^e)$ for all $l\ge 0$, but $N\otimes_kM\cong (K^e)^{mn}\in \proj(K^e)$. This contradicts $N\otimes_kM\cong \Omega^l_{K^e}(K)$ in $\underline{\mod}K^e$.	
}\end{example}

\begin{question}{\rm
		Is there a singular equivalence of Morita type with level which is not a singular equivalence of 2-adjoint type?
}\end{question}

\begin{question}{\rm
		If two finite dimensional algebras are singularly equivalent, is there a standard singular equivalence, or equivalently, a singular equivalence of 1-adjoint type between them?
}\end{question}

\begin{question}{\rm
		Is every singular equivalence standard, i.e., of 1-adjoint type, up to natural isomorphisms?
}\end{question}

\subsection{Constructions}

In this subsection, we provide two constructions of singular equivalences of $n$-adjoint type by opposite algebras and tensor products of algebras for all $n\ge 2$.

The following result is a construction of singular equivalences of $n$-adjoint type by opposite algebras for all $n\ge 2$.

\begin{proposition} \label{Prop-SE-OppositeAlg}
	Let $A$ and $B$ be two finite dimensional algebras. Then a complex $X$ of $A$-$B$-bimodules induces a singular equivalence $_AX\otimes^L_B-: \sg(B)\to\sg(A)$ of (right) $n$-adjoint type for some $n\ge 2$ if and only if $X$ induces a singular equivalence $X\otimes^L_{A^\op}- = -\otimes^L_AX: \sg(A^\op)\to\sg(B^\op)$ of left $n$-adjoint type.	
\end{proposition}

\begin{proof}
	It follows from Theorem~\ref{Thm-Exist-SingEquAdjTye-FD} that a complex $X$ of $A$-$B$-bimodules induces a singular equivalence $_AX\otimes^L_B-: \sg(B)\to\sg(A)$ of (right) $n$-adjoint type for some $n\ge 2$ if and only if the bimodule complexes $X_1,\cdots,X_n$ are biperfect,
	where $X_1:=X, X_{2i}:=\RHom_A(X_{2i-1},A), X_{2i+1}:=\RHom_B(X_{2i},B)$, for all $i\ge 1$, and the $B$-bimodule complex $\Cone(r^B_X:B\to \RHom_A(X,X))$ and the $A$-bimodule complex $\Cone(\ev^A_X:X\otimes^L_B\RHom_A(X,A)\to A)$ are $\otimes$-perfect.
	
	Note that $X_i$ is a complex of $A_i$-$B_i$-bimodules, or equivalently, a complex of $B_i^\op$-$A_i^\op$-bimodules, where $A_i:=A$ and $B_i:=B$ if $i$ is odd, and $A_i:=B$ and $B_i:=A$ if $i$ is even. It is clear that $\RHom_{B_i^\op}(X_i,B_i^\op) \cong \RHom_{B_i^\op}(\RHom_{A_{i-1}}(X_{i-1},A_{i-1}),B_i^\op) \cong X_{i-1}$ for all $i\ge 2$. Moreover, the $B^\op$-bimodule complex $\Cone(r^{B^\op}_{X_2}:B^\op\to \RHom_{A^\op}(X_2,X_2)) \cong \Cone(r^B_X:B\to \RHom_A(X,X))$ in $\D(B^e)$ and the $A^\op$-bimodule complex $\Cone(\ev^{A^\op}_{X_2}:X_2\otimes^L_{B^\op}\RHom_{A^\op}(X_2,A^\op)\to A^\op)\cong \Cone(\ev^A_X:X\otimes^L_B\RHom_A(X,A)\to A)$ in $\D(A^e)$.	
	$$\begin{tikzcd}[column sep=100]
		\sg(A^\op) \arrow[r,"{X\otimes^L_{A^\op}-\ =\ -\otimes_AX}"] \arrow[r,shift left=12,"{X_3\otimes^L_{A^\op}-\ =\ -\otimes^L_AX_3}"] \arrow[r,no head,shift left=24,"{X_n\otimes^L_{A_n^\op}-\ =\ -\otimes^L_{A_n}X_n}"] & \sg(B^\op). \arrow[l,shift right=6,"{X_2\otimes^L_{B^\op}-\ =\ -\otimes^L_BX_2}"'] \arrow[l,shift right=18,"\vdots"']
	\end{tikzcd}$$
	Thus $X$ induces a singular equivalence $X\otimes^L_{A^\op}-= -\otimes^L_AX: \sg(A^\op)\to\sg(B^\op)$ of left $n$-adjoint type, if and only if the complex $X_n$ of $A_n$-$B_n$-bimodules induces a singular equivalence $_{B_n^\op}X_n\otimes^L_{A_n^\op}-: \sg(A_n^\op)\to\sg(B_n^\op)$ of (right) $n$-adjoint type, if and only if the bimodule complexes $X_1,\cdots,X_n$ are biperfect, and the $B$-bimodule complex $\Cone(r^B_X:B\to \RHom_A(X,X))$ and the $A$-bimodule complex $\Cone(\ev^A_X:X\otimes^L_B\RHom_A(X,A)\to A)$ are $\otimes$-perfect by Theorem~\ref{Thm-Exist-SingEquAdjTye-FD}, if and only if the complex $X$ of $A$-$B$-bimodules induces a singular equivalence $_AX\otimes^L_B-: \sg(B)\to\sg(A)$ of (right) $n$-adjoint type.	
\end{proof}

The following result is a construction of singular equivalences of $n$-adjoint type by tensor products of algebras for all $n\ge 2$.

\begin{proposition} \label{Prop-SE-TensorProdAlg}
	Let $A$ and $B$ be finite dimensional algebras, $C$ a finite dimensional algebra with $\gl C<\infty$, and $B/\rad B$ or $C/\rad C$ separable over $k$. Then a complex $X$ of $A$-$B$-bimodules induces a singular equivalence $X\otimes^L_B-: \sg(B)\to\sg(A)$ of $n$-adjoint type for some $n\ge 2$ if and only if the complex $X\otimes C$ of $(A\otimes C)$-$(B\otimes C)$-bimodules induces a singular equivalence $(X\otimes C)\otimes^L_{B\otimes C}- \cong X\otimes^L_B-: \sg(B\otimes C)\to\sg(A\otimes C)$ of $n$-adjoint type.
\end{proposition}

\begin{proof}
	Note that the assumption $\gl C<\infty$ implies $\ol{C}\otimes^L_C\ol{C}\in\D^b(k)$ where $\ol{C}=C/\rad C$. And the assumption that $B/\rad B$ or $C/\rad C$ is separable over $k$ implies $\ol{B\otimes C}=\ol{B}\otimes\ol{C}$ by Lemma~\ref{Lem-TensorProdTop=TopTensorProd}. 
	
	It follows from Theorem~\ref{Thm-Exist-SingEquAdjTye-FD} that a complex $X$ of $A$-$B$-bimodules induces a singular equivalence $X\otimes^L_B-: \sg(B)\to\sg(A)$ of $n$-adjoint type if and only if the bimodule complexes $X_1,\cdots,X_n$ are biperfect, where $X_1:=X, X_{2i}:=\RHom_A(X_{2i-1},A), X_{2i+1}:=\RHom_B(X_{2i},B)$, for all $i\ge 1$, and both the $B$-bimodule complex $\Cone(r^B_X:B\to \RHom_A(X,X))$ and the $A$-bimodule complex $\Cone(\ev^A_X:X\otimes^L_B\RHom_A(X,A)\to A)$ are $\otimes$-perfect.
	
	Note that $X_{2i}\otimes C \cong \RHom_{A\otimes C}(X_{2i-1}\otimes C,A\otimes C)$ in $\D((B\otimes C)\otimes(A\otimes C)^\op)$ and $X_{2i+1}\otimes C \cong \RHom_{B\otimes C}(X_{2i}\otimes C,B\otimes C)$ in $\D((A\otimes C)\otimes(B\otimes C)^\op)$ for all $i\ge 1$. And the bimodule complexes $X_1,\cdots,X_n$ are biperfect if and only if the bimodule complexes $X_1\otimes C,\cdots,X_n\otimes C$ are biperfect.	
	Moreover, we have the commutative diagram	
	$$\begin{tikzcd}[column sep=30]
		B\otimes C \arrow[r,"r^B_X\otimes C"] \arrow[d,equal]  & \RHom_A(X,X)\otimes C \arrow[r]\arrow[d,"\cong"] & \Cone(r^B_X)\otimes C\to \arrow[d,"\cong"] \\
		B\otimes C \arrow[r,"r^{B\otimes C}_{X\otimes C}"] & \RHom_{A\otimes C}(X\otimes C,X\otimes C) \arrow[r]  & \Cone(r^{B\otimes C}_{X\otimes C})\to
	\end{tikzcd}$$
	 in $\D((B\otimes C)^e)$ and the commutative diagram	
	$$\begin{tikzcd}[column sep=30]
		X\otimes^L_B\RHom_A(X,A)\otimes C \arrow[r,"\ev^A_X\otimes C"] \arrow[d,"\cong"]  & A\otimes C \arrow[r] \arrow[d,equal] & \Cone(\ev^A_X)\otimes C\to \arrow[d,"\cong"] \\
		(X\otimes C)\otimes^L_{B\otimes C}\RHom_{A\otimes C}(X\otimes C,A\otimes C) \arrow[r,"\ev^{A\otimes C}_{X\otimes C}"] & A\otimes C \arrow[r] & \Cone(\ev^{A\otimes C}_{X\otimes C})\to
	\end{tikzcd}$$
	\noindent in $\D((A\otimes C)^e)$. Thus the $(B\otimes C)$-bimodule complex $\Cone(r^{B\otimes C}_{X\otimes C})$ is $\otimes$-perfect if and only if so is $\Cone(r^B_X)\otimes C$, if and only if $(\ol{B}\otimes\ol{C})\otimes^L_{B\otimes C}(\Cone(r^B_X)\otimes C)\otimes^L_{B\otimes C}(\ol{B}\otimes\ol{C}) \cong (\ol{B}\otimes^L_B\Cone(r^B_X)\otimes^L_B\ol{B})\otimes(\ol{C}\otimes^L_C\ol{C})\in\D^b(k)$ by Proposition~\ref{Prop-BimodComplex-TensorPerfect}, if and only if $\ol{B}\otimes^L_B\Cone(r^B_X)\otimes^L_B\ol{B}\in\D^b(k)$ by the K\"{u}nneth formula for complexes of $k$-vector spaces (Ref. \cite[Theorem 3.6.3]{Weibel94}) since $\ol{C}\otimes^L_C\ol{C}\in\D^b(k)$ is nonzero, if and only if the $B$-bimodule complex $\Cone(r^B_X)$ is $\otimes$-perfect by Proposition~\ref{Prop-BimodComplex-TensorPerfect}. Similarly, the $(A\otimes C)$-bimodule complex  $\Cone(\ev^{A\otimes C}_{X\otimes C})$ is $\otimes$-perfect if and only if the $A$-bimodule complex  $\Cone(\ev^A_X)$ is $\otimes$-perfect.
	
	It follows from Theorem~\ref{Thm-Exist-SingEquAdjTye-FD} that the complex $X\otimes C$ of $(A\otimes C)$-$(B\otimes C)$-bimodules induces a singular equivalence $(X\otimes C)\otimes^L_{B\otimes C}-: \sg(B\otimes C)\to\sg(A\otimes C)$ of $n$-adjoint type if and only if the bimodule complexes $X_1\otimes C,\cdots,X_n\otimes C$ are biperfect, and the $(B\otimes C)$-bimodule complex $\Cone(r^{B\otimes C}_{X\otimes C})$ and the $(A\otimes C)$-bimodule complex $\Cone(\ev^{A\otimes C}_{X\otimes C})$ are $\otimes$-perfect, if and only if the bimodule complexes $X_1,\cdots,X_n$ are biperfect, and the $B$-bimodule complex $\Cone(r^B_X)$ and the $A$-bimodule complex $\Cone(\ev^A_X)$ are $\otimes$-perfect, if and only if the complex $X$ of $A$-$B$-bimodules induces a singular equivalence $X\otimes^L_B-: \sg(B)\to\sg(A)$ of $n$-adjoint type.
\end{proof}

\begin{remark} \label{Rem-SE-TensorNotSE} {\rm
	In Proposition~\ref{Prop-SE-TensorProdAlg}, the conditions that the algebra $C$ satisfies $\gl C<\infty$ and the algebra $B/\rad B$ or $C/\rad C$ is separable over $k$ are necessary. From the proof of Proposition~\ref{Prop-SE-TensorProdAlg}, it is clear that the condition $\gl C<\infty$ is necessary even though the algebra $C/\rad C$ is separable over $k$. 	
	Moreover, the extension $k:=\mathbb{Z}_p(t) \subset K:= k[x]/(x^p-t)$ of fields induces a triangle functor $_KK\otimes_k-: \D(k)\to \D(K)$ with a right adjoint $_kK\otimes_K-: \D(K)\to \D(k)$, which induces a trivial singular equivalence $_KK\otimes_k-=0: \sg(k)=0\to \sg(K)=0$ with a right adjoint $_kK\otimes_K-=0: \sg(K)=0\to \sg(k)=0$. However, the triangle functor $_KK\otimes_k-: \D(K)\to \D(K\otimes_kK)$ with a right adjoint $_kK\otimes_K-: \D(K\otimes_kK)\to \D(K)$ cannot induce a singular equivalence $_KK\otimes_k-: \sg(K)=0\to \sg(K\otimes_kK)$ with a right adjoint $_kK\otimes_K-: \sg(K\otimes_kK)\to \sg(K)=0$ due to $\gl K=0$ and $\gl (K\otimes_kK)=\infty$. 
}\end{remark}

\subsection{Examples}

In this subsection, we give the examples of singular equivalences of $n$-adjoint type induced by homomorphisms of algebras and idempotents. By Theorem~\ref{Thm-Exist-SingEquAdjTye-FD}, it suffices to consider the examples of singular equivalences of 2-adjoint type, since singular equivalences of $n$-adjoint type are just the singular equivalences of 2-adjoint type for which some bimodule complexes satisfy biperfectness. 

\medspace

\noindent{\bf Singular equivalences of $n$-adjoint type induced by homomorphisms of algebras.}

\begin{corollary} \label{Cor-SingEqu-AlgHom}
	Let $f:A\to B$ be a homomorphism of finite dimensional algebras. Then the triangle functor $_AB\otimes^L_B-:\sg(B)\to\sg(A)$ is a triangle equivalence with a left adjoint inverse $_BB\otimes^L_A-:\sg(A)\to\sg(B)$ if and only if the $A$-bimodule complex $\Cone(f)$ and the $B$-bimodule complex $\Cone(\ev^B_B:B\otimes^L_AB\to B)$ are $\otimes$-perfect.
\end{corollary}

\begin{proof}
	By Theorem~\ref{Thm-Exist-SingEquLeftAdjoit-FD}, the triangle functor $_AB\otimes^L_B-:\sg(B)\to\sg(A)$ is a singular equivalence with a left adjoint inverse $_BB\otimes^L_A-:\sg(A)\to\sg(B)$ if and only if $_AB\in\per(A)$, $B_A\in\per(A^\op)$, and the $A$-bimodule complex $\Cone(l^A_B:A\to\RHom_{B^\op}(B,B))$ and the $B$-bimodule complex $\Cone(\ev^B_B:B\otimes^L_AB\to B)$ are $\otimes$-perfect. The following commutative diagram implies $\Cone(l^A_B) \cong \Cone(f)$ in $\D(A^e)$. 
	$$\begin{tikzcd}
		A \arrow[r,"l^A_B"] \arrow[d,equal] & \RHom_{B^\op}(B,B) \arrow[d,"\cong"] \arrow[r] & \Cone(l^A_B) \arrow[d,"\cong"] \arrow[r] & A[1] \arrow[d,equal] \\
		A \arrow[r,"f"] & B \arrow[r] & \Cone(f) \arrow[r] & A[1]
	\end{tikzcd}$$
	Moreover, if the $A$-bimodule complex $\Cone(f)$ is $\otimes$-perfect then $_A\Cone(f)\in\per(A)$ and $\Cone(f)_A\in\per(A^\op)$. From the triangle $A\xrightarrow{f}B\to\Cone(f)\to$, we obtain $_AB\in\per(A)$ and $B_A\in\per(A^\op)$. 
\end{proof}

\begin{corollary} \label{Cor-SingEqu-HomEpi}
	Let $f:A\to B$ be a homological epimorphism of finite dimensional algebras. Then the triangle functor $_AB\otimes^L_B-:\sg(B)\to\sg(A)$ is a triangle equivalence with a left adjoint inverse $_BB\otimes^L_A-:\sg(A)\to\sg(B)$ if and only if the $A$-bimodule complex $\Cone(f)$ is $\otimes$-perfect.
\end{corollary}

\begin{proof}
	Since $f:A\to B$ is a homological epimorphism, the morphism $\ev^B_B:B\otimes^L_AB\to B$ is an isomorphism in $\D(B^e)$. Then the $B$-bimodule complex $\Cone(\ev^B_B)$ is zero, thus $\otimes$-perfect. So the corollary follows from Corollary~\ref{Cor-SingEqu-AlgHom}
\end{proof}

Recall that a two-sided ideal $I$ of a finite dimensional algebra $A$ is called a {\it homological ideal} \cite{Pena-Xi06} if the canonical homomorphism $f: A \to A/I$ is a homological epimorphism. 

A direct corollary of Corollary~\ref{Cor-SingEqu-HomEpi} is the following result which improves \cite[Theorem]{Chen14}.

\begin{corollary} {\rm (cf. \cite[Theorem]{Chen14})}
	Let $A$ be a finite dimensional algebra, $I$ a homological ideal of $A$, and $B:=A/I$. Then the triangle functor $_AB\otimes^L_B-:\sg(B)\to\sg(A)$ is a triangle equivalence with a left adjoint inverse $_BB\otimes^L_A-:\sg(A)\to\sg(B)$ if and only if the $A$-bimodule (complex) $I$ is $\otimes$-perfect.
\end{corollary}

\begin{proof}
	Note that the $A$-bimodule complex $\Cone(f:A\to A/I) \cong I[1]$ is $\otimes$-perfect if and only if so is $I$ by Proposition~\ref{Prop-TensorPerfect-Property}. Then the corollary follows from Corollary~\ref{Cor-SingEqu-HomEpi}. 
\end{proof}

The following result implies that bounded extensions and strongly proj-bounded extensions induce singular equivalences. 
For bounded extension, it is due to \cite[Theorem 2.7]{Qin-Xu-Zhang-Zhou24}.

\begin{corollary} \label{Cor-SingEqu-AlgExt} 
	Let $B\subseteq A$ be a bounded or strongly proj-bounded extension of finite dimensional algebras. Then the triangle functor $_BA\otimes^L_A-: \sg(A)\to\sg(B)$ is a triangle equivalence with a left adjoint $_AA\otimes^L_B-:\sg(B)\to\sg(A)$.
\end{corollary}

\begin{proof}	
	By Theorem~\ref{Thm-Exist-SingEquLeftAdjoit-FD}, the triangle functor $_BA\otimes^L_A-: \sg(A)\to\sg(B)$ is a triangle equivalence with a left adjoint inverse $_AA\otimes^L_B-:\sg(B)\to\sg(A)$ if and only if $_BA\in\per(B)$, $A_B\in\per(B^\op)$, and the $B$-bimodule complex $\Cone(B\hookrightarrow A)\cong A/B$ and the $A$-bimodule complex $\Cone(\ev^A_A:A\otimes^L_BA\to A)$ are $\otimes$-perfect.
	
	By the definition of bounded or strongly proj-bounded extension, we have $A/B\in\per(B^e)$. By Proposition~\ref{Prop-HomSmooth-TensorPerfect-BimodComplex}, the $B$-bimodule complex $A/B$ is $\otimes$-perfect. In particular, the $B$-bimodule (complex) $A/B$ is biperfect. From the triangle $B\hookrightarrow A\to A/B\to$, we obtain $_BA\in\per(B)$ and $A_B\in\per(B^\op)$.
	By the proof of Corollary~\ref{Cor-BoundExt-EHI}, the $A$-bimodule complex $\Cone(\ev^A_A)$ is $\otimes$-perfect.
\end{proof}

\medspace

\noindent{\bf Singular equivalences of $n$-adjoint type induced by idempotents of algebras.} Now we can recover the following characterization of singular equivalences induced by an idempotent in \cite{Psaroudakis-Skartsaeterhagen-Solberg14} and \cite{Shen21} respectively for finite dimensional algebras.

\begin{corollary} \label{Cor-Exist-SingEqu-Idemp} 
	Let $A$ be a finite dimensional algebra, and $e$ an idempotent in $A$. Then the following statements hold.
	
	{\rm (1)} {\rm (\cite[Main theorem (ii)]{Psaroudakis-Skartsaeterhagen-Solberg14})} The triangle functor $eA\otimes^L_A-: \sg(A)\to \sg(eAe)$ is a triangle equivalence if and only if $\pd_{eAe}eA<\infty$ and $\pd_A\frac{A/AeA}{\rad(A/AeA)}<\infty$.
	
	{\rm (2)} {\rm (\cite[Theorem II]{Shen21})} The triangle functor $Ae\otimes^L_{eAe}-: \sg(eAe)\to \sg(A)$ is a triangle equivalence if and only if $\pd_{(eAe)^\op} Ae<\infty$ and $\id_A\frac{A/AeA}{\rad(A/AeA)}<\infty$.	
\end{corollary}

\begin{proof}	
	The idempotent $e$ induces a recollement (Ref. \cite[Proposition 2.16]{Jin-Yang-Zhou23})
	$$\begin{tikzcd}[column sep=80]
		\D(Q) \arrow[r,"{_AQ\otimes^L_Q-}"] & \D(A) \arrow[r,"{_{eAe}eA\otimes^L_A-}"] \arrow[l,shift left=6,"{\RHom_A(Q,-)}"'] \arrow[l,shift right=6,"_QQ\otimes^L_A-"'] & \D(eAe) \arrow[l,shift left=6,"{\RHom_{eAe}(eA,-)}"'] \arrow[l,shift right=6,"{_AAe\otimes^L_{eAe}-}"']
	\end{tikzcd}$$
	where $Q$ is a cohomology non-positive dg algebra. 
	
	(1) By Lemma~\ref{Lem-Exist-TriFunSingCat-FD-Complex}, the triangle functor $eA\otimes^L_A-: \D(A)\to \D(eAe)$ induces a triangle functor $eA\otimes^L_A-: \sg(A)\to \sg(eAe)$ if and only if $\pd_{eAe}eA<\infty$. 	
	Under the assumption $\pd_{eAe}eA<\infty$, the recollement extends one step downwards by \cite[Lemma 2.11]{Jin-Yang-Zhou23}. Moreover, from the triangle $Ae\otimes^L_{eAe}eA\to A\to Q\to$ in $\D(A)$, we know that $Q$ is a cohomology non-positive and finite dimensional dg algebra.  
	By \cite[Theorem 4.1]{Jin-Yang-Zhou23}, we have a short exact sequence 
	$$\sg(Q)\xrightarrow{_AQ\otimes^L_Q-}\sg(A)\xrightarrow{_{eAe}eA\otimes^L_A-}\sg(eAe).$$	
	Then the triangle functor $eA\otimes^L_A-: \sg(A)\to \sg(eAe)$ is a triangle equivalence if and only if $_AQ\otimes^L_Q\D_\fd(Q) = \D^b(A)_{\mod A/AeA} \subseteq \per(A)$, if and only if $\mod A/AeA\subseteq \per(A)$, if and only if $\frac{A/AeA}{\rad (A/AeA)}\in\per(A)$, if and only if $\pd_A\frac{A/AeA}{\rad(A/AeA)}<\infty$.
	
	(2) The triangle functor $Ae\otimes^L_{eAe}-: \sg(eAe)\to \sg(A)$ is a triangle equivalence if and only if the triangle functor $e^\op A^\op\otimes^L_{A^\op}-: \sg(A^\op)\to \sg(e^\op A^\op e^\op)$ is a triangle equivalence by Proposition~\ref{Prop-SE-OppositeAlg}, if and only if $\pd_{e^\op A^\op e^\op}e^\op A^\op<\infty$ and $\pd_{A^\op}\frac{A^\op/A^\op e^\op A^\op}{\rad(A^\op/A^\op e^\op A^\op)}<\infty$ by (1), if and only if $\pd_{(eAe)^\op} Ae<\infty$ and $\id_A\frac{A/AeA}{\rad(A/AeA)}<\infty$.
\end{proof}

\medspace

\noindent{\bf Singular equivalences of $n$-adjoint type induced by triangular matrix algebras.} The criteria for a triangular matrix algebra to be singularly equivalent to its diagonal subalgebras were given in \cite[Corollary 4.2]{Jin-Yang-Zhou23}. Now we recognize singular equivalences of 2 or 3-adjoint type. 
A singular equivalence of 2-adjoint type but not 3-adjoint type will be given in Example~\ref{Ex-SE2AT-NotEssSurjEHI}.

\begin{proposition} \label{Prop-TriMatAlg-SEAT}
	Let $B$ and $C$ be finite dimensional algebras, $M$ a finite dimensional $C$-$B$-bimodule, 
	$A=\begin{pmatrix} B&0\\ M&C \end{pmatrix}$, 
	$e_1=\begin{pmatrix} 1_B&0\\ 0&0 \end{pmatrix}$, and 
	$e_2=\begin{pmatrix} 0&0\\ 0&1_C \end{pmatrix}$. 
	
	{\rm (1)} If $\gl C<\infty$ then the triangle functor $e_1A\otimes^L_A-:\D(A)\to\D(B)$ induces a singular equivalence $e_1A\otimes^L_A-:\sg(A)\to\sg(B)$ of 2-adjoint type. 
	
	{\rm (2)} If $\gl C<\infty$ and $\pd_{B^\op}M<\infty$ then the triangle functor $Ae_1\otimes^L_B-:\D(B)\to\D(A)$ induces a singular equivalence $Ae_1\otimes^L_B-:\sg(B)\to\sg(A)$ of 3-adjoint type. 
	
	{\rm (3)} If $\gl B<\infty$ then the triangle functor $A/Ae_1A\otimes^L_A-:\D(A)\to\D(C)$ induces a singular equivalence $A/Ae_1A\otimes^L_A-:\sg(A)\to\sg(C)$ of 2-adjoint type. 
	
	{\rm (4)} If $\gl B<\infty$ and $\pd_CM<\infty$ then the triangle functor $A/Ae_1A\otimes^L_A-:\D(A)\to\D(C)$ induces a singular equivalence $A/Ae_1A\otimes^L_A-:\sg(A)\to\sg(C)$ of 3-adjoint type. 	
\end{proposition}

\begin{proof}
	Clearly, $e_1A=e_1Ae_1=\begin{pmatrix} B&0\\ 0&0 \end{pmatrix}\cong B$, 
	$Ae_1A=Ae_1=\begin{pmatrix}	B&0\\ M&0 \end{pmatrix}$, 
	$A/Ae_1A=\begin{pmatrix} 0&0\\ 0&C \end{pmatrix}\cong C$,
	$Ae_2=e_2Ae_2=\begin{pmatrix} 0&0\\ 0&C \end{pmatrix}\cong C$, 
	$Ae_2A=e_2A=\begin{pmatrix} 0&0\\ M&C \end{pmatrix}$, and
	$A/Ae_2A=\begin{pmatrix} B&0\\ 0&0 \end{pmatrix}\cong B$.
	Thus $Ae_1\otimes^L_{e_1Ae_1}e_1A\cong Ae_1=Ae_1A$ and $Ae_2\otimes^L_{e_2Ae_2}e_2A\cong e_2A=Ae_2A$. So $Ae_1A$ and $Ae_2A$ are stratifying ideals. Hence we have the following standard 2-recollement of derived categories of algebras \cite{Qin-Han16}.
	$$\begin{tikzcd} [column sep=95]
		\D(C) \arrow[r,"A/Ae_1A\otimes^L_C-\ \cong\ Ae_2\otimes^L_C-"] \arrow[r,shift right=12,"{\RHom_C(e_2A,-)}"] & \D(A) \arrow[r,"e_1A\otimes^L_A-\ \cong\ A/Ae_2A\otimes^L_A-"] \arrow[r,shift right=12,"{\RHom_A(A/Ae_2A,-)}"] \arrow[l,shift right=6,"{A/Ae_1A\otimes^L_A-}"'] \arrow[l,shift left=6,"e_2A\otimes^L_A-"'] & \D(B) \arrow[l,shift right=6,"Ae_1\otimes^L_B-"'] \arrow[l,shift left=6,"{A/Ae_2A\otimes^L_B-}"']
	\end{tikzcd}$$	
	
	(1) Due to $\gl C<\infty$, we have $\D^b(C)=\per(C)$ and $e_2A\in\per(C)$. Then $A/Ae_2A\in\per(A)$ by \cite[Proposition 3.3]{Angeleri Hugel-Konig-Liu-Yang17}. So both $A/Ae_2A\in\D^b(B\otimes A^\op)$ and $A/Ae_2A\in\D^b(A\otimes B^\op)$ are biperfect. By \cite[Theorem 1]{Han14}, we have the following recollement {\bf (R)}: 
	$$\begin{tikzcd} [column sep=95]
		\D(C\otimes A^\op) \arrow[r,"A/Ae_1A\otimes^L_C-\ \cong\ Ae_2\otimes^L_C-"] & \D(A^e) \arrow[r,"e_1A\otimes^L_A-\ \cong\ A/Ae_2A\otimes^L_A-"] \arrow[l,shift right=6,"{A/Ae_1A\otimes^L_A-}"'] \arrow[l,shift left=6,"e_2A\otimes^L_A-"'] & \D(B\otimes A^\op). \arrow[l,shift right=6,"Ae_1\otimes^L_B-"'] \arrow[l,shift left=6,"{A/Ae_2A\otimes^L_B-}"']
	\end{tikzcd}$$
	Then we have a triangle 
	$$Ae_2\otimes^L_Ce_2A \xrightarrow{\ev^A_{Ae_2}} A\xrightarrow{r^A_{A/Ae_2A}} A/Ae_2A\otimes^L_BA/Ae_2A \to$$
	in $\D(A^e)$. Thus $\Cone(r^A_{A/Ae_2A})\otimes^L_A\ol{A} \cong Ae_2\otimes^L_Ce_2A[1]\otimes^L_A\ol{A} \in Ae_2\otimes^L_C\D^b(C) = Ae_2\otimes^L_C\per(C) \subseteq \per(A)$. So the $A$-bimodule complex $\Cone(r^A_{A/Ae_2A})$ is $\otimes$-perfect by Proposition~\ref{Prop-BimodComplex-TensorPerfect} (2''). Moreover, by \cite[Theorem 1]{Han14}, we have the following recollement: 
	$$\begin{tikzcd} [column sep=95]
		\D(C\otimes B^\op) \arrow[r,"A/Ae_1A\otimes^L_C-\ \cong\ Ae_2\otimes^L_C-"] & \D(A\otimes B^\op) \arrow[r,"e_1A\otimes^L_A-\ \cong\ A/Ae_2A\otimes^L_A-"] \arrow[l,shift right=6,"{A/Ae_1A\otimes^L_A-}"'] \arrow[l,shift left=6,"e_2A\otimes^L_A-"'] & \D(B^e). \arrow[l,shift right=6,"Ae_1\otimes^L_B-"'] \arrow[l,shift left=6,"{A/Ae_2A\otimes^L_B-}"']
	\end{tikzcd}$$
	Thus $\ev^B_{A/Ae_2A}:A/Ae_2A\otimes^L_AA/Ae_2A\to B$ is an isomorphism in $\D(B^e)$. Hence the $B$-bimodule complex $\Cone(\ev^B_{A/Ae_2A})=0$ is $\otimes$-perfect. By Theorem~\ref{Thm-Exist-SingEquAdjTye-FD}, the triangle functor $e_1A\otimes^L_A-:\D(A)\to\D(B)$ induces a singular equivalence $e_1A\otimes^L_A-:\sg(A)\to\sg(B)$ of 2-adjoint type.
	
	(2) Due to $\pd_{B^\op}M<\infty$, we have $Ae_1\in\per(B^\op)$. Then $Ae_1\in\D^b(A\otimes B^\op)$ is biperfect. Thus the triangle functor $Ae_1\otimes^L_B-:\D(B)\to\D(A)$ induces a singular equivalence $Ae_1\otimes^L_B-:\sg(B)\to\sg(A)$ of 3-adjoint type by (1) and Theorem~\ref{Thm-Exist-SingEquAdjTye-FD}. 
	
	(3) Due to $\gl B<\infty$, we have $\D^b(B)=\per(B)$ and $Ae_1\in\per(B^\op)$. Then $A/Ae_1A\in\per(A^\op)$. So both $A/Ae_1A\in\D^b(C\otimes A^\op)$ and $A/Ae_1A\in\D^b(A\otimes C^\op)$ are biperfect. From the recollement {\bf (R)}, we obtain a triangle 
	$$Ae_1\otimes^L_Be_1A \xrightarrow{\ev^A_{Ae_1}} A\xrightarrow{r^A_{A/Ae_1A}} A/Ae_1A\otimes^L_CA/Ae_1A \to$$
	in $\D(A^e)$. Thus $\Cone(r^A_{A/Ae_1A})\otimes^L_A\ol{A} \cong Ae_1\otimes^L_Be_1A[1]\otimes^L_A\ol{A} \in Ae_1\otimes^L_B\D^b(B) = Ae_1\otimes^L_B\per(B) \subseteq \per(A)$. So the $A$-bimodule complex $\Cone(r^A_{A/Ae_1A})$ is $\otimes$-perfect by Proposition~\ref{Prop-BimodComplex-TensorPerfect} (2''). Moreover, by \cite[Theorem 1]{Han14}, we have the following recollement: 
	$$\begin{tikzcd} [column sep=95]
		\D(C^e) \arrow[r,"A/Ae_1A\otimes^L_C-\ \cong\ Ae_2\otimes^L_C-"] & \D(A\otimes C^\op) \arrow[r,"e_1A\otimes^L_A-\ \cong\ A/Ae_2A\otimes^L_A-"] \arrow[l,shift right=6,"{A/Ae_1A\otimes^L_A-}"'] \arrow[l,shift left=6,"e_2A\otimes^L_A-"'] & \D(B\otimes C^\op). \arrow[l,shift right=6,"Ae_1\otimes^L_B-"'] \arrow[l,shift left=6,"{A/Ae_2A\otimes^L_B-}"']
	\end{tikzcd}$$
	Thus $\ev^C_{A/Ae_1A}:A/Ae_1A\otimes^L_AA/Ae_1A\to C$ is an isomorphism in $\D(C^e)$. Hence the $C$-bimodule complex $\Cone(\ev^C_{A/Ae_1A})=0$ is $\otimes$-perfect. By Theorem~\ref{Thm-Exist-SingEquAdjTye-FD}, the triangle functor $A/Ae_1A\otimes^L_A-:\D(A)\to\D(C)$ induces a singular equivalence $A/Ae_1A\otimes^L_A-:\sg(A)\to\sg(C)$ of 2-adjoint type.
	
	(4) Due to $\pd_CM<\infty$, we have $e_2A\in\per(C)$. Then $e_2A\in\D^b(C\otimes A^\op)$ is biperfect. Thus the triangle functor $A/Ae_1A\otimes^L_A-:\D(A)\to\D(C)$ induces a singular equivalence $A/Ae_1A\otimes^L_A-:\sg(A)\to\sg(C)$ of 3-adjoint type by (3) and Theorem~\ref{Thm-Exist-SingEquAdjTye-FD}. 
\end{proof}

The following example gives singular equivalences of $n$-adjoint type for all $n\ge 1$. Indeed, all derived discrete algebras are such kind of examples by \cite[Proposition 4]{Qin16}.

\begin{example}{\rm	 
		Let $A$ be the 5-dimensional bound quiver algebra $kQ/I$ where the quiver $Q$ is
		$$\xymatrix{
			1 \ar[r]^-y & 2 \ar@(ur,dr)[]^x			 
		}$$		
		\noindent and the admissible ideal $I:=(x^2)$. 
		Then $e_2Ae_2=Ae_2=k[x]/(x^2)$ and $Ae_2A=e_2A=e_2Ae_2\oplus e_2Ae_1$ where $e_2Ae_1\cong e_2Ae_2$ as left $e_2Ae_2$-modules. Thus $\pd_{e_2Ae_2}e_2A=0$ and $\pd_{(e_2Ae_2)^\op}Ae_2=0$. Moreover, $A/Ae_2A\cong k$ is a simple algebra.  
		Since $Ae_2\otimes^L_{e_2Ae_2}e_2A \cong Ae_2\otimes_{e_2Ae_2}e_2A \cong e_2A = Ae_2A$ in $\D(A^e)$, $Ae_2A$ is a stratifying ideal of $A$, and the factor map $A\twoheadrightarrow A/Ae_2A$ is a homological epimorphism of algebras. Then we have the following recollement:
		$$\begin{tikzcd}[column sep=100]
			\D(A/Ae_2A) \arrow[r,"{A/Ae_2A\otimes^L_{A/Ae_2A}-}"] & \D(A) \arrow[r,"{e_2A\otimes^L_A-}"] \arrow[l,shift left=6,"{\RHom_A(A/Ae_2A,-)}"'] \arrow[l,shift right=6,"A/Ae_2A\otimes^L_A-"'] & \D(e_2Ae_2). \arrow[l,shift left=6,"{\RHom_{e_2Ae_2}(e_2A,-)}"'] \arrow[l,shift right=6,"{Ae_2\otimes^L_{e_2Ae_2}-}"']
		\end{tikzcd}$$
		The triangle functor $_AA/Ae_2A\otimes^L_{A/Ae_2A}-:\D^b(A/Ae_2A)\to\D^b(A)$ is fully faithful, so it is an eventually homological isomorphism, but the triangle functor $_AA/Ae_2A\otimes^L_{A/Ae_2A}-:\sg(A/Ae_2A)\to\sg(A)$ is not a singular equivalence due to $\gl A/Ae_2A=0$ and $\gl A=\infty$.
		
		Note that $\frac{A/Ae_2A}{\rad(A/Ae_2A)}\cong A/Ae_2A\cong S_1=I_1$ is both simple and injective left $A$-module corresponding to the vertex 1. Thus $\id_A\frac{A/Ae_2A}{\rad(A/Ae_2A)} =\id_AI_1=0$. Since $0\to P_2\to P_1\to S_1\to 0$ is a minimal projective resolution of $S_1$, we have $\pd_A\frac{A/Ae_2A}{\rad(A/Ae_2A)}=\pd_AS_1=1$. 
		Hence the triangle functor $e_2A\otimes^L_A-:\D^b(A)\to\D^b(e_2Ae_2)$ is an eventually homological isomorphism by Proposition~\ref{Prop-Exist-EvHomIso-Idemp}, and the triangle functor $e_2A\otimes^L_A-:\sg(A)\to\sg(e_2Ae_2)$ is a singular equivalence by Corollary~\ref{Cor-Exist-SingEqu-Idemp}. 
		
		Since the algebra $A$ is just the derived discrete algebra $\Lambda(1,1,1)$, the recollement above can be extended downward to an $n$-recollement for all $n\ge 1$ by \cite[Proposition 4]{Qin16}. Then all triangle functors in the right half part of the $n$-recollement are singular equivalences of $m$-adjoint type for all $m\ge 1$ by Theorem~\ref{Thm-EssSurjEHI=SingEqWithLeftAdj}.
}\end{example}

\subsection{Applications to reduce homological conjectures}

In this subsection, we apply singular equivalences of 2 or 3-adjoint type to reduce homological conjectures and transfer homological properties.

\medspace

\noindent{\bf Finitistic dimension conjecture.} (\cite[Page 487]{Bass60}) The finitistic dimension of an algebra $A$, i.e., the supremum of the projective dimensions of finitely generated $A$-modules with finite projective dimension, is finite.

Finitistic dimension conjecture is stronger than Nakayama conjecture and Gorenstein symmetry conjecture.
It holds true for monomial algebras \cite{Green-Kirkman-Kuzmanovich91}, radical cube zero algebras \cite{Green-Zimmermann-Huisgen91}, etc. We refer to \cite{Zimmermann-Huisgen94} for more history. Moreover, it can be reduced by recollements of derived categories \cite{Happel93,Chen-Xi17} and singular equivalences of Morita type with levels \cite{Wang15}.

The following result implies that the finitistic dimension conjecture can be reduced by singular equivalences of $2$-adjoint type.

\begin{theorem} \label{Thm-SE2AT-FDC}
	Let two finite dimensional algebras $A$ and $B$ be singularly equivalent of 2-adjoint type. Then $A$ satisfies the finitistic dimension conjecture if and only if $B$ does. 
\end{theorem}

\begin{proof}
	Since $A$ and $B$ are singularly equivalent of 2-adjoint type, there exists a biperfect $A$-$B$-bimodule complex $X$ which induces a singular equivalence $X\otimes^L_B-:\sg(B)\to\sg(A)$ with a right adjoint \linebreak $\RHom_A(X,-)\cong \RHom_A(X,A)\otimes^L_A-:\sg(A)\to\sg(B)$. By Theorem~\ref{Thm-Exist-SingEquRightAdjoit-FD}, the bimodule complexes $\Cone(r^B_X:B\to \RHom_A(X,A)\otimes^L_AX)$ and $\Cone(\ev^A_X:X\otimes^L_B\RHom_A(X,A)\to A)$ are $\otimes$-perfect.
	
	{\it Necessity.} Assume that $A$ satisfies the finitistic dimension conjecture, i.e., its finitistic dimension ${\rm fin.dim} A=m<\infty$. We want to prove ${\rm fin.dim} B<\infty$. For all $N\in\mod B$ with $\pd_BN<\infty$, we need to show that $\pd_BN$ has a common upper bound. From the triangle $B\to \RHom_A(X,A)\otimes^L_AX\to \Cone(r^B_X)\to$ in $\D(B^e)$, we obtain a triangle $$\ol{B}\otimes^L_BN\to \ol{B}\otimes^L_B\RHom_A(X,A)\otimes^L_AX\otimes^L_BN\to \ol{B}\otimes^L_B\Cone(r^B_X)\otimes^L_BN\to.$$
	
	Due to $_B\RHom_A(X,A)\in\per(B)$, we have $\ol{B}\otimes^L_B\RHom_A(X,A)_A\in \D^b(A^\op)$ by Lemma~\ref{Lem-LeftModComplex-Perfect}. Then $\ol{B}\otimes^L_B\RHom_A(X,A)$ is isomorphic to a bounded complex $L^l\to\cdots\to L^{l'}$ in $\D^b(A^\op)$ where $L^i\in\mod A^\op$ for all $l\le i\le l'$.
	Due to $X_B\in\per(B^\op)$, $X$ is isomorphic to a bounded complex $P^p\to\cdots\to P^{p'}$ in $\per(B^\op)$ where $P^i\in\proj B^\op$ for all $p\le i\le p'$. Then $X\otimes^L_BN$ is isomorphic to a bounded complex $P^p\otimes_BN\to\cdots\to P^{p'}\otimes_BN$ in $\D^b(k)$. Thus $H^i(X\otimes^L_BN)=0$ for all $i\le p-1$ and $i\ge p'+1$. Let $Q:=(\cdots\to Q^{p'-1}\xrightarrow{d^{p'-1}_Q} Q^{p'})$ be a minimal semi-projective resolution of $X\otimes^L_BN\in\D^b(A)$ where $Q^i\in\proj A$ for all $i\le p'$. Since $_AX_B$ is biperfect, we have $_AX\otimes^L_BN\in\per(A)$ by Lemma~\ref{Lem-RightModComplex-Perfect}. From the triangle $\Im d^{p-1}_Q[p]\to \tau^{\ge p}Q\to X\otimes^L_BN\to$ in $\D(A)$, where $\tau^{\ge p}Q:=(Q^p\to\cdots\to Q^{p'})$ is the brutal truncation of $Q$, we obtain $\Im d^{p-1}_Q\in\per(A)$. Thanks to ${\rm fin.dim} A=m<\infty$, we have $\pd_A\Im d^{p-1}_Q\le m$. Then $Q^i=0$ for all $i\le p-m-2$. Thus $\ol{B}\otimes^L_B\RHom_A(X,A)\otimes^L_AX\otimes^L_BN$ is isomorphic to a bounded complex $L^l\otimes_AQ^{p-m-1}\to\cdots\to L^{l'}\otimes_AQ^{p'}$. Hence $H^i(\ol{B}\otimes^L_B\RHom_A(X,A)\otimes^L_AX\otimes^L_BN)=0$ for all $i\le l+p-m-2$.  
	
	Since $\Cone(r^B_X)$ is $\otimes$-perfect, we have $\ol{B}\otimes^L_B\Cone(r^B_X)\in\per(B^\op)$. Then $\ol{B}\otimes^L_B\Cone(r^B_X)$ is isomorphic to a bounded complex $R^r\to\cdots\to R^{r'}$ where $R^i\in \proj(B^\op)$ for all $r\le i\le r'$. Thus $\ol{B}\otimes^L_B\Cone(r^B_X)\otimes^L_BN$ is isomorphic to the bounded complex $R^r\otimes_BN\to\cdots\to R^{r'}\otimes_BN$ in $\D^b(k)$. Hence $H^i(\ol{B}\otimes^L_B\Cone(r^B_X)\otimes^L_BN)=0$ for all $i\le r-1$. 
	
	From the triangle $\ol{B}\otimes^L_BN\to \ol{B}\otimes^L_B\RHom_A(X,A)\otimes^L_AX\otimes^L_BN\to \ol{B}\otimes^L_B\Cone(r^B_X)\otimes^L_BN\to$, we get a long exact sequence $\cdots\to H^{i-1}(\ol{B}\otimes^L_B\Cone(r^B_X)\otimes^L_BN) \to H^i(\ol{B}\otimes^L_BN) \to H^i(\ol{B}\otimes^L_B\RHom_A(X,A)\otimes^L_AX\otimes^L_BN) \to \cdots.$ By the above analyses, we have $H^i(\ol{B}\otimes^L_BN)=0$ for all $i\le \min\{r,l+p-m-2\}$. Thus $\pd_BN\le\max\{-r-1,m-l-p+1\}$. Hence ${\rm fin.dim} B\le \max\{-r-1,m-l-p+1\}<\infty$.

	{\it Sufficiency.} Assume ${\rm fin.dim} B=n<\infty$. We want to prove ${\rm fin.dim} A<\infty$. For all $M\in\mod A$ with $\pd_AM<\infty$, we need to show that $\pd_AM$ has a common upper bound. From the triangle $X\otimes^L_B\RHom_A(X,A)\to A\to \Cone(\ev^A_X)\to$ in $\D(A^e)$, we obtain a triangle $$\ol{A}\otimes^L_AX\otimes^L_B\RHom_A(X,A)\otimes^L_AM\to \ol{A}\otimes^L_AM\to \ol{A}\otimes^L_A\Cone(\ev^A_X)\otimes^L_AM\to.$$
	
	Due to $_AX\in\per(A)$, we have $\ol{A}\otimes^L_AX_B\in \D^b(B^\op)$ by Lemma~\ref{Lem-LeftModComplex-Perfect}. Then $\ol{A}\otimes^L_AX$ is isomorphic to a bounded complex $L^l\to\cdots\to L^{l'}$ in $\D^b(B^\op)$ where $L^i\in\mod B^\op$ for all $l\le i\le l'$.
	Due to $\RHom_A(X,A)_A\in\per(A^\op)$, $\RHom_A(X,A)$ is isomorphic to a bounded complex $P^p\to\cdots\to P^{p'}$ in $\per(A^\op)$ where $P^i\in\proj A^\op$ for all $p\le i\le p'$. Then $\RHom_A(X,A)\otimes^L_AM$ is isomorphic to a bounded complex $P^p\otimes_AM\to\cdots\to P^{p'}\otimes_AM$ in $\D^b(k)$. Thus $H^i(\RHom_A(X,A)\otimes^L_AM)=0$ for all $i\le p-1$ and $i\ge p'+1$. Let $Q:=(\cdots\to Q^{p'-1}\xrightarrow{d^{p'-1}_Q} Q^{p'})$ be a minimal semi-projective resolution of $\RHom_A(X,A)\otimes^L_AM\in\D^b(B)$ where $Q^i\in\proj B$ for all $i\le p'$. Since $_B\RHom_A(X,A)_A$ is biperfect, we have $_B\RHom_A(X,A)\otimes^L_AM\in\per(B)$ by Lemma~\ref{Lem-RightModComplex-Perfect}. From the triangle $\Im d^{p-1}_Q[p]\to \tau^{\ge p}Q\to \RHom_A(X,A)\otimes^L_AM\to$, we obtain $\Im d^{p-1}_Q\in\per(B)$. Thanks to ${\rm fin.dim} B=n<\infty$, we have $\pd_B\Im d^{p-1}_Q\le n$. Then $Q^i=0$ for all $i\le p-n-2$. Thus $\ol{A}\otimes^L_AX\otimes^L_B\RHom_A(X,A)\otimes^L_AM$ is isomorphic to a bounded complex $L^l\otimes_BQ^{p-n-1}\to\cdots\to L^{l'}\otimes_BQ^{p'}$. Hence $H^i(\ol{A}\otimes^L_AX\otimes^L_B\RHom_A(X,A)\otimes^L_AM)=0$ for all $i\le l+p-n-2$.  
	
	Since $\Cone(\ev^A_X)$ is $\otimes$-perfect, we have $\ol{A}\otimes^L_A\Cone(\ev^A_X)\in\per(A^\op)$. Then $\ol{A}\otimes^L_A\Cone(\ev^A_X)$ is isomorphic to a bounded complex $R^r\to\cdots\to R^{r'}$ where $R^i\in \proj(A^\op)$ for all $r\le i\le r'$. Thus $\ol{A}\otimes^L_A\Cone(\ev^A_X)\otimes^L_AM$ is isomorphic to the bounded complex $R^r\otimes_AM\to\cdots\to R^{r'}\otimes_AM$. Hence $H^i(\ol{A}\otimes^L_A\Cone(\ev^A_X)\otimes^L_AM)=0$ for all $i\le r-1$. 
	
	From the triangle $\ol{A}\otimes^L_AX\otimes^L_B\RHom_A(X,A)\otimes^L_AM\to \ol{A}\otimes^L_AM\to \ol{A}\otimes^L_A\Cone(\ev^A_X)\otimes^L_AM\to$, we get a long exact sequence $\cdots\to H^i(\ol{A}\otimes^L_AX\otimes^L_B\RHom_A(X,A)\otimes^L_AM) \to H^i(\ol{A}\otimes^L_AM) \to H^i(\ol{A}\otimes^L_A\Cone(\ev^A_X)\otimes^L_AM) \to \cdots.$ By the above analyses, we have $H^i(\ol{A}\otimes^L_AM)=0$ for all $i\le \min\{r,l+p-n-2\}$. Thus $\pd_AM\le\max\{-r-1,n-l-p+1\}$. Hence ${\rm fin.dim} A\le \max\{-r-1,n-l-p+1\}<\infty$.
\end{proof}

\noindent{\bf Han's property.} (\cite[Definition 3.9]{Cruz23}) We say that an algebra $A$ satisfies {\it Han's property} if its global dimension is finite once its $n$-th Hochschild homology group $HH_n(A)$ vanishes for $n\gg 0$.

Han's conjecture \cite{Han06} says that all finite dimensional $k$-algebras $A$ with $A/\rad A$ being separable over $k$ satisfy Han's property. It holds true for monomial algebras \cite{Han06}, exterior algebras \cite{Xu-Han07}, algebras with 2-truncated cycles \cite{Bergh-Han-Madsen12}, etc. We refer to \cite{Cruz23} for a survey on Han's conjecture. Recently, a counterexample of Han's conjecture was given by Kong, Liu and Shen in \cite{Kong-Liu-Shen26}. Han's property can be transferred by recollements \cite{Wang-Xu-Zhang-Zhou24} and extensions of algebras \cite{Cibils-Lanzilotta-Marcos-Solotar22,Qin-Xu-Zhang-Zhou24,Iusenko-MacQuarrie25}. 

The following result implies that Han's property can be transferred by singular equivalences of 2-adjoint type.

\begin{theorem} \label{Thm-SE2AT-HH-Han}
	Let finite dimensional algebras $A$ and $B$ be singularly equivalent of 2-adjoint type, and $A/\rad A$ and $B/\rad B$ be separable over $k$. Then the following statements hold:
	
	{\rm(1)} $HH_i(A) \cong HH_i(B)$ for $i\gg 0$.
	
	{\rm(2)} $A$ satisfies Han's property if and only if $B$ does.
\end{theorem}

\begin{proof} 
	(1) Since $A$ and $B$ are singularly equivalent of 2-adjoint type, there is a complex $X_1\in\D(A\otimes B^\op)$ which induces an adjoint pair $X_1\otimes^L_B-: \sg(B) \rightleftarrows \sg(A) :X_2\otimes^L_A-$ where $X_2:=\RHom_A(X_1,A)$.
	
	We have an adjoint pair $X_1\otimes^L_B-:\D(B\otimes B^\op)\rightleftarrows\D(A\otimes B^\op):X_2\otimes^L_A-$. Its unit induces a triangle $B\to X_2\otimes^L_AX_1\to \Cone(r^B_{X_1})\to$ in $\D(B^e)$, then a triangle $B\otimes^L_{B^e}B\to B\otimes^L_{B^e}(X_2\otimes^L_AX_1)\to B\otimes^L_{B^e}\Cone(r^B_{X_1})\to$ in $\D(k)$. By Theorem~\ref{Thm-Exist-SingEquRightAdjoit-FD}, $\Cone(r^B_{X_1})$ is $\otimes$-perfect in $\D(B^e)$. By the assumption that $A/\rad A$ and $B/\rad B$ are separable over $k$ and Proposition~\ref{Prop-HomSmooth-TensorPerfect-BimodComplex}, $\Cone(r^B_{X_1})$ is perfect in $\D(B^e)$. Thus $B\otimes^L_{B^e}\Cone(r^B_{X_1})\in \D^b(k)$. Hence $H_i(B\otimes^L_{B^e}B)\cong H_i(B\otimes^L_{B^e}(X_2\otimes^L_AX_1))$ for $i\gg 0$.
	
	We have an adjoint pair $X_1\otimes^L_B-:\D(B\otimes A^\op)\rightleftarrows\D(A\otimes A^\op):X_2\otimes^L_A-$. Its counit induces a triangle $X_1\otimes^L_BX_2\to A\to \Cone(\ev^A_{X_1})\to$ in $\D(A^e)$, then a triangle $A\otimes^L_{A^e}(X_1\otimes^L_BX_2)\to A\otimes^L_{A^e}A\to A\otimes^L_{A^e}\Cone(\ev^A_{X_1})\to$ in $\D(k)$. By Theorem~\ref{Thm-Exist-SingEquRightAdjoit-FD}, $\Cone(\ev^A_{X_1})$ is $\otimes$-perfect in $\D(A^e)$. By the assumption that $A/\rad A$ and $B/\rad B$ are separable over $k$ and Proposition~\ref{Prop-HomSmooth-TensorPerfect-BimodComplex}, $\Cone(\ev^A_{X_1})$ is perfect in $\D(A^e)$. Thus $A\otimes^L_{A^e}\Cone(\ev^A_{X_1})\in \D^b(k)$. Hence $H_i(A\otimes^L_{A^e}A)\cong H_i(A\otimes^L_{A^e}(X_1\otimes^L_BX_2))$ for $i\gg 0$.
	
	Clearly, $B\otimes^L_{B^e}(X_2\otimes^L_AX_1) \cong A\otimes^L_{A^e}(X_1\otimes^L_BX_2)$. Thus $HH_i(A) \cong HH_i(B)$ for $i\gg 0$.
	
	(2) Due to $\sg(A)\simeq \sg(B)$, we know $\gl A<\infty$ if and only if $\gl B<\infty$. It follows from (1) that $A$ satisfies Han's property if and only if $B$ does.
\end{proof}

\begin{remark}{\rm
		Alternately, Theorem~\ref{Thm-SE2AT-HH-Han} can be proved as follows: It follows from \cite[Theorem 1.1]{Qin22} that $A$ and $B$ are singularly equivalent with level. By \cite[Proposition 3.8]{Wang15}, we have $HH_i(A) \cong HH_i(B)$ for $i> 0$. Furthermore, $A$ satisfies Han's property if and only if $B$ does.
}\end{remark}

The direct corollary of Theorem~\ref{Thm-SE2AT-HH-Han} and Corollary~\ref{Cor-SingEqu-AlgExt} is the following result, which had been proved previously in different ways.

\begin{corollary} {\rm(\cite[Theorem 4.6]{Cibils-Lanzilotta-Marcos-Solotar22}, \cite[Corollary 6.18]{Iusenko-MacQuarrie25}, \cite[Theorem 4.1 (2)]{Qin-Xu-Zhang-Zhou24})}
	Let $B\subseteq A$ be a bounded or strongly proj-bounded extension of finite dimensional algebras, and $A/\rad A$ and $B/\rad B$ be separable over $k$. Then $A$ satisfies Han's property if and only if $B$ does. 
\end{corollary}

\begin{proposition} \label{Prop-SE2AT-G-GSC}
	Let $A$ and $B$ be finite dimensional algebras, and $X\otimes^L_B-: \sg(B)\to\sg(A)$ a singular equivalence of $2$-adjoint type induced by $X\in\D(A\otimes B^\op)$. Then the following statements hold:
	
	{\rm(1)} If $B$ is a Gorenstein algebra then so is $A$. 
	
	{\rm(2)} If $A$ satisfies the Gorenstein symmetry conjecture then so does $B$. 
\end{proposition}

\begin{proof} 
	Since $X\otimes^L_B-: \sg(B)\to\sg(A)$ is a singular equivalence of $2$-adjoint type, we have an adjoint 5-tuple 
	$$\begin{tikzcd}[column sep=110]
		\D(B) \arrow[r,"{X\otimes^L_B-}"] \arrow[r,shift right=12,"{X_3\otimes^L_B-}"] & \D(A) \arrow[l,shift right=6,"{\RHom_{B^\op}(X,B)\otimes^L_A-}"'] \arrow[l,shift left=6,"{X_2\otimes^L_A-}"'] \arrow[l,shift left=18,"{\RHom_A(X_3,A)\otimes^L_A-}"']
	\end{tikzcd}$$	
	where $X_2:=\RHom_A(X,A)$ and $ X_3:=\RHom_B(X_2,B)$, and both $X$ and $X_2$ are biperfect.

{\it Claim 1.} $K^b(\proj B)\subseteq K^b(\inj B) \Rightarrow K^b(\proj A)\subseteq K^b(\inj A)$.
	
	For any $Y\in K^b(\proj A)$, we need to show $Y\in K^b(\inj A)$. Let $\varepsilon$ be the unit of the adjoint pair
$(X_2\otimes^L_A-, X_3\otimes^L_B-)$. Then there is a triangle $$Y \xrightarrow{\varepsilon _Y} X_3\otimes^L_BX_2\otimes^L_AY \to \Cone(\varepsilon _Y)\to$$ in $\D(A)$.
By Lemma~\ref{Lem-Cone-Cone}, we have that $\mathrm{Cone}(\varepsilon_Y)\cong \Cone(r^A_{X_2})\otimes^L_AY$.
Moreover, it follows from Theorem~\ref{Thm-Exist-SingEquRightAdjoit-FD} that $\Cone(r^A_{X_2})$ 
is $\otimes$-coperfect, which implies that $\Cone(r^A_{X_2})\otimes^L_AY \in K^b(\inj A)$. On the other hand,
the functor $X_2\otimes^L_A-:\D(A)\to\D(B)$ sends $K^b(\proj A)$ into $K^b(\proj B)$, and then $X_2\otimes^L_AY\in K^b(\proj B)\subseteq K^b(\inj B)$.
Since the functor $X_3\otimes^L_B-:\D(B)\to\D(A)$ sends $K^b(\inj B)$ into $K^b(\inj A)$, we infer that $X_3\otimes^L_BX_2\otimes^L_AY \in K^b(\inj A)$.
Now the triangle above implies that $Y\in K^b(\inj A)$.

{\it Claim 2.} $K^b(\inj B)\subseteq K^b(\proj B) \Rightarrow K^b(\inj A)\subseteq K^b(\proj A)$.

For any $Y\in K^b(\inj A)$, we need to show $Y\in K^b(\proj A)$. Let $\eta$ be the counit of the adjoint pair
$(X\otimes^L_B-, X_2\otimes^L_A-)$. Then there is a triangle $$X\otimes^L_BX_2\otimes^L_AY  \xrightarrow{\eta _Y} Y \to \Cone(\eta _Y)\to$$ in $\D(A)$.
By Lemma~\ref{Lem-Cone-Cone}, we have that $\mathrm{Cone}(\eta _Y)\cong \Cone(\ev _X^A)\otimes^L_AY$.
Moreover, it follows from Theorem~\ref{Thm-Exist-SingEquRightAdjoit-FD} that $\Cone(\ev _X^A)$ 
is $\otimes$-perfect, which implies that $\Cone(\ev _X^A)\otimes^L_AY \in K^b(\proj A)$. On the other hand,
the functor $X_2\otimes^L_A-:\D(A)\to\D(B)$ sends $K^b(\inj A)$ into $K^b(\inj B)$, and then $X_2\otimes^L_AY\in K^b(\inj B)\subseteq K^b(\proj B)$.
Since the functor $X\otimes^L_B-:\D(B)\to\D(A)$ sends $K^b(\proj B)$ into $K^b(\proj A)$, we infer that $X\otimes^L_BX_2\otimes^L_AY  \in K^b(\proj A)$.
Now the triangle above implies that $Y\in K^b(\proj A)$.

(1) If $B$ is Gorenstein, then $K^b(\proj B)\subseteq K^b(\inj B)$ and $K^b(\inj B)\subseteq K^b(\proj B)$.
It follows from the claims 1 and 2 that $K^b(\proj A)\subseteq K^b(\inj A)$ and $K^b(\inj A)\subseteq K^b(\proj A)$,
that is, $A$ is a Gorenstein algebra.

(2) If $A$ satisfies Gorenstein symmetry conjecture, we need to show so does $B$. 

If $K^b(\proj B)\subseteq K^b(\inj B)$ then we have $K^b(\proj A)\subseteq K^b(\inj A)$ by Claim 1, and thus $K^b(\inj A)\subseteq K^b(\proj A)$ since $A$ satisfies Gorenstein symmetry conjecture. Then $A$ is a Gorenstein algebra.
As the functor $X_3\otimes^L_B-$ sends $K^b(\inj B)$ into $K^b(\inj A)$, we have that $X_3\cong X_3\otimes^L_BB \in X_3\otimes^L_BK^b(\proj B)
\subseteq  X_3\otimes^L_BK^b(\inj B) \subseteq  K^b(\inj A)\subseteq  K^b(\proj A)$. Hence $X_3$ is biperfect and the complex $X_2$ induces
a singular equivalence $X_2\otimes^L_A-:\sg(A)\to\sg(B)$ of $2$-adjoint type. It follows from the statement (1) that $B$ is Gorenstein, and thus $K^b(\inj B)\subseteq K^b(\proj B)$. 

If $K^b(\inj B)\subseteq K^b(\proj B)$ then we have $K^b(\inj A)\subseteq K^b(\proj A)$ by Claim 2, and thus $K^b(\proj A)\subseteq K^b(\inj A)$ since $A$ satisfies Gorenstein symmetry conjecture. Then $A$ is a Gorenstein algebra. 
As the functor $X\otimes^L_B-$ sends $K^b(\proj B)$ into $K^b(\proj A)$, we have that $X\otimes^L_BB^* \in X\otimes^L_BK^b(\inj B)
\subseteq  X\otimes^L_BK^b(\proj B) \subseteq  K^b(\proj A)\subseteq  K^b(\inj A)$. Then $\RHom_{B^\op}(X,B)\cong (X\otimes^L_BB^*)^\ast
\in K^b(\proj A^{\op})$, and thus $\RHom_{B^\op}(X,B)$ is a biperfect complex of $B$-$A$-bimodules, which induces a singular equivalence
$\RHom_{B^\op}(X,B)\otimes^L_A-:\sg(A)\to\sg(B)$ of $2$-adjoint type. It follows from the statement (1) that $B$ is Gorenstein, and thus $K^b(\proj B)\subseteq K^b(\inj B)$.
\end{proof}

\begin{remark}{\rm
	A singular equivalence $X\otimes^L_B-: \sg(B)\to\sg(A)$ of 2-adjoint type induced by $X\in\D(A\otimes B^\op)$ cannot transfer the Gorensteinness of $A$ to that of $B$ in general by \cite[Example 5.5]{Psaroudakis-Skartsaeterhagen-Solberg14}. 
}\end{remark}

\medspace

\noindent{\bf Auslander-Reiten's conjecture.} (\cite[Page 72]{Auslander-Reiten75-GNC}) If a finitely generated module $X$ over an Artin algebra $A$ satisfies $\Ext^i_A(X, X\oplus A) = 0$ for all $i \ge 1$ then $X$ is projective.

\begin{theorem} \label{Thm-SE3AT-G-GSC-ARC}
	Let finite dimensional algebras $A$ and $B$ be singularly equivalent of $3$-adjoint type. Then the following statements hold:
	
	{\rm(1)} $A$ is Gorenstein if and only if so is $B$. 
	
	{\rm(2)} $A$ satisfies Gorenstein symmetry conjecture if and only if $B$ does. 
	
	{\rm(3) \cite{Chen-Hu-Qin-Wang23}} $A$ satisfies Auslander-Reiten's conjecture if and only if $B$ does.	
\end{theorem}

\begin{proof} 
Let $X\otimes^L_B-: \sg(B)\to\sg(A)$ be a singular equivalence of $3$-adjoint type induced by $X\in\D(A\otimes B^\op)$. 
Then $X\otimes^L_B-: \sg(B)\to\sg(A)$ is also a singular equivalence of $2$-adjoint type. 
Applying Proposition~\ref{Prop-SE2AT-G-GSC}, we obtain that the Gorensteinness of $B$ implies that of $A$, and the Gorenstein symmetry conjecture holds for $B$ whenever it holds for $A$. On the other hand, the $B$-$A$-bimodule complex $X_2:=\RHom_A(X,A)$ induces a singular equivalence
$X_2\otimes^L_A-:\sg(A)\to\sg(B)$ of $2$-adjoint type. Applying Proposition~\ref{Prop-SE2AT-G-GSC} again, we obtain that the Gorensteinness of 
$A$ implies that of $B$, and the Gorenstein symmetry conjecture holds for 
$A$ whenever it holds for $B$. Therefore, the statements (1) and (2) hold,
and (3) follows from \cite[Theorem 1.1 or Theorem 1.3]{Chen-Hu-Qin-Wang23}.
\end{proof}

\begin{remark}{\rm
	Gorenstein symmetry conjecture can be reduced by essentially surjective eventually homological isomorphisms (Theorem~\ref{Thm-EHI-GorSymConj}) and singular equivalences of $3$-adjoint type (Theorem~\ref{Thm-SE3AT-G-GSC-ARC}) respectively. These two theorems are independent (Ref. Remark~\ref{Rem-ESEHI-Gorenstein-General}) though any standard essentially surjective eventually homological isomorphism can induce a singular equivalence of 2-adjoint type (Theorem~\ref{Thm-EssSurjEHI=SingEqWithLeftAdj}).  
}\end{remark}

\medspace

\noindent{\bf Happel's property.} (\cite[Definition 3.9]{Cruz23}) We say that a finite dimensional algebra $A$ satisfies {\it Happel's property} if its global dimension is finite once its $n$-th Hochschild cohomology group $HH^n(A)$ vanishes for $n\gg 0$.

All commutative algebras \cite{Avramov-Iyengar05}, truncated quiver algebras \cite{Xu-Han-Jiang07} and exterior algebras \cite{Xu-Han07} satisfy  Happel's property. The first example that does not satisfy Happel's property was given in \cite{Buchweitz-Green-Madsen-Solberg05}. 

\medspace

\noindent{\bf Fg condition.} \cite{Erdmann-Holloway-Snashall-Solberg-Taillefer04,Snashall-Solberg04}
We say that an algebra $A$ satisfies the {\it Fg condition} if its Hochschild cohomology ring $HH^\bullet(A)$ is Noetherian and its Yoneda algebra $\Ext^\bullet_A(A/\rad A,   A/\rad A)$ is a finitely generated $HH^\bullet(A)$-module. 

Fg condition enables developing support variety theory via Hochschild cohomology \cite{Erdmann-Holloway-Snashall-Solberg-Taillefer04,Snashall-Solberg04}.

\begin{theorem}  \label{Thm-SE3AT-HCH-Happel-Fg}
	Let finite dimensional algebras $A$ and $B$ be singularly equivalent of 3-adjoint type, and $A/\rad A$ and $B/\rad B$ be separable over $k$. Then the following statements hold:
	
	{\rm(1)} $HH^i(A) \cong HH^i(B)$ for $i\gg 0$.
	
	{\rm(2)} $A$ satisfies Happel's property if and only if $B$ does.
	
	{\rm(3)} $A$ satisfies the Fg condition if and only if $B$ does.	
\end{theorem}

\begin{proof} 
	(1) Since $A$ and $B$ are singularly equivalent of 3-adjoint type, there is a complex $X_1\in\D(A\otimes B^\op)$ which induces an adjoint triple
	$$\begin{tikzcd}[column sep=50]
		\sg(B) \arrow[r,"X_1\otimes^L_B-"] \arrow[r,shift right=12, "X_3\otimes^L_B-"] & \sg(A) \arrow[l,shift left=6,"X_2\otimes^L_A-"'] 
	\end{tikzcd}$$
	where $X_1, X_2:=\RHom_A(X_1,A)$ and $X_3:=\RHom_B(X_2,B)$ are biperfect.
	
	We have an adjoint pair $X_1\otimes^L_B-:\D(B\otimes A^\op)\rightleftarrows\D(A\otimes A^\op):X_2\otimes^L_A-$. Its counit induces a triangle $X_1\otimes^L_BX_2\to A\to \Cone(\ev^A_{X_1})\to$ in $\D(A^e)$, then a triangle $\RHom_{A^e}(\Cone(\ev^A_{X_1}),A)\to \RHom_{A^e}(A,A)\to \RHom_{A^e}(X_1\otimes^L_BX_2,A)\to$ in $\D(k)$. By Theorem~\ref{Thm-Exist-SingEquRightAdjoit-FD}, $\Cone(\ev^A_{X_1})$ is $\otimes$-perfect in $\D(A^e)$. By the assumption that $A/\rad A$ and $B/\rad B$ are separable over $k$ and Proposition~\ref{Prop-HomSmooth-TensorPerfect-BimodComplex}, $\Cone(\ev^A_{X_1})$ is perfect in $\D(A^e)$. Thus $\RHom_{A^e}(\Cone(\ev^A_{X_1}),A)\in \D^b(k)$. Hence $H^i\RHom_{A^e}(A,A)\cong H^i\RHom_{A^e}(X_1\otimes^L_BX_2,A)$ for $i\gg 0$.
	
	We have an adjoint pair $X_2\otimes^L_A-:\D(A\otimes B^\op)\rightleftarrows\D(B\otimes B^\op):X_3\otimes^L_B-$. Its counit induces a triangle $X_2\otimes^L_AX_3\to B\to \Cone(\ev^B_{X_2})\to$ in $\D(B^e)$, then a triangle $\RHom_{B^e}(\Cone(\ev^B_{X_2}),B)\to \RHom_{B^e}(B,B)\to \RHom_{B^e}(X_2\otimes^L_AX_3,B)\to$ in $\D(k)$. By Theorem~\ref{Thm-Exist-SingEquRightAdjoit-FD}, $\Cone(\ev^B_{X_2})$ is $\otimes$-perfect in $\D(B^e)$. By the assumption that $A/\rad A$ and $B/\rad B$ are separable over $k$ and Proposition~\ref{Prop-HomSmooth-TensorPerfect-BimodComplex}, $\Cone(\ev^B_{X_2})$ is perfect in $\D(B^e)$. Thus $\RHom_{B^e}(\Cone(\ev^B_{X_2}),B)\in \D^b(k)$. Hence $H^i\RHom_{B^e}(B,B)\cong H^i\RHom_{B^e}(X_2\otimes^L_AX_3,B)$ for $i\gg 0$.
	
	We have isomorphisms 
	$$\begin{array}{ll}
		& \RHom_{A^e}(X_1\otimes^L_BX_2,A) \\
		\cong & \RHom_{B\otimes A^\op}(X_2,\RHom_A(X_1,A)) \\
		= & \RHom_{B\otimes A^\op}(X_2,X_2) \\
		\cong & \RHom_{B\otimes A^\op}(X_2,\RHom_{B^\op}(\RHom_B(X_2,B),B)) \\
		= & \RHom_{B\otimes A^\op}(X_2,\RHom_{B^\op}(X_3,B)) \\
		\cong & \RHom_{B^e}(X_2\otimes^L_AX_3,B).
	\end{array}$$ 
	Thus $HH^i(A) \cong HH^i(B)$ for $i\gg 0$.
	
	(2) Due to $\sg(A)\simeq \sg(B)$, we know $\gl A<\infty$ if and only if $\gl B<\infty$. It follows from (1) that $A$ satisfies Happel's property if and only if $B$ does.
	
	(3) By Theorem~\ref{Thm-SE3AT-G-GSC-ARC}, $A$ is Gorenstein if and only if $B$ does. Assume that $A$ satisfies the Fg condition. Then $A$ is a Gorenstein by \cite[Theorem 1.5 (a)]{Erdmann-Holloway-Snashall-Solberg-Taillefer04}, and so does $B$. Since $A/\rad A$ and $B/\rad B$ are separable over $k$, by \cite[Theorem 1.1]{Qin22}, $A$ and $B$ are singularly equivalent of Morita type with level. Thus $B$ satisfies the Fg condition by \cite[Theorem 7.4]{Skartsaeterhagen16}. Similarly, if $B$ satisfies the Fg condition then so does $A$.	
\end{proof}

\begin{remark}{\rm
		A singular equivalence of 2-adjoint type does not preserve Fg condition in general by \cite[Example 7.5]{Skartsaeterhagen16}.
}\end{remark}

\section{EHI and singular equivalences }

In this section, we show that essentially surjective standard eventually homological isomorphisms between derived categories correspond to standard singular equivalences with canonical left adjoint.

\subsection{EHI and fully faithful triangle functors between singularity categories}

The following result establishes a relationship between standard eventually homological isomorphisms and fully faithful standard triangle functors between singularity categories with canonical left adjoint.

\begin{proposition} \label{Prop-EHI-FullFaithFunSingCat-LeftAdj}
	Let $A$ and $B$ be finite dimensional algebras, $X\in\D(A\otimes B^\op)$ biperfect and $\RHom_{B^\op}(X,\linebreak B)\in\D(B\otimes A^\op)$ right perfect. Then the triangle functor $X\otimes^L_B-: \D^b(B)\to\D^b(A)$ is an eventually homological isomorphism if and only if the triangle functor $X\otimes^L_B-:\sg(B)\to\sg(A)$ is fully faithful and admits a left adjoint $\RHom_{B^\op}(X,B)\otimes^L_A-:\sg(A)\to\sg(B)$. 
\end{proposition}

\begin{proof}
	Under the assumption that both $X\in\D(A\otimes B^\op)$ and $\RHom_{B^\op}(X,B)\in\D(B\otimes A^\op)$ are biperfect, the triangle functor $X\otimes^L_B-: \D^b(B)\to\D^b(A)$ is an eventually homological isomorphism if and only if the $B$-bimodule complex $\Cone(\ev^B_X:\RHom_{B^\op}(X,B)\otimes^L_AX\to B)$ is $\otimes$-perfect by Theorem~\ref{Thm-EvHomIso-TensorPerfect-BimodComplex}, if and only if the triangle functor $X\otimes^L_B-:\sg(B)\to\sg(A)$ is fully faithful and admits a left adjoint $\RHom_{B^\op}(X,B)\otimes^L_A-:\sg(A)\to\sg(B)$ by Lemma~\ref{Lem-SingCatFFTriFunLeftAdj-FD-Complex}.
\end{proof}

\subsection{Essentially surjective EHI and singular equivalences}

Let $e$ be an idempotent of a finite dimensional algebra $A$. Then by Proposition~\ref{Prop-Exist-EvHomIso-Idemp} and Corollary~\ref{Cor-Exist-SingEqu-Idemp}, the essentially surjective triangle functor $eA\otimes^L_A-:\D^b(A) \to \D^b(eAe)$ is an eventually homological isomorphism if and only if both $eA\otimes^L_A-:\sg(A)\to\sg(eAe)$ and $Ae\otimes^L_{eAe}-:\sg(eAe)\to\sg(A)$ are singular equivalences. 
More generally, we have the following result which clarifies the relations between essentially surjective standard eventually homological isomorphism and singular equivalences of 2-adjoint type.

\begin{theorem} \label{Thm-EssSurjEHI=SingEqWithLeftAdj}
	Let $A$ and $B$ be finite dimensional algebras, and a right perfect complex $X$ of $A$-$B$-bimodules induce an essentially surjective triangle functor $X\otimes^L_B-: \D^b(B)\to\D^b(A)$. Then the following statements are equivalent:
	
	{\rm(1)} The triangle functor $X\otimes^L_B-: \D^b(B)\to\D^b(A)$ is an eventually homological isomorphism.
	
	{\rm(2)} The triangle functor $X\otimes^L_B-:\sg(B)\to\sg(A)$ is a triangle equivalence with a left adjoint $\RHom_{B^\op}(X,B)\otimes^L_A-:\sg(A)\to\sg(B)$.
\end{theorem}

\begin{proof} 
	(1)$\Rightarrow$(2): 	
	Since the triangle functor $X\otimes^L_B-: \D^b(B)\to\D^b(A)$ is an essentially surjective eventually homological isomorphism, by the Claim 1 in the proof of Theorem~\ref{Thm-EHI-GorSymConj}, we have $_AX\in\per(A)$. So $X$ is biperfect. Since $_B(B^*)$ is injective, by the Claim 2 in the proof of Theorem~\ref{Thm-EHI-GorSymConj}, we have $X\otimes^L_BB^*\in K^b(\inj A)$. Then for all $Y\in\D^b(A)$, $(\RHom_{B^\op}(X,B)\otimes^L_AY)^* \cong \RHom_B(\RHom_{B^\op}(X,B)\otimes^L_AY,B^*) \cong \RHom_A(Y,X\otimes^L_BB^*) \in \D^b(k)$. Thus $\RHom_{B^\op}(X,B)\otimes^L_AY \in \D^b(k)$. Hence $\RHom_{B^\op}(X,B)_A\in\per(A^\op)$ by Proposition~\ref{Prop-BimodComplex-TensorPerfect}. So $\RHom_{B^\op}(X,B)$ is biperfect. It follows from Theorem~\ref{Thm-EvHomIso-TensorPerfect-BimodComplex} that the $B$-bimodule complex $\Cone(\ev^B_X)$ is $\otimes$-perfect. By Lemma~\ref{Lem-SingCatFFTriFunLeftAdj-FD-Complex}, the triangle functor $X\otimes^L_B-:\sg(B)\to\sg(A)$ is fully faithful and admits a left adjoint $\RHom_{B^\op}(X,B)\otimes^L_A-:\sg(A)\to\sg(B)$. Since the triangle functor $X\otimes^L_B-: \D^b(B)\to\D^b(A)$ is essentially surjective, so is the triangle functor $X\otimes^L_B-:\sg(B)\to\sg(A)$. Therefore, the triangle functor $X\otimes^L_B-:\sg(B)\to\sg(A)$ is a triangle equivalence with a left adjoint $\RHom_{B^\op}(X,B)\otimes^L_A-:\sg(A)\to\sg(B)$.	
	
	(2)$\Rightarrow$(1): 
	Since the triangle functor $X\otimes^L_B-:\sg(B)\to\sg(A)$ is a triangle equivalence with a left adjoint $\RHom_{B^\op}(X,B)\otimes^L_A-:\sg(A)\to\sg(B)$, by Theorem~\ref{Thm-Exist-SingEquLeftAdjoit-FD}, the bimodule complexes $X$ and $\RHom_{B^\op}(X,B)$ are biperfect, and the bimodule complexes $\Cone(\ev^B_X)\in\D(B^e)$ and $\Cone(l^A_X)\in\D(A^e)$ are $\otimes$-perfect. By Theorem~\ref{Thm-EvHomIso-TensorPerfect-BimodComplex}, the triangle functor $X\otimes^L_B-: \D^b(B)\to\D^b(A)$ is an eventually homological isomorphism.
\end{proof}

\begin{remark}{\rm 
	Thanks to Theorem~\ref{Thm-EssSurjEHI=SingEqWithLeftAdj}, essentially surjective standard eventually homological isomorphisms can reduce the finitistic dimension conjecture by Theorem~\ref{Thm-SE2AT-FDC} and transfer Han's property by Theorem~\ref{Thm-SE2AT-HH-Han} under the assumption of separability of factor algebras of the algebras modulo their Jacobson radicals.
}\end{remark}

\subsection{Examples}

The following example implies that for two derived equivalent algebras there is always an essentially surjective standard eventually homological isomorphism between their derived categories but generally no essentially surjective eventually homological isomorphism between their module categories. So Theorem~\ref{Thm-EssSurjEHI=SingEqWithLeftAdj} is not only stronger than but also a strict generalization of the singular equivalence part of \cite[Theorem I]{Qin-Shen24}.

\begin{example} \label{Ex-EssSurEHI-DerCat-NotModCat} {\rm 
	Let $A$ be the path algebra $kQ$ of quiver $Q:=(1\xleftarrow{a} 2\xleftarrow{b} 3)$ and $B$ be the bound quiver algebra $kQ/I$ with admissible ideal $I:=(ab)$. Then $A$ and $B$ are tilting equivalent (Ref. \cite[Chapter VI, 3.11. Examples (a)]{Assem-Simson-Skowronski06}), thus derived equivalent. So there is a two-sided tilting complex $X\in\D^b(A\otimes B^\op)=\per(A\otimes B^\op)$ such that the derived tensor product functor $_AX\otimes^L_B-:\D^b(B)\to\D^b(A)$ is a derived equivalence. 
	Then $_AX\otimes^L_B-:\D^b(B)\to\D^b(A)$ is an essentially surjective eventually homological isomorphism.	
	Assume that there exists an essentially surjective eventually homological isomorphism $F:\mod B\to\mod A$. 
	Note that the Auslander-Reiten quivers of $A$ and $B$ are the following:	
	$$
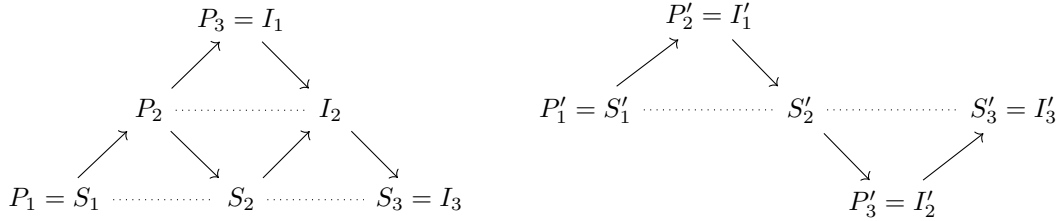
\begin{tikzcd} [column sep=4]
		&& P_3=I_1 \arrow[dr] &&\\
		& P_2 \arrow[ur]\arrow[dr] \arrow[rr,no head,dotted] && I_2 \arrow[dr] &\\
		P_1=S_1 \arrow[ur] \arrow[rr,no head,dotted] && S_2 \arrow[ur] \arrow[rr,no head,dotted] && S_3=I_3
	\end{tikzcd}\quad\quad
	\begin{tikzcd} [column sep=4]
		& P'_2=I'_1 \arrow[dr] &&& \\
		P'_1=S'_1 \arrow[ur] \arrow[rr,no head,dotted] && S'_2 \arrow[dr] \arrow[rr,no head,dotted] && S'_3=I'_3\\
		&&& P'_3=I'_2 \arrow[ur] &
	\end{tikzcd}$$
	
	\noindent where $S_i,P_i,I_i$ (resp. $S'_i,P'_i,I'_i$) are simple, indecomposable projective, and indecomposable injective modules over $A$ (resp. $B$) corresponding to the vertex $i$ for all $1\le i\le 3$. Since the additive functor $F$ preserves finite direct sums and is essentially surjective, $F$ induces a surjection from the set $\ind B =\{S'_1,S'_2,S'_3,P'_2,P'_3\}$ to the set $\ind A =\{S_1,S_2,S_3,P_2,I_2,P_3\}$. It is a contradiction.   
}\end{example}

By Theorem~\ref{Thm-EssSurjEHI=SingEqWithLeftAdj}, an essentially surjective standard eventually homological isomorphism can induce a singular equivalence with a canonical left adjoint. However, a standard singular equivalence with a canonical left adjoint need not be induced by an essentially surjective standard eventually homological isomorphism. Moreover, a standard eventually homological isomorphism cannot induce a singular equivalence in general. The following examples give a standard singular equivalence with a canonical left adjoint which is not induced by an essentially surjective eventually homological isomorphism, a standard eventually homological isomorphisms which cannot induce a singular equivalence, and a singular equivalence of 2-adjoint type but not a singular equivalence of 3-adjoint type.

\begin{example} \label{Ex-SE2AT-NotEssSurjEHI} {\rm	
		Let $A$ be the 4-dimensional bound quiver algebra $kQ/I$ where the quiver $Q$ is
		$$\xymatrix{
			\ar@(ul,dl)[]_x 1 \ar[r]^-y & 2			 
		}$$			
		\noindent and the admissible ideal $I:=(x^2,yx)$. 
		Then $e_1A=e_1Ae_1\cong k[x]/(x^2)$ as $k$-algebras, and $Ae_1A=Ae_1=e_1Ae_1\oplus e_2Ae_1$ where $e_2Ae_1\cong e_1Ae_1/\rad e_1Ae_1$ as a right $e_1Ae_1$-module. Thus $\pd_{e_1Ae_1}e_1A=0$ and $\pd_{(e_1Ae_1)^\op}Ae_1=\infty$. Clearly, $A/Ae_1A\cong k$ as $k$-algebra. 
		Moreover, $Ae_2=e_2Ae_2\cong k$ as $k$-algebras, $e_2A=Ae_2A$, and $A/Ae_2A \cong e_1Ae_1$ as $k$-algebras. Then $\pd_{(e_2Ae_2)^\op}Ae_2=0=\pd_{e_2Ae_2}e_2A$. Obviously, $\pd_A\frac{A/Ae_2A}{\rad(A/Ae_2A)} =\pd_AS_1 = \infty$ and $\id_A\frac{A/Ae_2A}{\rad(A/Ae_2A)} = \id_AS_1 = \infty$. Thus the triangle functor $e_2A\otimes^L_A-:\D^b(A)\to\D^b(e_2Ae_2)$ is not an eventually homological isomorphism by Proposition~\ref{Prop-Exist-EvHomIso-Idemp}, and the triangle functor $e_2A\otimes^L_A-:\sg(A)\to\sg(e_2Ae_2)$ is not a singular equivalence due to $\gl e_2Ae_2=0$ and $\gl A=\infty$, or by Corollary~\ref{Cor-Exist-SingEqu-Idemp}.
		
		Since $Ae_1\otimes^L_{e_1Ae_1}e_1A \cong Ae_1\otimes_{e_1Ae_1}e_1A \cong Ae_1 = Ae_1A$ in $\D(A^e)$, $Ae_1A$ is a stratifying ideal of $A$, and the factor map $A\twoheadrightarrow A/Ae_1A$ is a homological epimorphism of algebras. Then we have the following recollement:
		$$\begin{tikzcd}[column sep=80]
			\D(A/Ae_1A) \arrow[r,"{A/Ae_1A\otimes^L_{A/Ae_1A}-}"] & \D(A) \arrow[r,"{e_1A\otimes^L_A-}"] \arrow[l,shift left=6,"{\RHom_A(A/Ae_1A,-)}"'] \arrow[l,shift right=6,"A/Ae_1A\otimes^L_A-"'] & \D(e_1Ae_1). \arrow[l,shift left=6,"{\RHom_{e_1Ae_1}(e_1A,-)}"'] \arrow[l,shift right=6,"{Ae_1\otimes^L_{e_1Ae_1}-}"']
		\end{tikzcd}$$
		Since $Ae_2\otimes^L_{e_2Ae_2}e_2A \cong Ae_2\otimes_{e_2Ae_2}e_2A \cong e_2A = Ae_2A$ in $\D(A^e)$, $Ae_2A$ is a stratifying ideal of $A$, and the factor map $A\twoheadrightarrow A/Ae_2A$ is a homological epimorphism of algebras. Then we have the following 2-recollement \cite{Qin-Han16}:
		$$\begin{tikzcd}[column sep=125]
			\D(A/Ae_1A) \arrow[r,"{A/Ae_1A\otimes^L_{A/Ae_1A}-\ \cong\ Ae_2\otimes^L_{e_2Ae_2}-}"] \arrow[r,shift right=12,"{\RHom_{e_2Ae_2}(e_2A,-)}"] & \D(A) \arrow[r,"{e_1A\otimes^L_A-\ \cong\ A/Ae_2A\otimes^L_A-}"] \arrow[r,shift right=12,"{\RHom_A(A/Ae_2A,-)}"] \arrow[l,shift left=6,"{e_2A\otimes^L_A-}"'] \arrow[l,shift right=6,"A/Ae_1A\otimes^L_A-"'] & \D(e_1Ae_1). \arrow[l,shift left=6,"{A/Ae_2A\otimes^L_{A/Ae_2A}-}"'] \arrow[l,shift right=6,"{Ae_1\otimes^L_{e_1Ae_1}-}"']
		\end{tikzcd}$$
		The triangle functor $_AA/Ae_1A\otimes^L_{A/Ae_1A}-\cong Ae_2\otimes^L_{e_2Ae_2}-:\D^b(A/Ae_1A)\to\D^b(A)$ is fully faithful, so it is an eventually homological isomorphism, but the triangle functor $_AA/Ae_1A\otimes^L_{A/Ae_1A}-\cong Ae_2\otimes^L_{e_2Ae_2}-:\sg(A/Ae_1A)\to\sg(A)$ is not a singular equivalence due to $\gl A/Ae_1A=0$ and $\gl A=\infty$, or by Corollary~\ref{Cor-Exist-SingEqu-Idemp}.
		
		Note that $\frac{A/Ae_1A}{\rad(A/Ae_1A)}\cong A/Ae_1A\cong S_2=P_2$, the simple projective left $A$-module corresponding to the vertex 2. Thus $\pd_A\frac{A/Ae_1A}{\rad(A/Ae_1A)}=\pd_AP_2=0$. Since $0\to S_2\to I_2\to I_1\to I_1\to\cdots$ is a minimal injective resolution of $S_2$, we have $\id_A\frac{A/Ae_1A}{\rad(A/Ae_1A)} =\id_AS_2=\infty$. 
		Hence the triangle functor $e_1A\otimes^L_A-:\sg(A)\to\sg(e_1Ae_1)$ is a singular equivalence by Corollary~\ref{Cor-Exist-SingEqu-Idemp}, but the triangle functor $e_1A\otimes^L_A-:\D^b(A)\to\D^b(e_1Ae_1)$ is not an eventually homological isomorphism by Proposition~\ref{Prop-Exist-EvHomIso-Idemp}. 
		
		We have a short exact sequence $0\to Ae_2\to Ae_1\to A/Ae_2A\to 0$ of left $A$-modules, so $\pd_AA/Ae_2A=1$. Then the $A$-$(A/Ae_2A)$-bimodule $A/Ae_2A$ is biperfect. By Lemma~\ref{Lem-SingCatTriFunLeftAdj-FD-Complex}, the triangle functor\linebreak $A/Ae_2A\otimes^L_{A/Ae_2A}-:\sg(A)\to \sg(e_1Ae_1)$ is a singular equivalence since its left adjoint $A/Ae_2A\otimes^L_A-:\sg(e_1Ae_1)  \to\sg(A)$ is. Moreover, the triangle functor $A/Ae_2A\otimes^L_{A/Ae_2A}-:\D^b(e_1Ae_1)\to \D^b(A)$ is an eventually homological isomorphism since it is fully faithful, but not essentially surjective.
		
		Clearly, $\RHom_{e_2Ae_2}(e_2A,e_2Ae_2)\cong (e_2A)^*\cong I_2$, the indecomposable injective left $A$-module corresponding to the vertex $2$, which admits a minimal projective resolution $\cdots \to P_1\oplus P_2\to\cdots\to P_1\oplus P_2\to P_1\to I_2\to 0$. Then $\pd_AI_2=\infty$, i.e., $\RHom_{e_2Ae_2}(e_2A,e_2Ae_2)\in\D(A)$ is not perfect. Thus $\RHom_A(A/Ae_2A,A)\in\D(e_1Ae_1)$ is not perfect by \cite[Lemma 4.3 (a)]{Angeleri Hugel-Konig-Liu-Yang17}. Hence the triangle functor $e_1A\otimes^L_A-:\sg(A)\to \sg(e_1Ae_1)$ is a singular equivalence of 2-adjoint type but not 3-adjoint type by Theorem~\ref{Thm-Exist-SingEquAdjTye-FD}.
}\end{example}

The following example implies that an essentially surjective standard eventually homological isomorphism can induce a singular equivalence with a canonical left adjoint by Theorem~\ref{Thm-EssSurjEHI=SingEqWithLeftAdj} but not a canonical right adjoint in general.

\begin{example}{\rm 
	Let $A$ be the bound quiver algebra $kQ/I$ where the quiver $Q$ is
	$$\xymatrix{
		\ar@(ul,dl)[]_x 1 \ar[r]^-y & 2 \ar[r]^-z & 3 			 
	}$$		
	\noindent and the admissible ideal $I:=(x^2,yx)$, $B:=(e_1+e_2)A(e_1+e_2), C:=e_3Ae_3\cong k$, and the $C$-$B$-bimodule $M:=e_3A(e_1+e_2)$. Then $A=\begin{pmatrix}
		B&0\\ M&C
	\end{pmatrix}$ and $M_B\cong e_2B$ is projective. Since $M$ is biprojective, we have the following 4-recollement
	$$\begin{tikzcd} [column sep=140]
		\D(C) \arrow[r,"i_*:=A/A(e_1+e_2)A\otimes^L_C-\ \cong\ Ae_3\otimes^L_C-"] \arrow[r,shift right=12,"{i_!:=\RHom_C(e_3A,-)}"] \arrow[r,shift left=12] & \D(A) \arrow[r,"j^!:=(e_1+e_2)A\otimes^L_A-\ \cong\ A/Ae_3A\otimes^L_A-"] \arrow[r,shift right=12,"{j^*:=\RHom_A(A/Ae_3A,-)}"] \arrow[l,shift right=6,"{i^*:=A/A(e_1+e_2)A\otimes^L_A-}"'] \arrow[l,shift left=6,"i^!:=e_3A\otimes^L_A-"'] \arrow[l,shift left=18] \arrow[r,shift left=12] & \D(B). \arrow[l,shift right=6,"j_!:=A(e_1+e_2)\otimes^L_B-"'] \arrow[l,shift left=6,"{j_*:=A/Ae_3A\otimes^L_B-}"'] \arrow[l,shift left=18]
	\end{tikzcd}$$
	The algebra $C\times B$ is an arrow removal of $A$ for the arrow $z$, by \cite[Main Theorem]{Erdmann-Psaroudakis-Solberg22}, we have a singular equivalence $F=(i^!,j^!):\sg(A)\to\sg(C\times B)$ which is induced by the essentially surjective eventually homological isomorphism $F=(i^!,j^!):\D^b(A)\to\D^b(C)\times \D^b(B)=\D^b(C\times B)$. We have the following adjoint triple
	$$\begin{tikzcd}[column sep=60]
	  \D^b(A) \arrow[r,"{F:=(i^!,j^!)}"] & \D^b(C)\times \D^b(B)=\D^b(C\times B) \arrow[l,shift left=6,"{R:=(i_!,j_*)}"'] \arrow[l,shift right=6,"{L:=(i_*,j_!)}"']
	\end{tikzcd}$$
	where $L$ (resp. $R$) is the canonical left (resp. right) adjoint of $F$. Since $i_!(C) = \RHom_{e_3Ae_3}(e_3A,e_3Ae_3)=\Hom_k(e_3A,k)=I_3$ and $\pd_AI_3=\infty$, $\RHom_{e_3Ae_3}(e_3A,e_3Ae_3)\in\D(A\otimes C^\op)$ is not biperfect. Thus $R$ cannot induce a right adjoint of $F$ on singularity categories. 
}\end{example}

The following example gives a singular equivalence of 1-adjoint type but not of 2-adjoint type.

\begin{example}{\rm 
		Let $A$ be the bound quiver algebra $kQ/I$ where the quiver $Q$ is
		$$\xymatrix{
			1 \ar@(ul,dl)[]_x\ar[r]^-y & 2  \ar@(ur,dr)[]^z			 
		}$$		
		\noindent and the admissible ideal $I:=(x^2,yx,z^2)$, $B:=e_1Ae_1\cong k[x]/(x^2)$, $C:=e_2Ae_2\cong k[z]/(z^2)$, and the $C$-$B$-bimodule $M:=e_2Ae_1=\langle y,zy \rangle$. Then $A=\begin{pmatrix} B&0\\ M&C \end{pmatrix}$, $M_B$ is semisimple, and $_CM$ is projective. Since $_C(e_2A)$ is projective, we have the following 3-recollement
		$$\begin{tikzcd} [column sep=130]
			\D(C) \arrow[r,"i_*:=A/Ae_1A\otimes^L_C-\ \cong\ Ae_2\otimes^L_C-"] \arrow[r,shift right=12,"{i_!:=\RHom_C(e_2A,-)}"] & \D(A) \arrow[r,"j^!:=e_1A\otimes^L_A-\ \cong\ A/Ae_2A\otimes^L_A-"] \arrow[r,shift right=12,"{j^*:=\RHom_A(A/Ae_2A,-)}"] \arrow[l,shift right=6,"{i^*:=A/Ae_1A\otimes^L_A-}"'] \arrow[l,shift left=6,"i^!:=e_2A\otimes^L_A-"'] \arrow[l,shift left=18] & \D(B). \arrow[l,shift right=6,"j_!:=Ae_1\otimes^L_B-"'] \arrow[l,shift left=6,"{j_*:=A/Ae_2A\otimes^L_B-}"'] \arrow[l,shift left=18]
		\end{tikzcd}$$
		By \cite[Theorem 3.3]{Qin24} and its proof, we have a singular equivalence $(i_*,j_*):\sg(C\times B)=\sg(C)\times\sg(B)\to \sg(A)$ which has neither canonical left adjoint and canonical right adjoint since the canonical left adjoint $(i^*,j^!):\D(A)\to\D(C)\times\D(B)$ of  $(i_*,j_*):\D(C)\times\D(B)\to \D(A)$ cannot restrict to bounded derived categories and the canonical right adjoint $(i^!,j^*):\D(A)\to\D(C)\times\D(B)$ of $(i_*,j_*):\D(C)\times\D(B)\to \D(A)$ cannot restrict to perfect derived categories. Otherwise, the recollement above can be extended to a 4-recollement, then $A$ is Gorenstein by \cite[Theorem 3]{Qin-Han16}. This contradicts $\pd_{B^\op}M=\infty$ by \cite[Corollary 3]{Qin-Han16}.
}\end{example}

\end{document}